\documentclass[11pt, a4paper]{article}

\usepackage{amsmath}
\usepackage{amssymb}
\usepackage{physics} 

\usepackage{amsthm}
\usepackage{thmtools} 

\usepackage[a4paper, top = 3cm, bottom = 3cm, left = 2cm, right = 2cm]{geometry}

\usepackage[hyperfootnotes=false]{hyperref}

\usepackage[pdftex]{graphicx}
\usepackage{subcaption}
\usepackage[percent]{overpic}
\usepackage{xcolor}
\usepackage[normalem]{ulem} 
\usepackage{mathtools}

\usepackage{libertine}
\usepackage[libertine]{newtxmath}

\usepackage[backend=biber]{biblatex}
\declaretheorem[name=Lemma, numberwithin=section]{lemma}
\declaretheorem[name=Theorem, numberlike=lemma]{theorem}
\declaretheorem[name=Proposition, numberlike=lemma]{prop}
\declaretheorem[name=Remark, numberlike=lemma, style=remark]{remark}

\newcommand*{\calO}{\mathcal{O}}
\newcommand*{\bbR}{\mathbb{R}}
\newcommand*{\bbN}{\mathbb{N}}

\newcommand*{\calC}{\mathcal{C}}

\newcommand*{\calI}{\mathcal{I}}
\newcommand*{\calB}{\mathcal{B}}

\newcommand*{\SigmaOutTwo}{{\Sigma_{2, k_0}^{\mathrm{out}}}}
\newcommand*{\SigmaInTwo}{{\Sigma_{2, k_0}^{\mathrm{in}}}}
\newcommand*{\SigmaOutOne}{{\Sigma_{1, k_0}^{\mathrm{out}}}}
\newcommand*{\SigmaInOne}{{\Sigma_{1, k_0}^{\mathrm{in}}}}
\newcommand*{\SigmaOutThree}{{\Sigma_{3, k_0}^{\mathrm{out}}}}
\newcommand*{\SigmaInThree}{{\Sigma_{3, k_0}^{\mathrm{in}}}}

\newcommand*{\rin}{r^{\mathrm{in}}}
\newcommand*{\uonein}{{u}^{\mathrm{in}}_1}

\newcommand*{\vonein}{{v}^{\mathrm{in}}_1}

\newcommand*{\ukin}{{u}^{\mathrm{in}}_k}
\newcommand*{\vkin}{{v}^{\mathrm{in}}_k}

\newcommand*{\uoneout}{{u}^{\mathrm{out}}_1}

\newcommand*{\voneout}{{v}^{\mathrm{out}}_1}

\newcommand*{\ukout}{{u}^{\mathrm{out}}_k}
\newcommand*{\vkout}{{v}^{\mathrm{out}}_k}
\newcommand*{\ukoutone}{{u}^{\mathrm{out}}_{k, 1}}
\newcommand*{\uoneoutone}{{u}^{\mathrm{out}}_{1, 1}}
\newcommand*{\vkoutone}{{v}^{\mathrm{out}}_{k, 1}}

\newcommand*{\rout}{{r}^{\mathrm{out}}}

\newcommand*{\uout}{{u}^{\mathrm{out}}}
\newcommand*{\uin}{{u}^{\mathrm{in}}}
\newcommand*{\vout}{{v}^{\mathrm{out}}}
\newcommand*{\vin}{{v}^{\mathrm{in}}}
\newcommand*{\epsout}{{\varepsilon}^{\mathrm{out}}}
\newcommand*{\epsin}{{\varepsilon}^{\mathrm{in}}}

\newcommand*{\DeltaIn}{{\Delta^{\mathrm{in}}}}
\newcommand*{\DeltaOut}{{\Delta^{\mathrm{out}}}}

\newcommand*{\Cin}{C^{\mathrm{in}}}
\newcommand*{\Cout}{C^{\mathrm{out}}}

\newcommand*{\Rin}{R^{\mathrm{in}}}

\newcommand*{\bzero}{\mathbf{0}}

\def\R{\mathbb{R}}

\def\cC{\mathcal{C}}

\def\cO{\mathcal{O}}

\def\txtd{{\textnormal{d}}}

\def\txtD{{\textnormal{D}}}

\def\ra{\rightarrow}

\def\eps{\varepsilon}

\newcommand{\be}{\begin{equation}}
\newcommand{\ee}{\end{equation}}
\newcommand{\benn}{\begin{equation*}}
\newcommand{\eenn}{\end{equation*}}
\newcommand{\bea}{\begin{eqnarray}}
\newcommand{\eea}{\end{eqnarray}}
\newcommand{\beann}{\begin{eqnarray*}}
\newcommand{\eeann}{\end{eqnarray*}}

\newcommand{\myendex}{$\blacklozenge$\end{ex}}
\newcommand{\myendexerc}{$\lozenge$\end{exerc}}
\newcommand{\myendpexerc}{$\lozenge$\end{pexerc}}

\title{Geometric analysis of fast-slow PDEs with fold singularities via Galerkin discretisation}

\author{Maximilian Engel$^{1, 7}$ \hspace{1em} Felix Hummel$^2$ \hspace{1em} 
Christian Kuehn$^{2,3,4}$ \\ 
Nikola Popovi\'c$^5$ \hspace{1em} Mariya Ptashnyk$^6$ \hspace{1em} 
Thomas Zacharis$^{5,}$\footnote{Corresponding author, \href{mailto:tzachar2@ed.ac.uk}{tzachar2@ed.ac.uk}.}}

\date{October 5, 2024} 

\begin{document}

\maketitle

\begin{abstract}
	We study a singularly perturbed fast-slow system of two partial differential equations of
	reaction-diffusion type on a bounded domain via Galerkin discretisation. We assume that the
	reaction kinetics in the fast variable realise a generic fold singularity, whereas the slow
	variable takes the role of a dynamic bifurcation parameter, thus extending the classical
	analysis of the singularly perturbed fold.  Our approach combines a spectral Galerkin
	discretisation with techniques from geometric singular perturbation theory which are
	applied to the resulting high-dimensional systems of ordinary differential equations. In
	particular, we show the existence of invariant slow manifolds in the phase space of the original
	system of PDEs away from the fold singularity, while the passage past the singularity of the
	Galerkin manifolds obtained after discretisation is described by geometric desingularisation, or
	blow-up. Finally, we discuss the relation between these Galerkin manifolds and the underlying
	slow manifolds.
\end{abstract}

\noindent{\it Keywords}: geometric singular perturbation theory, fast-slow systems, fold
singularities, reaction-diffusion equations, Galerkin discretisation

\medskip

\noindent{\it Mathematics Subject Classification numbers}: 35B25, 35K57, 34Cxx, 34D15, 34E15, 37G10 

\footnotetext[1]{Department of Mathematics and Computer Science, Freie Universität Berlin, Arnimallee
14, 14195 Berlin, Germany}
\footnotetext[2]{Department of Mathematics, Technical University of Munich, Boltzmannstraße 3,
Garching bei München 85748, Germany}
\footnotetext[3]{Munich Data Science Institute, Walther-von-Dyck Straße 10, 85748 Garching bei
München, Germany}
\footnotetext[4]{Complexity Science Hub Vienna, Josefstädter Straße 39, 1080 Vienna, Austria}
\footnotetext[5]{School of Mathematics and Maxwell Institute for Mathematical Sciences,
University of Edinburgh, James Clerk Maxwell Building, King’s Buildings, Peter
Guthrie Tait Road, Edinburgh EH9 3FD, United Kingdom}
\footnotetext[6]{School of Mathematical \& Computer Sciences and Maxwell Institute for
Mathematical Sciences, Heriot-Watt University, Edinburgh EH14 4AP, United
Kingdom}
\footnotetext[7]{Korteweg–de Vries Institute for Mathematics, University of Amsterdam, Science
Park 105-107, 1098 XG Amsterdam, The Netherlands}

%
%
%
%

\section{Introduction}

Systems with multiple time scales have been established as a key mathematical tool across a broad
number of applications~\cite{Desrochesetal,KuehnBook,Wechselberger4}. At the centre of the theory of
multiple-scale dynamics are so-called fast-slow systems, which are given in standard form by
\begin{subequations}
\label{eq:fsstandard}
\begin{align}
\varepsilon\frac{\txtd u}{\txtd \tau} &=\dot{u} = f(u,v,\varepsilon),\\
\frac{\txtd v}{\txtd \tau} &=\dot{v} = g(u,v,\varepsilon),
\end{align}
\end{subequations}
where $u=u(\tau)\in\R^m$ are the fast variables, $v=v(\tau)\in\R^n$ are the slow variables,
$\varepsilon>0$ is a small parameter, $\tau$ is the slow time, and $f$ and $g$ are sufficiently
smooth functions of $u$, $v$, and $\varepsilon$. A wide variety of techniques have been developed
for analysing ordinary differential equations (ODEs) of the form in~\eqref{eq:fsstandard}, such as
asymptotic analysis~\cite{BenderOrszag,MisRoz,MKKR,Verhulst}, invariant manifold
theory~\cite{Fenichel4,JonesKopell,Tikhonov}, nonstandard
analysis~\cite{Benoit2,BenoitCallotDienerDiener},  geometric
desingularisation~\cite{DumortierRoussarie,KruSzm3}, and numerical
methods~\cite{DesrochesKrauskopfOsinga2,GuckenheimerKuehn2}. Appealingly, several of these
techniques allow for a highly visual description of the geometry of trajectories, attractors,
invariant sets, and sometimes even the entire phase space via a decomposition of the dynamics into
its fast and slow components. This highly intuitive viewpoint is emphasised by reference to the
corresponding techniques as geometric singular perturbation theory (GSPT). Indeed, in the singular
limit as $\varepsilon\ra 0$, we immediately identify the critical set 
\be
\label{eq:critman}
\cC_0:=\{(u,v)\in\R^{m+n}:f(u,v,0)=0\},   
\ee
which is commonly referred to as the critical manifold for \eqref{eq:fsstandard}. The slow (or
reduced) subsystem on that manifold is given by
\begin{subequations}
\label{eq:slowsub}
\begin{align}
0 &= f(u,v,0),\\
\dot{v} &= g(u,v,0).
\end{align}
\end{subequations}
The differential-algebraic Equation~\eqref{eq:slowsub} has the geometric interpretation of a
(generically) lower-dimensional dynamical system for the slow variables $v$. If $p=(u,v)\in\cC_0$ is
a normally hyperbolic point, i.e., if $\cC_0$ is locally a sufficiently smooth manifold and the
Jacobian matrix $\txtD_u f(p,0)$ at $p\in\cC_0$ has no spectrum on the imaginary axis, then
Fenichel's Theorem~\cite{Fenichel4,Jones,KuehnBook} implies that the critical manifold $\cC_0$
perturbs near $p$ to a slow manifold $\cC_\varepsilon$. The manifold $\cC_\varepsilon$ is then
$\cO(\varepsilon)$-close, in the Hausdorff distance, to $\cC_0$ as $\varepsilon\ra 0$; moreover, the
dynamics on $\cC_\varepsilon$ is locally topologically conjugate to that on $\cC_0$. Effectively,
Fenichel's Theorem thus geometrically asserts that the normally hyperbolic regime can be viewed as a
regular perturbation of its singular limit. However, it is relatively easy to prove
that~\eqref{eq:fsstandard} also gives rise to singular perturbations, as non-hyperbolic points
generically occur for $m,n\geq 1$, which can intuitively be understood by introducing the fast time
scale $t:=\tau/\varepsilon$ in \eqref{eq:fsstandard} and by then taking again the singular limit of
$\varepsilon\ra 0$:
\begin{subequations}
	\label{eq:fastsub}
	\begin{align}
	\frac{\txtd u}{\txtd t}&= u' = f(u,v,0),\\
	\frac{\txtd v}{\txtd t}&= v' = 0.
	\end{align}
\end{subequations}
The fast subsystem, or layer problem, is a parametrised system of ODEs, and is hence not even
structurally similar to the slow subsystem, or reduced problem, Equation~\eqref{eq:slowsub}.
Important transitions between slow and fast dynamics occur at points where normal hyperbolicity is
lost, which can also be interpreted as bifurcation points of the fast subsystem,
Equation~\eqref{eq:fastsub}. The most important geometric technique for the analysis of such
singularities is geometric desingularisation via ``blow-up"~\cite{DumortierRoussarie,KruSzm3}; see
also~\cite{JardonKuehn} for a recent review. A geometric blow-up of a point -- or a more general
submanifold -- amounts to defining a vector field on a higher-dimensional manifold, such as a
sphere, with the aim of regaining some hyperbolicity. That approach has been highly successful
across a variety of classes of low-dimensional systems of ODEs, such as for the desingularisation of
classical fold bifurcations~\cite{DumortierRoussarie,KruSzm3,DeMaesschalckDumortierRoussarie}, more
degenerate folded singularities~\cite{KrupaWechselberger,Wechselberger2}, Hopf
bifurcations~\cite{HayesKaperSzmolyanWechselberger}, and transcritical or pitchfork
bifurcations~\cite{KruSzm4,KuehnSzmolyan}.

However, for infinite-dimensional multiple-scale dynamical systems, there are significant conceptual
and technical challenges to the generalisation of GSPT. Naively, one might anticipate that an
extension of~\eqref{eq:fsstandard} to the partial differential equation (PDE)
\begin{subequations}%
	\label{eq:fsstandardPDE}
	\begin{align}
		u_t &= u_{x x} + f(u,v,\varepsilon),\\
		v_t &= \varepsilon \left( v_{x x} + g(u,v,\varepsilon) \right)
	\end{align}
\end{subequations}
with suitable
boundary conditions, where $u=u(t,x)\in\R^m$, $v=v(t,x)\in\R^n$, and $x\in \Omega$, with $\Omega$
being a bounded interval, may yield a sufficiently basic reaction-diffusion system to which
techniques from standard GSPT can be adapted. 
Naturally, on unbounded domains, an approach via spatial dynamics~\cite{Kirchgaessner,Sandstede1}
will allow one to apply finite-dimensional techniques directly. However, no ODE-based geometric
approach is available for the study of bounded and $\varepsilon$-independent domains. 

While some techniques from the theory of ODEs do translate well to an infinite-dimensional setting
\cite{Henry,KuehnBook1} on such domains, the PDE in \eqref{eq:fsstandard} presents
challenges~\cite{VoBertramKaper}.  When there is only a bounded perturbation in the slow variables,
i.e., when the term $\varepsilon v_{xx}$ is absent, the persistence of invariant manifolds in the
normally hyperbolic regime was resolved in~\cite{BatesLuZheng1,BatesLuZheng}. In that case, the
perturbation is, in essence, finite-dimensional, such that more classical invariant manifold
techniques apply~\cite{Avitabileetal1,HaragusIooss,Sieber}.  When the slow variables involve
unbounded operators, however, the situation is far more complicated, as the $\varepsilon
v_{xx}$-term results in non-trivial interactions between fast and slow modes in the limit of
$\varepsilon \ra 0$. Therefore, there is a crucial need for developing an infinite-dimensional
analogue of GSPT, which is the key motivation for this work. The normally hyperbolic regime in
\eqref{eq:fsstandardPDE} was resolved only recently in~\cite{HummelKuehn}, where an invariant
manifold theory was developed for~\eqref{eq:fsstandardPDE} on the basis of functional-analytic
techniques. 

An alternative approach via spectral Galerkin discretisation was proposed in~\cite{EngelKuehn1},
while a comparison of the two approaches can be found in~\cite{EngelHummelKuehn}. Since Galerkin
discretisation yields, upon truncation at a finite number of modes, large systems of singularly
perturbed ODEs of fast-slow type, one may hope that even a loss of normal hyperbolicity at singular
points can be treated by geometric desingularisation, or blow-up \cite{DumortierRoussarie,KruSzm3},
which was the focus in~\cite{EngelKuehn1}. There, the blow-up technique was applied to a Galerkin
truncation
resulting from a transcritical singularity, i.e., for $f(u,v,\varepsilon)=u^2-v^2+\mu \varepsilon$
in \eqref{eq:fsstandardPDE}, with $\mu$ a real parameter. As is well-known in the finite-dimensional
context, transcritical (and pitchfork) singularities are slightly more straightforward to
desingularise than fold singularities; cf.~the analysis in~\cite{KruSzm4} and~\cite{KruSzm3},
respectively.
Hence, in this work we consider a generic fold singularity in \eqref{eq:fsstandardPDE} as a logical
next step; specifically, we study the system
\begin{subequations}%
	\label{eq:initial_problem}
	\begin{align}
		u_t &= u_{xx} - v + u^2 + H^u(u, v, \varepsilon) && \text{for } x \in (-a, a)\text{ and } t>0, \\
		v_t &= \varepsilon (v_{xx} - 1 + H^v(u, v, \varepsilon)) && \text{for }  x \in (-a, a)\text{ and } t>0,\\
		& u_x(t,x) = 0 = v_x(t,x)  && \text{for } x= \pm a\text{ and } t>0, \\
	& u(0,x) = u_0(x)\ \text{ and }\ v(0,x) =v_0(x)  && \text{for } x \in (-a, a) \label{eq:initial_problem:b}
	\end{align}
\end{subequations}
on bounded domains, where the domain length $a > 0$ is fixed, with zero Neumann boundary conditions.
Here, $H^u$ and $H^v$ are higher-order terms which are specified below.


\begin{remark}
	Note that locally well-defined (smooth) solutions for~\eqref{eq:initial_problem} can be obtained from
	classical theory on sectorial operators with reaction kinetics \cite{Henry} and parabolic
	regularity \cite{Evans}.
\end{remark}

This work is divided into two parts. In the first part, we apply results from~\cite{HummelKuehn} to
obtain slow manifolds which drive the (semi)flow of \eqref{eq:initial_problem} within a
neighbourhood of the origin away from the fold in an appropriately chosen phase
space for suitable initial data. Using results from~\cite{EngelHummelKuehn}, we then approximate
these manifolds by slow manifolds in a truncated -- and thus \emph{finite-dimensional} -- Galerkin
discretisation of~\eqref{eq:initial_problem}. To avoid confusion, we henceforth refer to these
finite-dimensional manifolds as \emph{Galerkin manifolds}. For any fixed
$\varepsilon>0$, the resulting approximation can be made arbitrarily precise provided an
appropriate truncation level, denoted by $k_0>0$, is chosen; furthermore, we show that solutions of
the Galerkin discretisation converge to those of Equation~\eqref{eq:initial_problem} under suitable
assumptions for $k_0\rightarrow \infty$, which allows us to interpret the
corresponding Galerkin manifolds as ``approximately invariant slow manifolds" for
\eqref{eq:initial_problem}. As shown in~\cite{EngelHummelKuehn}, to study the
dynamics beyond trajectories and track families of Galerkin slow manifolds as $\varepsilon
\rightarrow 0$ and $k_0\rightarrow \infty$, one has to consider a coupling between the two
parameters. The resulting double singular limit~\cite{Kuehnetal} is not specific to PDE-type
settings, as it occurs also in the time discretisation of fast-slow
ODEs~\cite{ArcidiaconoEngelKuehn,EngelKuehn,JelbartKuehn}. Yet, as for the time-discretised case, we
find that there exist open parameter sets, for $\varepsilon$ and $k_0^{-1}$ close to zero, where our
results hold~\cite{EngelHummelKuehn}.

In the second part, which is the main result of this work, we apply the blow-up technique to extend
these Galerkin manifolds around the singularity at the origin in the truncated, $2k_0$-dimensional
Galerkin discretisation of~\eqref{eq:initial_problem}. Under appropriate assumptions on initial
conditions, e.g.~by restricting  to solutions of~\eqref{eq:initial_problem} that are close to
spatially homogeneous ones in an appropriately chosen norm, we show that the dynamics of the
Galerkin truncation in a neighbourhood of the origin can be reduced to that of the corresponding ODE
for the singularly perturbed planar fold, a well-known prototypical fast-slow system that was
studied via blow-up in~\cite{KruSzm3}.

There are evident similarities between our analysis and classical GSPT, where Fenichel's Theorem
\cite{Fenichel4} is combined with geometric desingularisation in the form of blow-up;
correspondingly, fast-slow systems of arbitrary dimension in both the fast and slow variables have
been studied geometrically in previous work \cite{brons2006mixed, wechselberger2012propos}.

However, the high dimensionality of our Galerkin discretisation, in combination with the inherent
spatial dependence of Equation~\eqref{eq:initial_problem}, poses both conceptual and technical
challenges. Firstly, a preparatory rescaling of the domain length with (a fractional power of)
$\varepsilon$ is essential to our approach, and is required to obtain both well-defined and
non-trivial dynamics in the singular limit after blow-up. A consequence of the rescaling is,
however, that the approach in \cite{wechselberger2012propos} does not apply, as the assumptions
therein are not satisfied. Secondly, careful consideration of initial data, in tandem with precise
estimates on the evolution of higher-order modes in the Galerkin discretisation, is required to
ensure that solutions do not exhibit finite-time blowup before reaching the singularity at the
origin; again, that blowup is inherently due to the Galerkin discretisation arising from a system of
PDEs.

Our approach has a number of further advantages: it achieves an effective reduction to the
singularly perturbed planar fold which can be studied via blow-up of the non-hyperbolic origin,
rather than of a submanifold of singularities; furthermore, it allows us to account for the impact
of data related to the original PDE, Equation~\eqref{eq:initial_problem}, such as the domain length
or the eigenvalues of the Laplacian therein, on the flow in its passage past the origin.
Correspondingly, our approach yields explicit asymptotics, rather than merely an existence
statement, and hence seems highly suited to the study of singular perturbation problems of fast-slow
type obtained by Galerkin discretisation.Such asymptotics will also be crucial for future work on
the double singular limit as $\varepsilon \rightarrow 0$ and $k_0\rightarrow \infty$, which will
build on~\cite{EngelHummelKuehn}.

Our main results can hence be summarised as follows; precise statements will be given below.
\begin{itemize}
	\item Equation~\eqref{eq:initial_problem} possesses a family of slow manifolds $S_{\varepsilon, \zeta}$
	for small $\varepsilon > 0$, where $\zeta > 0$ is an additional control parameter. These can be
		approximated by Fenichel-type slow manifolds
	$\calC_{\varepsilon} = \calC_{\varepsilon, k_0}$ in the corresponding Galerkin discretisation truncated at $k_0>0$, provided $k_0$ is sufficiently large.
	\item For any $k_0 > 0$ fixed, the Galerkin manifolds $\calC_{\varepsilon, k_0}$ are extended around the
		fold singularity at the origin in the Galerkin discretisation, which we show by combining the
		well-known fast-slow analysis of the singularly perturbed planar fold with {\it a priori} estimates that control higher-order modes.
\end{itemize}

In summary, our work is a stepping stone towards the development of a geometric approach for the
study of singularities in multiple-scale (systems of) PDEs. However, it still remains to
relate, rigorously and uniformly in $\varepsilon$ and $k_0$, the extension of the
Galerkin manifolds $\calC_{\varepsilon, k_0}$ after passage past the fold singularity to
corresponding manifolds for~\eqref{eq:initial_problem}when $\varepsilon
\rightarrow 0$ and $k_0\rightarrow \infty$. In the normally hyperbolic regime, we do know the
scaling relation between $\varepsilon$ and $k_0$ in the double singular limit
\cite{EngelHummelKuehn}; however, further work is required to understand that limit near
non-normally hyperbolic singularities. Here, we contribute to this ongoing research programme by
providing detailed estimates, at the level of the Galerkin discretisation of
\eqref{eq:initial_problem} near a generic fold singularity, in dependence of both $\varepsilon$ and
$k_0$.
%
%
%
%
\section{Galerkin discretisation}%
\label{sec:slow-and-galerkin}
The starting point for our analysis is the singularly perturbed system of PDEs in
\eqref{eq:initial_problem}. In analogy with the canonical form for the singularly perturbed planar
fold studied in \cite{KruSzm3}, we refer to $u$ and $v$ therein as the fast and slow variables,
respectively. The functions $H^u$ and $H^v$ are assumed to be smooth and of the form
\begin{subequations}%
	\begin{align}
		H^u(u, v, \varepsilon) &= \calO(\varepsilon, uv, v^2, u^3)\quad\text{and} \\
		H^v(u, v, \varepsilon) &= \calO(v^2),
	\end{align}
\end{subequations}
respectively. In addition, we assume that the higher-order terms $H^v$ in
\eqref{eq:initial_problem:b} are orthogonal in $L^2(-a, a)$ to the subspace of constant functions,
which is not an essential restriction that is only imposed for technical reasons, as will become
apparent in  estimates for solutions of the system of ODEs resulting from a Galerkin discretisation
of \eqref{eq:initial_problem}; see \autoref{lemma:K1-estimates}. In other words, we restrict $H^v$
so that $H^v(u, v, \varepsilon)$ has zero mean over $[-a,a]$ for any $u,v \in L^2(-a,a)$. One
specific example is given by $H^v(u, v, \varepsilon) = \tilde H^v(u, v, \varepsilon) - \frac{1}{2a}
\int_{-a}^{a} \tilde H^v(u, v, \varepsilon) dx$, where $\tilde H^v : \bbR^3 \rightarrow \bbR$ is
smooth. Note that we do not permit linear terms in $H^v$, since  $v_{xx}$ is a linear operator in
\eqref{eq:initial_problem:b}; however, we could consider more general $H^v$, such as $H^v(u, v,
\varepsilon) = \calO(u^2, uv, v^2, \varepsilon)$, with the caveat that we would have to impose
further restrictions on the initial values for the higher-order modes $u_k$ ($k \ge 2$), in analogy
to those imposed on $v_k$ for $k \ge 2$.

More compactly, we can write \eqref{eq:initial_problem} as 
\[
	w_t = A w + F(w), \quad\text{with } w(0)= w_0,
\]
where $w = (u,v)^T$, $w_0 = (u_0, v_0)^T$, $F(w) = (- v + u^2 + H^u(u, v, \varepsilon), - \varepsilon + \varepsilon H^v(u, v, \varepsilon))^T$, and 
\[
	Aw =\begin{pmatrix} 
	u_{xx} & 0 \\
	0 & \varepsilon v_{xx}
	\end{pmatrix}, \quad\text{with } 
	\mathcal D(A) = \{w \in  H^2(-a,a)^2 \; : \; u_x = 0 = v_x \; \text{ at }x=\pm a \}.  
\]
We have that $F(w)$ is locally Lipschitz continuous on $Z^\alpha = \mathcal D(A^\alpha)$ for
$1/4<\alpha <1$; moreover, the operator $A$ is sectorial and a generator of an analytic semigroup on
$Z=L^2(-a,a)^2$. Thus, for $w_0 \in Z^\alpha$, there exists a unique local-in-time solution $w \in
C([0, t_\ast); Z^\alpha) \cap C^1((0, t_\ast); Z)$, with $w(t) \in \mathcal D(A)$, to
\eqref{eq:initial_problem} for some $t_\ast >0$; see e.g.~\cite{Henry}. The quadratic nonlinearity
in \eqref{eq:initial_problem} implies a potential finite-time blowup of solutions to
\eqref{eq:initial_problem}; cf.~e.g.~\cite{Ball}. However, simple estimates show that, for initial
values $u_0 <0$ and $v_0>0$, a solution of \eqref{eq:initial_problem} exists for $t>0$ such that
$u(t)\leq 0$ and $v(t)\geq 0$; see \ref{app:boundconv} for details.

Before giving a precise statement of our results, we introduce the Galerkin discretisation of the
system of PDEs in \eqref{eq:initial_problem} with respect to the eigenbasis $\{ e_k(x) : k = 1, 2,
\dots \}$ of the Laplacian on $L^2(-a, a)$ with Neumann boundary conditions. Specifically, the
relevant orthonormal basis and the corresponding eigenvalues are given by
\begin{equation} \label{eq:eigen}
		e_{k + 1} (x) = \sqrt{\frac{1}{a}} \cos \left( \frac{k \pi (x+ a)}{2a} \right)\quad\text{and}\quad
		\lambda_{k + 1} = -\frac{k^2 \pi^2}{4 a^2}\quad\text{for }
	k = 1, 2, \dots,
\end{equation}
respectively, with $e_1(x) = \frac{1}{\sqrt{2a}}$ and $\lambda_1 = 0$. 
Next, we define 
\begin{equation}%
	\label{eq:def-bk}
	b_k := -{(k - 1)}^2 \pi^2,
\end{equation}
so that $\lambda_{k + 1} = \frac{b_{k + 1}}{4 a^2}$.

Then, solutions of~\eqref{eq:initial_problem} can be expanded as
\begin{equation}
	\label{eq:expansion}
    u(x, t) = \sum_{k = 1}^\infty e_k(x) u_k(t)\quad\text{and}\quad v(x, t) = \sum_{k = 1}^\infty
   	 e_k(x) v_k(t).
\end{equation}
Substitution of~\eqref{eq:expansion} into~\eqref{eq:initial_problem} results in the infinite system of ODEs
\begin{subequations}%
	\label{eq:infinite_ode_system}
	 \begin{align}
		 u_k^\prime &= \lambda_k u_k - \langle v, e_k \rangle + \langle u^2, e_k \rangle
			+ \langle H^u, e_k \rangle, \\
		 v_k^\prime &= \varepsilon \left( \lambda_k v_k - \langle 1, e_k \rangle +
			\langle H^v, e_k \rangle \right) 
	\end{align}
\end{subequations}
for $k = 1, 2, \dots$, where 
\[
    \langle \phi, \psi \rangle = \int_{-a}^a \phi(x) \psi(x) dx \quad\text{for }\phi, \psi \in L^2(-a, a).
\]
Using the formulae in~\eqref{eq:eigen}, we can then derive the following explicit form of~\eqref{eq:infinite_ode_system}:
\begin{prop}%
	\label{prop:discretisation}
	The system in~\eqref{eq:infinite_ode_system}, truncated at $k_0 \in \bbN$, reads
	\begin{subequations}%
		\label{eq:galerkin_system_original}
		\begin{align}
			 u_1^\prime &= - v_1 +\frac{1}{\sqrt{2a}}\sum_{j = 1}^{k_0} u_j^2 
				+ H^u_1, \\
			v_1^\prime &= -\sqrt{2a} \varepsilon, \\
			u_k^\prime &= \frac{1}{4} a^{-2} b_k u_k - v_k + \frac{2}{\sqrt{2a}} u_1 u_k + 
				\frac{1}{\sqrt{a}} \sum_{i, j = 2}^{k_0} \eta^k_{i, j} u_i u_j + H^u_k, \\
			v_k^\prime &= \varepsilon \frac{1}{4} a^{-2} b_k  v_k + \varepsilon H^v_k 
		\end{align}
	\end{subequations}
	for $2 \le k \le k_0$, where $0 \le \eta^k_{i, j} \le 1$ is non-zero if and only if
	$i + j - 2 = k - 1$ or $|i - j| = k - 1$, and
	\begin{align}
		H^u_1 &= \calO(\varepsilon, v_1^2, v_j^2, u_1 v_1, u_j v_j, u_1 u_j^2, u_i u_j u_l) 
			\quad\text{for } 2 \le i,j,l \le k_0, \\
		H^u_k &= \calO(v_1 v_k, v_i v_j, u_1 v_k, u_k v_1, u_i v_j, u_1^2 u_k, u_1 u_i u_j, 
		u_i u_j	u_l)\quad\text{for } 2 \le i,j,l \le k_0,\text{ and} \\
		H^v_k &= \calO (v_1 v_k, v_i v_j) \quad\text{for }  2 \le i,j \le k_0.
	\end{align}
\end{prop}
\begin{remark}
Our assumption that the higher-order terms $H^v$ are orthogonal to the subspace of constant
functions ensures that $H^v_1 = 0$ in \eqref{eq:galerkin_system_original}.
\end{remark}
\begin{proof}
	Because the basis ${\{ e_j \}}_{j \ge 1}$ is orthonormal in $L^2(-a, a)$, we have $\langle v, e_k \rangle = v_k$ for all
	$k \ge 1$.  We observe that $\langle 1, e_1 \rangle = \sqrt{2a}$ and 
	$ \langle 1, e_k \rangle = 0$ for any $k \ge 2$. Recalling that $e_1$ is a constant function, we find
	\[
		\langle e_1 e_j, e_k \rangle = e_1 \langle e_j, e_k \rangle = {(2a)}^{-1/2} \delta_{j, k}
	\]
	for all $j, k \ge 1$, where $\delta_{j, k}$ denotes the standard Kronecker delta. In addition, simple calculations show that 
	\[
		\langle e_i e_j, e_k \rangle := a^{-1/2} \eta^k_{i, j},
	\]
	where $\eta^k_{i, j}$ is independent of $a$ and given by
	\[
		\eta^k_{i, j} = \int_0^1 \cos ((i + j - 2) \pi x) \cos ( (k - 1) \pi x ) dx
			+ \int_0^1 \cos( (i - j) \pi x) \cos ( (k - 1) \pi x) dx.
	\]
	It follows that $0 \le \eta^k_{i, j} \le 1$ is non-zero if and only if $i + j  = k + 1$ or
	$|i - j| = k - 1$. In particular, $\langle e_k^2, e_k \rangle = 0$ for $2 \le k \le k_0$. Equipped
	with the relations above, we can now calculate the term $\langle u^2, e_k \rangle$
	in~\eqref{eq:infinite_ode_system}. For $k = 1$, we have
	\begin{align*}
		\Big\langle \Big( \sum_{j = 1}^{k_0} u_j e_j \Big)^2, e_1 \Big\rangle 
			&= \sum_{i, j = 1}^{k_0} u_i u_j \langle e_j e_i, e_1 \rangle = \sum_{i, j = 1}^{k_0} u_i u_j e_1 \langle e_j , e_i \rangle = {(2a)}^{-1/2} \sum_{j = 1}^{k_0} u_j^2,
	\end{align*}
	whereas for $2 \le k \le k_0$, it holds that
	\begin{align*}
		\Big\langle \Big( \sum_{j = 1}^{k_0} u_j e_j \Big)^2, e_k \Big\rangle
			&= \sum_{i, j = 1}^{k_0} u_i u_j \langle e_j e_i, e_k \rangle = 2 u_1 \sum_{j = 1}^{k_0} u_j \langle e_j e_1, e_k \rangle 
				+ \sum_{i, j = 2}^{k_0} u_i u_j \langle e_j e_i, e_k \rangle \\
			&= 2 {(2a)}^{-1/2} u_1 u_k + \sum_{i, j = 2}^{k_0} \eta^k_{i, j} u_i u_j,	
	\end{align*}
	as in the first sum only the term with $j = k$ is non-zero.
\end{proof}

The relation between solutions of the Galerkin discretisation in \eqref{eq:infinite_ode_system} and
those of Equation~\eqref{eq:initial_problem} is discussed briefly in
\ref{app:boundconv}.

\section{Slow and Galerkin manifolds}
\label{sec:manifolds}

In analogy to standard procedure for fast-slow ODEs of singular perturbation type, the first step in
our geometric analysis is to determine the critical manifold for \eqref{eq:initial_problem}.
Considering the slow formulation of~\eqref{eq:initial_problem}, obtained from the time rescaling
$\tau = \varepsilon t$, 
\begin{subequations}%
	\label{eq:initial-problem-slow}
	\begin{align}
		\varepsilon u_\tau &= u_{xx} - v + u^2 + H^u(u, v, \varepsilon) && \text{ for } x\in (-a,a)\text{ and }\tau>0, \\
		v_\tau &= v_{xx} - 1 + H^v(u,v, \varepsilon)  && \text{ for } x\in (-a,a)\text{ and }\tau>0, \\
		& u_x(\tau ,x) = 0 =  v_x(\tau,x)  && \text{ for } x=\pm a \text{ and }\tau>0,
	\end{align}
\end{subequations}
and setting $\varepsilon = 0$ therein, we find that the critical manifold is given by the set
\begin{equation}
	\left\{ (u,v) : 0 = u_{xx} - v + u^2 + H^u(u, v, 0),~u_x(\cdot,\pm a)=0=v_x(\cdot,\pm a) \right\}.
\end{equation}
Restricting to spatially homogeneous solutions, we define the critical manifold $S_0$ as the set of functions
\begin{equation}%
	\label{eq:critical-manifold-pde}
	S_0 := \left\{ (u,v) \in \bbR^2 : 0 = - v + u^2 + H^u(u, v, 0) \right\},
\end{equation}
abusing notation and identifying constant functions $u : [-a, a] \rightarrow \bbR$ with the values they take.
Due to our assumptions on the form of $H^u$, near the origin $(u, v) = (0,0)$ in $(u,v)$-space the set $S_0$ is given as a 
graph 
\begin{equation}
	S_0 = \left\{ (u, v) \in \bbR^2: v = u^2 + \calO( u^3) \right\}.
\end{equation}
Proceeding again as in a finite-dimensional setting, the second step in our analysis
concerns the persistence of the manifold $S_0$ for $\varepsilon$ positive and sufficiently small.
However, in an infinite-dimensional setting, the concept of ``fast'' and ``slow'' variables can be delicate, as for
any $\varepsilon > 0$, there exists $k>0$ such that $\varepsilon \lambda_k = O(1)$.
One way to address that complication is to split the slow variables $v$ into fast and slow parts,
which we make precise in the following proof of~\autoref{Prop_3_2}. We refer to~\cite{HummelKuehn}
for further discussion and details.
\begin{prop}\label{Prop_3_2}
	Let $(u, v) \in S_0$ with $u < 0$. Consider any small $\zeta>0$ and $u \le \omega_A < 0$, $\omega_f \in \bbR$, and $L_f > 0$ such that  $\omega_A + L_f
	<\omega_f < 0$. Then, there exist spaces $Y^\zeta_S \oplus Y^\zeta_F = L^2(-a, a)$ and a family
	of attracting slow manifolds around $(u, v)$ that are given as graphs
	\begin{equation}\label{slow_manig_1}
		S_{\varepsilon, \zeta} := \left\{ \left( h^{\varepsilon, \zeta}_X(v), 
			h^{\varepsilon,	\zeta}_{Y^\zeta_F}(v), v \right) : v \in Y_S^\zeta \right\}
	\end{equation}
	for $0 < \varepsilon < C \frac{\omega_f}{\omega_A} \zeta$ and some fixed  $C \in (0, 1)$,
	where $\left( h^{\varepsilon, \zeta}_X(v), h^{\varepsilon,	\zeta}_{Y^\zeta_F}(v)
	\right) :
	Y^\zeta_S \rightarrow H^2(-a,a) \cross (Y^\zeta_F \cap H^2(-a,a))$. 
\end{prop}
\begin{proof}
We show that the assumptions of~\cite[Theorem 2.4]{EngelHummelKuehn} are satisfied, which will imply
the existence of a family of slow manifolds stated in \eqref{slow_manig_1}. Given a point $(u, v) = (c, c^2+
\calO(c^3))$ on $S_0$, with $c < 0$ sufficiently small, we first translate that point to the origin
in~\eqref{eq:initial_problem}, which yields
\begin{subequations}%
	\label{eq:manif_translated_system}
	\begin{align}
		u_t &= u_{xx} - v + u^2 + 2c u + \widetilde H^u(u,v,\varepsilon)  && \text{for } x\in (-a,a) \text{ and } t>0, \\ 
		v_t &= \varepsilon \left( v_{xx} - 1 + H^v(u,v, \varepsilon) \right) && \text{for } x\in (-a,a) \text{ and } t>0, \\
		& u_x(t,x) = 0 = v_x(t,x) && \text{for } x= \pm a\text{ and } t>0.
	\end{align}
\end{subequations}
Here, $\widetilde H^u$ are new higher-order terms that are obtained from $H^u$ post-translation.
We choose 
	\begin{equation} 
		X = L^2(-a, a)\quad\text{and}\quad Y = L^2(-a, a) 
	\end{equation}
	as the basis spaces for $u$ and $v$, respectively, and  consider 
	$X_\alpha = H^{2\alpha}(-a,a)$ and $Y_\alpha = H^{2 \alpha}(-a,a)$ 
	for $\alpha \in [0, 1)$.
The linear operators $L_1$ and $L_2$ are defined as $L_1 u= u_{xx} + 2c u$ and
	$L_2= v_{xx}$, respectively, with $\mathcal D(L_1) = \mathcal D(L_2) = \{ \phi \in  H^2(-a,a) \; : \; \phi_x(-a) = 0 = \phi_x(a) \}$. 

Since we are interested in a neighbourhood of the origin in \eqref{eq:manif_translated_system} and by rescaling $v=\kappa_v \tilde v$, for any $\kappa_v>0$, we consider the modified nonlinear terms 
\begin{subequations}%
	\begin{align}
		f(u, v) &=  -\kappa_v v +  \chi(u) u^2  +  \chi(u) \chi(v) \widehat H^u \quad\text{and} \\
		g(u, v) &=   -\kappa_v^{-1} + \chi(u) \chi(v) \widehat H^v ,
	\end{align}
\end{subequations}		
where  $\chi : H^2(-a,a) \rightarrow [0, 1]$ is such that 
\[
	\chi (u) = 1 \; \text{ if } \Vert u \Vert_{H^2} \le \sigma^2, \quad 
		\chi (u) = 0 \; \text{ if } \; \Vert u \Vert_{H^2} \ge 2 \sigma, \quad\text{and}\quad 
		\Vert D \chi \Vert_{\mathcal L(H^2, \bbR)} \le \sigma
\]
for some $0 < \sigma < 1$ and $\widehat H^u$ and $\widehat H^v$ denote the higher-order terms with rescaled $\tilde v=v/\kappa_v$, where the tilde is omitted. Then, these modified nonlinearities 
		\begin{equation}
	f: H^2(-a,a) \cross L^2(-a,a) \rightarrow L^2(-a,a) \quad\text{and}\quad g: H^2(-a,a) \cross H^2(-a,a) \rightarrow H^2(-a,a)
		\end{equation}
	satisfy
		\begin{equation}%
			\label{eq:manif_nonlinearities}
			\begin{aligned}
				\Vert D f(u, v) \Vert_{\mathcal L(H^2 \cross L^2, L^2)} &\le L_{f_1}, \\
				\Vert D f(u, v) \Vert_{\mathcal L(H^2 \cross H^2, H^2)} &\le L_{f_2},\quad\text{and} \\
				\Vert D g(u, v) \Vert_{\mathcal L(H^2 \cross H^2, H^2)} &\le L_g,
			\end{aligned}
		\end{equation}
where $\mathcal L(V,W)$ is the space of linear operators from $V$ into $W$. Define $L_f := \min
\{ L_{f_1}, L_{f_2} \}$ and note that, by choosing $\sigma>0$ small, the constants $L_{f}$ and
$L_g$ can be made appropriately small. 

Note also that, for any $\varepsilon >0$, there exists $k>0$ such that $\varepsilon \lambda_k = \calO(1)$,
where $\lambda_k = - \frac{k^2 \pi ^2}{4 a^2}, k =0, 1, \dots$, are the eigenvalues of the operator $L_2$
with zero Neumann boundary conditions. Thus, to define fast and slow variables, we need
to split the basic space $Y=L^2(-a,a)$ for $v$ into $Y = Y^\zeta_S \oplus Y^\zeta_F$, where
\begin{subequations}%
	\label{eq:Y_split}
	\begin{align}
		Y^\zeta_S &:= \operatorname{span} \left\{ e_k(x) : 0 \le k \le k_0 \right\}\quad\text{and} \\
		Y^\zeta_F &:= \overline{\operatorname{span} \left\{e_k(x) : k > k_0 \right\}}^{L^2},
	\end{align}
\end{subequations}
with $\{e_k(x)\}_{k \in \mathbb N}$ being the eigenfunctions of the operator $L_2$ corresponding to
the eigenvalues $\{\lambda_k\}_{k \in \mathbb N}$ and $k_0 \in \bbN$ satisfying
\begin{equation}%
	\label{eq:zeta-and-k0}
	- \frac{{(k_0 + 1)}^2 \pi^2}{4a^2} < \zeta^{-1} \omega_A \le - \frac{k_0^2 \pi^2}{4 a^2}, 
\end{equation}
for  given $\zeta>0$ and $\omega_A \in (2c, 0)$.

Then, for the semigroups generated by $-B_S$ and $B_F$, which are the realisations of the operator $L_2$
in  $Y_S^\zeta \cap L^2(-a, a)$ and $Y_F^\zeta \cap L^2(-a, a)$, respectively, we have the following estimates:
\begin{subequations}\label{estim_semigroup}
	\begin{align}
		\Vert e^{-t B_S} y_S \Vert_{H^2} &\le e^{\frac{\pi^2 k_0^2}{4 a^2} t} \Vert y_S \Vert_{H^2} \quad  && \text{ for } y_S \in Y_S^\zeta, \\
		\Vert e^{t B_F} y_F \Vert_{H^2} &\le e^{-\frac{\pi^2 {(k_0 + 1)}^2}{4 a^2} t} \Vert y_F \Vert_{H^2} \quad && \text{ for } y_F \in Y_F^\zeta \cap H^2(-a,a),
	\end{align}
\end{subequations}
see e.g.~\cite[p.20]{Henry}.

Now, using \eqref{eq:Y_split} and the estimates in \eqref{estim_semigroup} and following the proof
of~\cite[Theorem 2.4]{EngelHummelKuehn} and \cite{HummelKuehn}, we obtain the stated results.
\end{proof} 
\begin{remark}
Here, we have written $Y_S^\zeta$ instead of $Y_S^\zeta \cap H^2(-a,a)$, as $Y_S^\zeta$ is a finite-dimensional subspace of $H^2(-a,a)$. 
\end{remark}
Next, for given $\zeta>0$, we also split the space $X=L^2(-a,a)$ into $X= X^\zeta_S \oplus
X^\zeta_F$, where $X^\zeta_S$ and $X^\zeta_F$ are defined in the same manner as $Y^\zeta_S$ and
$Y^\zeta_F$, see~\eqref{eq:Y_split}.

Then, the truncation of the Galerkin system in~\eqref{eq:infinite_ode_system} at $k_0$, which is related to $\zeta$ via~\eqref{eq:zeta-and-k0}, gives the projection
of~\eqref{eq:manif_translated_system} onto $\big( X^\zeta_S, Y^\zeta_S \big)$. Thus, we obtain a family of so-called Galerkin manifolds
\begin{equation}%
	G_{\varepsilon, \zeta} := \left\{ \left( h_G^{\varepsilon, \zeta} (v), v \right) : v \in
		Y^\zeta_S \right\}
\end{equation}
for a function $h^{\varepsilon, \zeta}_G : Y^\zeta_S \rightarrow X^\zeta_S$. 

\begin{prop}\label{Prop_3_3}
	There exists a constant $\tilde C>0$ such that, for $0 < \varepsilon < C
	\frac{\omega_f}{\omega_A} \zeta$, with $\zeta$, $\omega_A$, and $\omega_f$ as in
	Proposition~\ref{Prop_3_2} and some fixed $C \in (0, 1)$, the following estimate holds:
	\begin{equation}%
		\left\Vert h^{\varepsilon, \zeta}_X(v) - h_G^{\varepsilon, \zeta} (v) \right\Vert_{H^{2}} + \left\Vert h^{\varepsilon,
		\zeta}_{Y^\zeta_F}(v) \right\Vert_{H^{2}} \le \tilde C \left( \frac{4a^2}{\pi^2 (2 k_0 + 1)} + \zeta
		\right) \Vert v \Vert_{H^{2}}.
	\end{equation}
	In particular, using the relation between $\zeta$ and $k_0$ in~\eqref{eq:zeta-and-k0}, we have
	\begin{equation}%
		\left\Vert h^{\varepsilon, \zeta}_X(v) - h_G^{\varepsilon, \zeta} (v) \right\Vert_{H^{2}} + \left\Vert h^{\varepsilon,
		\zeta}_{Y^\zeta_F}(v) \right\Vert_{H^{2}} \le \tilde C \frac{1}{k_0} \Vert v \Vert_{H^{2}}.
	\end{equation}
\end{prop}
\begin{proof}
The proof follows the same steps as in~\cite{EngelHummelKuehn}.
\end{proof}

\begin{remark} 
	Note that $k_0 \to \infty$ corresponds to $\zeta \to 0$ which, due to the relation
	$0<\varepsilon < C \frac{\omega_f}{\omega_A} \zeta$, see Propositions~\ref{Prop_3_2}
	and~\ref{Prop_3_3}, also implies $\varepsilon \to 0$ when $k_0 \to \infty$. Hence, the limit of
	the Galerkin manifolds $G_{\varepsilon, \zeta}$ as $k_0 \to \infty$ cannot, in general, be
	guaranteed uniformly in $\varepsilon$. Thus, we perform the following analysis
	for $\varepsilon\in(0,\varepsilon_0)$, with $\varepsilon_0$ sufficiently small, and $k_0$
	arbitrarily large, but fixed. 
\end{remark}
%
%
%
%
%
\section{Fast-slow analysis}%
\label{sec:fast-slow-analysis}

Consider an arbitrary, fixed $k_0 \in \bbN$ in~\autoref{prop:discretisation}. A rescaling of the
variables in \eqref{eq:galerkin_system_original} via $u_k \mapsto a^{-1/2} u_k$ and $v_k \mapsto
a^{-1/2} v_k$ gives the fast-slow system
\begin{subequations}%
	\label{eq:galerkin_system_temp}
	\begin{align}
		 u_1^\prime &= - v_1 + 2^{-1/2} u_1^2 + 2^{-1/2} \sum_{j = 2}^{k_0} u_j^2 + H^u_1, \\
		 v_1^\prime &= -2^{1/2} \varepsilon, \\
		\label{eq:galerkin_system_temp_uk}
		u_k^\prime &= \frac{1}{4} a^{-2} b_k u_k - v_k + 2^{1/2} u_1 u_k + 
			\sum_{i, j = 2}^{k_0} \eta^k_{i, j} u_i u_j + H^u_k, \\
		\label{eq:galerkin_system_temp_vk}
		 v_k^\prime &= \frac{1}{4} a^{-2} b_k \varepsilon v_k + \varepsilon H^v_k,
	\end{align}
\end{subequations}
for $2\leq k\leq k_0$. The rescaled system in~\eqref{eq:galerkin_system_temp} is equivalent to the
original one in~\eqref{eq:galerkin_system_original}, in that orbits of the latter are mapped to
those of the former. Thus, without loss of generality, in our analysis, we will henceforth focus
on~\eqref{eq:galerkin_system_temp}.
In the slow time variable $\tau = \varepsilon t$, Equation~\eqref{eq:galerkin_system_temp}
becomes
\begin{subequations}%
	\label{eq:galerkin_system_temp_slow}
	\begin{align}
		 \varepsilon \dot u_1 &= - v_1 + 2^{-1/2} u_1^2 + 2^{-1/2} \sum_{j = 2}^{k_0} u_j^2 + H^u_1, \\
		 \dot v_1 &= -2^{1/2} , \\
		\varepsilon \dot u_k &= \frac{1}{4} a^{-2} b_k u_k - v_k + 2^{1/2} u_1 u_k + 
			\sum_{i, j = 2}^{k_0} \eta^k_{i, j} u_i u_j + H^u_k, \\
		 \dot v_k &= \frac{1}{4} a^{-2} b_k v_k + H^v_k,
	\end{align}
\end{subequations}
with the overdot denoting differentiation with respect to $\tau$.

Recalling the system of PDEs in \eqref{eq:initial_problem}, where the singularity is located at the origin of $L^2(-a, a)$, we will
be considering initial data in a neighbourhood thereof in the $L^2$-norm, with
\begin{equation}%
	\label{eq:initial-data-L2}
	\sum_{k = 1}^{\infty} |u_k(0)|^2  \le \kappa\quad\text{and}
		\quad \sum_{k = 1}^{\infty} |v_k(0)|^2 \le \kappa,
\end{equation}
%
%
where $0<\kappa <1$. In addition, we impose the bounds 
\begin{equation}%
	\label{eq:initial-data}
	|u_k(0)| \le C_{k, u_0} \quad\text{and}\quad |v_k(0)| \le C_{k, v_0}\varepsilon^{4/3} \quad \text{for } 
		k = 2, 3, \dots, k_0,
\end{equation}
where $C_{k, u_0}$ and $C_{k, v_0}$ are positive constants.
The initial conditions for the first mode $\{u_1, v_1\}$ are taken as in the finite-dimensional
(planar) case \cite{KruSzm3}, and are specified in Equation~\eqref{eq:delta-in-definition} below.

The assumption in \eqref{eq:initial-data} implies that the higher-order modes $u_k(0)$ and $v_k(0)$,
corresponding to non-constant eigenfunctions, are sufficiently small. The requirement that $v_k(0)$
is of the order ${\cal O}(\varepsilon^{4/3})$ is essential for ensuring that $v_k(t)$ does not
exhibit finite-time blowup before transiting through a neighbourhood of the singularity at the
origin; see~\autoref{lemma:K1-estimates} for details and~\ref{subsec:example-k0-2} for an example.

\subsection{Critical manifold}

Clearly, the system in~\eqref{eq:galerkin_system_temp} is a fast-slow system in the standard form of
GSPT, with $\varepsilon$ the (small) singular perturbation parameter and $\{u_k\}$ and
$\{v_k\}$, $k=1,2,\dots,k_0$, the fast and slow variables, respectively. The critical manifold
$\calC$ for~\eqref{eq:galerkin_system_temp} is hence given, to leading order, 
as a graph over $(u_1, u_2, \dots, u_{k_0})$, with
\begin{subequations}%
	\label{eq:critical_manif_ode}
	\begin{align}
		v_1 = f_1(u_1, u_2, \dots, u_{k_0}) 
			&:= 2^{-1/2} u_1^2 + 2^{-1/2} \sum_{j = 2}^{k_0} u_j^2\quad\text{and} \\
		v_k = f_k(u_1, u_2, \dots, u_{k_0}) 
			&:= \frac{1}{4} a^{-2} b_k u_k + 2^{1/2} u_1 u_k 
				+ \sum_{i, j = 2}^{k_0} \eta^k_{i, j} u_i u_j
	\end{align}
\end{subequations}
for $k=2,\dots,k_0$. 
Note that, in general, $\calC$ is  not normally hyperbolic: it contains attracting and saddle-type regions, as
well as non-hyperbolic sets separating those regions; examples can be found in \ref{subsec:example-k0-2}.
Of particular interest is the submanifold $\calC_0 \subset \calC$ of the critical manifold $\calC$ which is defined as
\begin{equation}%
\label{eq:C_0-definition}
\begin{aligned}
	\calC_0 := \big\{ \left( u_1, \dots, u_{k_0}, f_1(u_1, \dots, u_{k_0}),   \dots, f_{k_0}(u_1, \dots, 
	u_{k_0} \right) \in \calC \; : \qquad \\  \; u_1 < 0\text{ and } u_k = 0 \text{ for } 2 \le k \le k_0 \big\}.
 \end{aligned} 
\end{equation} 
In other words, $\calC_0$ is obtained by setting $u_k = 0$ for $k=2,\dots,k_0$
in~\eqref{eq:critical_manif_ode}, and can hence be written as the curve 
\begin{equation}
\begin{multlined}
	\calC_0 = \big\{ (u_1, \dots, u_{k_0}, v_1, \dots, v_{k_0}) \in \bbR^{2k_0} :
	 	v_1 =  2^{-1/2} u_1^2 + \calO \left( u_1^3 \right),\\
		\text{ with } u_1 < 0  
		\text{ and }  u_k = 0 = v_k \text{ for } 2 \le k \le k_0 \big\}
		\end{multlined}
\end{equation}
that lies in the $(u_1,v_1)$-plane. The set $\calC_0$ corresponds directly to the set of constant
functions $S_0$, given by~\eqref{eq:critical-manifold-pde}. We will denote the slow manifold that is
obtained from $\calC_0$ via GSPT by either $\calC_\varepsilon$ or $\calC_{\varepsilon, k_0}$, to
emphasise the dependence thereof on $k_0$.
\begin{remark}
	Note that in Section~\ref{sec:slow-and-galerkin}, both $Y^\zeta_S$ and $X^\zeta_S$ are finite-dimensional, and that
	$G_{\varepsilon, \zeta}$ can hence be viewed as the Fenichel slow manifold perturbing off the normally hyperbolic
	subset $\calC_0$ of the critical manifold of the fast-slow system in~\eqref{eq:galerkin_system_temp},
	for $k_0$ defined by $\zeta$ through~\eqref{eq:zeta-and-k0}.
\end{remark}

\begin{lemma}
	The subset $\calC_0$ of the critical manifold $\calC$ is normally hyperbolic and attracting
	under the layer flow that is obtained for $\varepsilon = 0$ in \eqref{eq:galerkin_system_temp}.
\end{lemma}
\begin{proof}
	Linearising the fast flow of \eqref{eq:galerkin_system_temp} about $\calC_0$, we find the Jacobian matrix
	\begin{equation}
		\mathrm{diag} \left\{ 2^{1/2} u_1, 2^{1/2} u_1 + \frac{1}{4} a^{-2} b_2, \dots,  2^{1/2} u_1 +
		\frac{1}{4} a^{-2} b_{k_0} \right\},
	\end{equation}
	which implies that $\calC_0$ is normally hyperbolic and attracting for $u_1 < 0$ bounded away from zero. (Recall that
	$b_k < 0$ for $k \in \bbN$ and $k\neq 1$.)
\end{proof}
Since the eigenvalues of a matrix depend continuously on its entries, and since the eigenvalues of
the linearisation about $\calC_0$ are all strictly negative, there exists a full neighbourhood
around $\calC_0$ in $\calC$, with $u_1 < 0$ bounded away from zero, which is normally hyperbolic and
attracting under the layer flow of \eqref{eq:galerkin_system_temp}. The flow in that neighbourhood
is directed towards the origin where, as can be seen from the above linearisation, normal
hyperbolicity is lost and which is hence a partially degenerate steady state
of~\eqref{eq:galerkin_system_temp}. The description of the dynamics near the origin therefore
requires the application of geometric desingularisation.

\subsection{Statement of main result}

We are now ready to formulate our main result, which concerns the transition between two
appropriately defined sections $\DeltaIn$ and $\DeltaOut$ for the flow generated
by~\eqref{eq:galerkin_system_temp}. These sections of the phase space are defined as follows:
consider the set
\begin{equation}
	\{ (u_1, v_1) \; : \; u_1 \in J\text{ and }v_1 = \rho^2 \} \subset \bbR^{k_0} \cross \bbR^{k_0}
\end{equation}
for small $\rho > 0$ and a suitable interval $J$, and let $\DeltaIn$ be a neighbourhood of that set
in $\bbR^{k_0} \cross \bbR^{k_0}$. Similarly, define $\DeltaOut$ as a neighbourhood of the set 
\begin{equation}
	\{ (u_1, v_1) \; : \; u_1 = \rho\text{ and }v_1 \in \bbR \} \subset \bbR^{k_0} \cross \bbR^{k_0}
\end{equation}
that is contained in the $(u_1, v_1)$-plane. More explicitly, let
\begin{equation}%
	\label{eq:delta-in-definition}
	\begin{aligned}
		\DeltaIn = \Big\{ &(u_1, \dots, u_{k_0}, v_1, \dots, v_{k_0}) \in \bbR^{2k_0} : 
		u_1 \in \left( -2^{1/4} \rho - \Cin_{u_1}, -2^{1/4} \rho + \Cin_{u_1} \right) , 
		v_1 = \rho^2, |u_k| \le \Cin_{u_k}, \\ 
		&\text{ and }|v_k| \le \Cin_{v_k}\text{ for } 2\leq k\leq k_0 \Big\}
	\end{aligned}
\end{equation}
and
\begin{equation}%
	\label{eq:delta-out-definition}
	\begin{aligned}
		\DeltaOut = \Big\{ (u_1, \dots, u_{k_0}, v_1, \dots, v_{k_0}) \in \bbR^{2k_0} : u_1 = \rho,
	 v_1 \in \bbR, |u_k| \le \Cout_{u_k}, 
		\text{ and } |v_k| \le \Cout_{v_k}\text{ for }
		2\leq k \leq k_0 \Big\},
	\end{aligned}
\end{equation}
where $\Cin_{u_1}$, $\Cin_{u_k}$, $\Cin_{v_k}$, $\Cout_{u_k}$, and $\Cout_{v_k}$, for $2 \le k \le k_0$,
are appropriately chosen small constants.
\begin{figure} 
	\centering
	\begin{overpic}[scale=0.65]{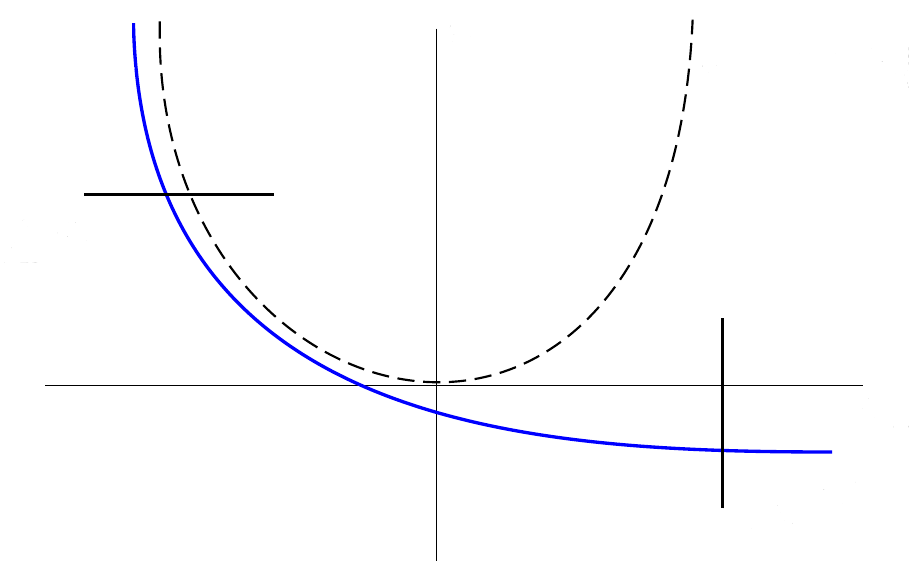}
		\put(3, 38){$\DeltaIn$}
		\put(80, 5){$\DeltaOut$}
		\put(12, 50){\textcolor{blue}{$\calC_\varepsilon$}}
		\put(20, 53){$\calC_0$}
		\put(78, 50){$v_1 = 2^{-1/2} u_1^2$}
		\put(50, 58){$v_1$}
		\put(97, 18){$u_1$}
	\end{overpic}
	\caption{Illustration of the main result, Theorem~\ref{prop:transition_full}, in its projection onto the $(u_1, v_1)$-plane. The sections
	$\DeltaIn$ and $\DeltaOut$ are, in fact, full neighbourhoods around the shown line
	intervals in $u_1$ and $v_1$. Given $k_0 \in \bbN$ fixed, trajectories of~\eqref{eq:galerkin_system_temp} that are
	initiated in $\DeltaIn$ will intersect $\DeltaOut$ transversely for $\varepsilon$ sufficiently
	small.}
\end{figure}
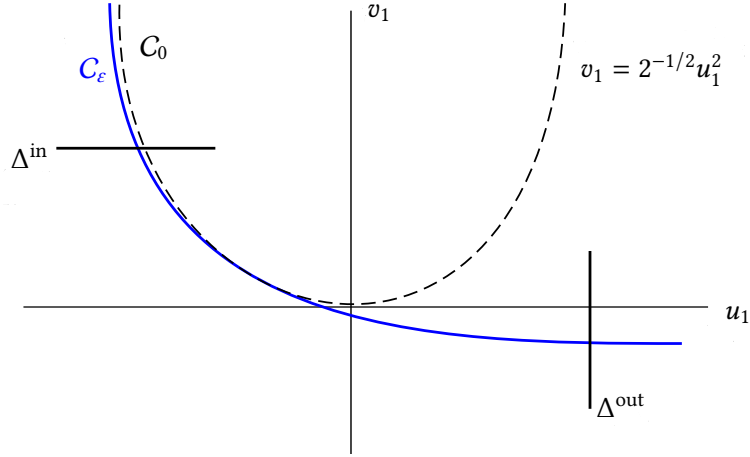
Given these definitions, we have the following result on the transition map between the sections
$\DeltaIn$ and $\DeltaOut$ that is induced by the flow of \eqref{eq:galerkin_system_temp}.
\begin{theorem}%
	\label{prop:transition_full}
	Fix $k_0 \in \bbN$, and consider the subset $\Rin \subset \DeltaIn$ defined by
	\begin{equation}%
		\begin{aligned}
			\Rin = \Rin(\varepsilon) := \Big\{ &(u_1, \dots, u_{k_0}, v_1, \dots, v_{k_0}) \in
		\bbR^{2k_0} : u_1 \in \left( -2^{1/4} \rho - \Cin_{u_1}, -2^{1/4} \rho + \Cin_{u_1} \right), 
		v_1 = \rho^2, |u_k| \le	\Cin_{u_k}, \\
			&\text{ and } \; |v_k| \le \Cin_{v_k} \varepsilon^{4/3}\text{ for }2\leq k\leq
			k_0 \Big\}.
		\end{aligned}
	\end{equation}
	Then, there exists $\varepsilon_0(k_0)$ such that for $0 < \varepsilon < \varepsilon_0$, the
	system in~\eqref{eq:galerkin_system_temp} admits a well-defined transition map
	\[
		\Pi : \Rin \rightarrow \DeltaOut.
	\] 
	Let $(\uonein, \vonein, \ukin, \vkin) \in \Rin$ and
	\[
		(\uoneout, \voneout, \ukout, \vkout) := \Pi (\uonein, \vonein, \ukin, \vkin);
	\]
	then,
	\begin{equation}
		\begin{aligned} 
			|\voneout| &= \calO\left(\varepsilon^{2/3}\right),  && \uoneout = \rho, \\
		|\ukout| & \le C |\ukin|,\quad{and} && |\vkout| \le C |\vkin| 
		\end{aligned}
	\end{equation}
	for $2\leq k\leq k_0$ and a positive generic constant $C$ which may differ between estimates.
	In particular, the slow manifolds $\calC_{\varepsilon}$ cross the section $\DeltaOut$
	transversely.  In addition, the restriction of $\Pi$ to $I := \{ (u_1, u_k, v_1, v_k) \in \Rin :
	u_k, v_k\text{ fixed for } 2 \le k \le k_0\}$ is a contraction with rate $e^{-c/\varepsilon}$
	for any suitable choice of $\{u_k, v_k\}$ and some constant $c > 0$.
\end{theorem}
\begin{remark}
	In~\eqref{eq:galerkin_system_temp}, the equations for $(u_1, v_1)$ reduce to those for the
	classical singularly perturbed planar fold \cite{KruSzm3} if we set $u_k = 0$ for $ 2 \le k \le
	k_0$. Here, we perform a similar analysis as in~\cite{KruSzm3} while controlling the
	higher-order modes $\{u_k, v_k\}$, $2 \le k \le k_0$. Note that we are restricting to initial
	data for the system of PDEs in \eqref{eq:initial_problem} that are close to constant functions,
	which translates to small initial data $\{u_k(0), v_k(0)\}$ for the system of ODEs in
	\eqref{eq:galerkin_system_temp}.  As mentioned already, the dependence on $\varepsilon$ in the
	initial values for $v_k$, $2 \le k \le k_0$, is essential to ensure that trajectories of the
	Galerkin system in \eqref{eq:galerkin_system_temp} do not exhibit finite-time blowup before
	reaching $\DeltaOut$; cf.~again Section~\ref{desingul} for the corresponding estimates
	and~\ref{subsec:example-k0-2} for an illustrative example.
\end{remark}

%
%
%
%
%
\section{Geometric desingularisation}\label{desingul}
To describe the dynamics of the system of equations in~\eqref{eq:galerkin_system_temp} near the
origin, which is a partially degenerate steady state,  we will apply the method of geometric
desingularisation by considering $\varepsilon$ as a variable in \eqref{eq:galerkin_system_temp},
which is included in the quasi-homogeneous spherical coordinate transformation
\begin{equation}
	\label{eq:blowup_transform}
	u_k = \bar r^{\alpha_k} \bar u_k,\quad v_k = \bar r^{\beta_k} \bar v_k,\quad\text{and}\quad 
	\varepsilon = \bar r^\gamma \bar \varepsilon.
\end{equation}
Here, $k = 1, 2, \dots, k_0$ and $(\bar u_1, \bar v_1, \dots, \bar u_{k_0}, \bar v_{k_0},
\bar\varepsilon) \in \mathbb{S}^{2 k_0}$, with $\mathbb{S}^{2 k_0}$ denoting the $2k_0$-sphere in
$\mathbb{R}^{2k_0+1}$ and $r\in[0,r_0]$, for $r_0>0$ sufficiently small. The weights $\alpha_k$,
$\beta_k$, and $\gamma$ in \eqref{eq:blowup_transform} will be determined by a rescaling argument
below.

In analogy to the desingularisation of the well-known planar fold via blow-up, performed in~\cite{KruSzm3}, we shall introduce
three coordinate charts $K_1$, $K_2$, and $K_3$, which are formally obtained by setting 
$\bar v_1 = 1$, $\bar\varepsilon = 1$, and $\bar u_1 = 1$, respectively, in \eqref{eq:blowup_transform}. As is convention, we will denote the variables
corresponding to $u_k$, $v_k$, and $\varepsilon$ in chart $K_i$ ($i=1,2,3$) by $u_{k,i}$, $v_{k,i}$, and $\varepsilon_{i}$, respectively.

In a nutshell, our strategy will be to retrace the analysis in \cite{KruSzm3} in each of these
charts; crucially, we will need to control the higher-order modes in
\eqref{eq:galerkin_system_temp}, i.e., the variables $\{u_k, v_k\}$ for $k=2,\dots,k_0$, in the
process. To be precise, we will verify that these additional variables will either remain uniformly
bounded (in $\varepsilon$ and $k$) or decay in the transition through the coordinate charts $K_1$,
$K_2$, and $K_3$. 

A significant challenge to our proposed strategy stems from the fact that, without taking into
consideration the length of the spatial domain $a$, one cannot obtain non-trivial dynamics on the
so-called blow-up locus that is given by $\{\bar r = 0\}$. 
To overcome that challenge, we could include $a$ as an auxiliary variable in the quasi-homogeneous
blow-up transformation in \eqref{eq:blowup_transform} by writing $a = \bar r^\eta \bar a$, which is
the approach taken in~\cite{EngelKuehn1}. That approach, however, has the disadvantage that the
resulting vector fields are not even continuous for $a = 0$, as the exponent $\eta$ is negative.

A key novelty here, in comparison to~\cite{EngelKuehn1}, is that we adopt an alternative approach by
defining a new constant $A$ via
\begin{equation}%
	\label{eq:a_and_A}
	a = A \varepsilon^{p},
\end{equation}
with $p \in \bbR$ to be determined, which allows us to obtain non-trivial dynamics for $\bar r = 0$
without the conceptual difficulties encountered in \cite{EngelKuehn1}.
Regardless of the approach used, it appears that some rescaling of the domain in
\eqref{eq:initial_problem} is necessary to perform a successful geometric desingularisation, which
is an intrinsic consequence of the Galerkin system in \eqref{eq:galerkin_system_temp} originating
from the discretisation of a system of PDEs.  Substitution of~\eqref{eq:a_and_A}
into~\eqref{eq:galerkin_system_temp} yields
\begin{subequations}%
	\label{eq:blowup_original}
	\begin{align}
		u_1^\prime &= - v_1 + 2^{-1/2} u_1^2 + 2^{-1/2} \sum_{j = 2}^{k_0} u_j^2 + H^u_1, \\
		 v_1^\prime &= -2^{1/2} \varepsilon , \\
		 u_k^\prime &= \frac{1}{4 A^2} b_k \varepsilon^{-2p} u_k - v_k + 2^{1/2} u_1 u_k + 
			\sum_{i, j = 2}^{k_0} \eta^k_{i, j} u_i u_j + H^u_k, \\
		 v_k^\prime &= \frac{1}{4 A^2} b_k \varepsilon^{-2p + 1} v_k + \varepsilon H^v_k, \\
		 \varepsilon^\prime &= 0.
	\end{align}
\end{subequations}

\begin{remark}
	The $\varepsilon$-dependent rescaling of the domain for \eqref{eq:initial_problem}
	through~\eqref{eq:a_and_A} changes the fast-slow structure of the original system
	in~\eqref{eq:galerkin_system_temp}; in particular, the origin is now a fully degenerate steady
	state of \eqref{eq:blowup_original}. While singular objects such as steady states or manifolds
	for \eqref{eq:blowup_original} in blow-up space no longer correspond directly to singular
	objects from the layer and reduced problems for \eqref{eq:galerkin_system_temp}, the two systems
	are equivalent for non-zero $\varepsilon$. Hence, our findings will equally apply to
	\eqref{eq:galerkin_system_temp} in the original coordinates, i.e., after ``blow-down".
\end{remark}
A rescaling argument shows that the weights in~\eqref{eq:blowup_transform}, as well as the power $p$
in \eqref{eq:a_and_A}, must satisfy the following relations: 
\begin{subequations}%
	\begin{align}
		\beta_1 &= 2\alpha_1, \\
		\alpha_k &= \alpha_1 \quad\text{for }2 \le k \le k_0, \\
		\gamma - \beta_1 &= \alpha_1, \\
		\label{eq:blowup_coeff_a}
		-2p \gamma &\ge \alpha_1, \\
		\beta_j &= 2 \alpha_1 \quad\text{for }2 \le j \le k_0, \\
		\gamma - 2 p \gamma &\ge \alpha_1.
	\end{align}
\end{subequations}
We see from the first three equations above that the consecutive ratios $\alpha_k : \beta_k :
\gamma$  must be $1:2:3$, as in the finite-dimensional case, see e.g.~\cite{KruSzm3}. The smallest
integers and the resulting power $p$ that satisfy these relations are
\begin{equation}\label{scalling_1}
	\alpha_k = 1, \quad \beta_k = 2, \quad \gamma = 3, \quad\text{and}\quad p = - \frac{1}{6}.
\end{equation}
\begin{remark}
	The choice $p = - \frac{1}{6}$ is the unique one that leaves no factor of $r_i$ after
	desingularisation in the resulting equations for $u_{k, i}$ in chart $K_i$, with  $i=1,2,3$,
	where one also requires equality in~\eqref{eq:blowup_coeff_a}, making use of the relation $3
	\alpha_1 = \gamma$.  Furthermore, note that the weights in \eqref{scalling_1} are consistent
	with the scaling obtained from a ``desingularisation'' of the system of PDEs in
	\eqref{eq:initial_problem}; see Section~\ref{concl} for details. 
\end{remark}

For future reference, we also state the changes of coordinates between charts $K_1$, $K_2$, and $K_3$, as follows.
\begin{lemma}
	The change of coordinates $\kappa_{12}$ between charts $K_1$ and $K_2$ is given by
	\begin{align}%
		\label{eq:change_K1_K2}
		\kappa_{12} &:\ u_{1, 2} = \varepsilon_1^{-1/3} u_{1, 1},\; v_{1, 2} = \varepsilon_1^{-2/3},
			\;u_{k, 2} = \varepsilon_1^{-1/3} u_{k, 1}, \;v_{k, 2} = \varepsilon_1^{-2/3} v_{k, 1},
			\;\text{and}\; r_2 = \varepsilon_1^{1/3} r_1;
    \end{align}
	its inverse $\kappa_{21}=\kappa_{12}^{-1}$ reads
    \begin{align}%
    \label{eq:change_K2_K1}
		\kappa_{21} &:\ u_{1, 1} = v_{1, 2}^{-1/2} u_{1, 2},\; r_1 = v_{1, 2}^{1/2} r_2,
			\; u_{k, 1} = v_{1, 2}^{-1/2} u_{k, 2},\; v_{k, 1} = v_{1, 2}^{-1} v_{k, 2}, 
			\;\text{and}\; \varepsilon_1 = v_{1, 2}^{-3 / 2}.
	\end{align}
	Between charts $K_2$ and $K_3$, we have the following change of coordinates:
	\begin{equation}%
		\label{eq:change_K2_K3}
		\kappa_{23}:\ r_3 = u_{1, 2} r_2, \;v_{1, 3} = u_{1, 2}^{-2} v_{1, 2}, 
			\;u_{k, 3} = u_{1, 2}^{-1} u_{k, 2}, \;v_{k, 3} = u_{1, 2}^{-2} v_{k, 2},
			\;\text{and}\;\varepsilon_3 = u_{1, 2}^{-3}.
	\end{equation}
\end{lemma}
\begin{proof}
	Direct calculation.
\end{proof}
%
%
%
%

\subsection{Chart $K_1$}%
\label{subsec:K1}

The coordinate chart $K_1$ is formally defined by $\bar v_1=1$. Expressed in the coordinates of that
chart, the blow-up transformation in \eqref{eq:blowup_transform} reads
\[
	u_1 = r_1 u_{1, 1}, \quad v_1 = r_1^2, \quad 
	u_k = r_1 u_{k, 1}, \quad v_k = r_1^2 v_{k, 1}, \quad\text{and}\quad
	\varepsilon = r_1^3 \varepsilon_1.
\]
With the above transformation and after desingularisation of the resulting vector field by a factor
of $r_1$, the system in~\eqref{eq:blowup_original} becomes
\begin{subequations}%
	\label{eq:K1-system}
	\begin{align}
		u_{1, 1}' &= F_1 u_{1, 1} - 1 + 2^{-1/2} u_{1, 1}^2
			+ 2^{-1/2} \sum_{j = 2}^{k_0} u_{j, 1}^2 + H^u_{1, 1}, \\
		\label{eq:K1-system-r1}
		r_1' &= -F_1 r_1, \\
		\label{eq:K1-system-uk1}
		u_{k, 1}' &= F_1 u_{k, 1} + \frac{b_k}{4 A^2} \varepsilon_1^{1/3} u_{k, 1}
			- v_{k, 1} + 2^{1/2} u_{1, 1} u_{k, 1} + \sum_{i, j = 2}^{k_0} \eta^k_{i, j} u_{i, 1} u_{j, 1}
			+ H_{k, 1}^u, \\
		\label{eq:K1-system-vk1}
		v_{k, 1}' &= 2 F_1 v_{k, 1} + \frac{b_k}{4 A^2} r_1^3 \varepsilon_1^{4/3} v_{k, 1}
			+ \varepsilon_1 H_{k, 1}^v, \\
		\varepsilon_1' &= 3 F_1 \varepsilon_1, \label{eq:K1-system-eps1}
	\end{align}
\end{subequations}
where
\[
	F_1 = F_1(\varepsilon_1) = 2^{-1/2} \varepsilon_1,
\]
as well as
\begin{align*}
	H^u_{1, 1} &= \calO \left( r_1 \varepsilon_1, r_1^2, r_1^2 v_{j, 1}^2,
		r_1 u_{1, 1}, r_1 u_{j, 1} v_{j, 1}, r_1 u_{1, 1} u_{j, 1}^2, 
		r_1 u_{1, 1} u_{j, 1}^2, r_1 u_{j, 1} u_{i, 1} u_{l, 1} \right), \\
	H^u_{k, 1} &= \calO \left( r_1^2 v_{k, 1}, r_1^2 v_{i, 1} v_{j, 1}, r_1 u_{1, 1} v_{k, 1},
		\right. \\
		&\qquad\qquad \left. r_1 u_{k, 1}, r_1 u_{i, 1} v_{j, 1}, r_1 u_{1, 1}^2 u_{k, 1}, 
		r_1 u_{1, 1} u_{i, 1} u_{j, 1}, r_1 u_{j, 1} u_{i, 1} u_{l, 1} \right),\quad\text{and} \\
	H^v_{k, 1} &= \calO \left(
	r_1^4 v_{k,1}, r_1^4 v_{i, 1} v_{j, 1} \right)
\end{align*}
for $2 \le i, j, l \le k_0$ and $2\leq k\leq k_0$. Due to the presence of fractional powers of
$\varepsilon_1$ in Equations~\eqref{eq:K1-system-uk1} and \eqref{eq:K1-system-vk1} for $u_{k,1}$ and
$v_{k,1}$, respectively, the corresponding flow will not even be $C^1$ in $\varepsilon_1$.  Hence,
we rewrite \eqref{eq:K1-system} in terms of $\varepsilon_1^{1/3}$, which gives
\begin{subequations}%
	\label{eq:k1_mu}
	\begin{align}
		u_{1, 1}' &= F_1 u_{1, 1} - 1 + 2^{-1/2} u_{1, 1}^2 
			+ 2^{-1/2} \sum_{j = 2}^{k_0} u_{j, 1}^2 + H^u_{1, 1}, \\
		r_1' &= -F_1 r_1, \label{eq:k1_mu_r1} \\
		u_{k, 1}' &= F_1 u_{k, 1} + \frac{b_k}{4 A^2} \big( \varepsilon_1^{1/3} \big)  u_{k, 1}
			- v_{k, 1} + 2^{-1/2} u_{1, 1} u_{k, 1} + \sum_{i, j = 2}^{k_0} \eta^k_{i, j} u_{i, 1} u_{j, 1}
			+ H_{k, 1}^u, \\
		v_{k, 1}' &= 2 F_1 v_{k, 1} + \frac{b_k}{4 A^2} r_1^3 \big( \varepsilon_1^{1/3} \big)^4 v_{k, 1}
			+ \big( \varepsilon_1^{1/3} \big)^3 H_{k, 1}^v, \\
		{\big( \varepsilon_1^{1/3} \big)}^\prime &= F_1 \big(\varepsilon_1^{1/3}\big).
	\end{align}
\end{subequations}
Clearly, the flow of Equation~\eqref{eq:k1_mu} will be smooth with respect to $\varepsilon_1^{1/3}$;
in the following, we will hence refer to~\eqref{eq:k1_mu} when a higher degree of smoothness is
required.

Equation~\eqref{eq:k1_mu} admits the two principal steady states 
\begin{equation}
	p_a^{k_0} := (-2^{1/4}, 0, \mathbf{0}, \mathbf{0}, 0)\quad\text{and}\quad 
	p_r^{k_0} := (2^{1/4}, 0, \mathbf{0}, \mathbf{0}, 0),
\end{equation}
where $\mathbf{0}$ denotes the zero vector in $\bbR^{k_0 - 1}$.
\begin{lemma}%
	\label{lemma:K1:steady-states}
	The point $p_a^{k_0}$ is a partially hyperbolic steady state of Equation~\eqref{eq:k1_mu}, with
	the following eigenvalues and eigenvectors in the corresponding linearisation:
	\begin{itemize}
		\item the simple eigenvalue $-2^{3/4}$ with eigenvector $(1, 0, \dots, 0)$, corresponding to $u_{1, 1}$;
		\item the eigenvalue $-2^{3/4}$ with multiplicity $k_0 - 1$ and eigenvectors $(0,
			\dots, 1, \dots, 0)$, where non-zero entries appear at the $(k + 2)$-th position,
			corresponding to $u_{k, 1}$ ($2 \le k \le k_0$); and
		\item the eigenvalue $0$ with multiplicity $k_0 + 1$, corresponding to $r_1$, 
			$v_{k, 1}$ ($2 \le k \le k_0$), and $\varepsilon_1^{1/3}$.
	\end{itemize}
\end{lemma}
\begin{proof}
	Direct calculation.
\end{proof}
To describe the transition through chart $K_1$, i.e., to approximate the corresponding transition
map, we define the following sections for the flow of \eqref{eq:K1-system}:
\begin{equation}
	\begin{aligned}
		\SigmaInOne &:= \{ (u_{1, 1}, r_1, u_{k, 1}, v_{k, 1}, \varepsilon_1) 
			\; : \; r_1 = \rho \}\quad\text{and} \\
		\SigmaOutOne &:= \{ (u_{1, 1}, r_1, u_{k, 1}, v_{k, 1}, \varepsilon_1)
			\; : \; \varepsilon_1 = \delta\}, 
 	\end{aligned}
\end{equation}
for sufficiently small $\delta >0$. Next, we need to determine the transition time between
$\SigmaInOne$ and $\SigmaOutOne$, which will allow us to derive estimates for the corresponding
orbits as they pass through chart $K_1$.
\begin{lemma}
	The transition time between the sections $\SigmaInOne$ and $\SigmaOutOne$ under the flow of~\eqref{eq:K1-system} is given by
	\begin{equation}%
		\label{eq:k0-2:trans-time}
		T_1 = \frac{\sqrt{2}}{3} \left( \frac{1}{\varepsilon_1(0)}-\frac{1}{\delta} \right).
	\end{equation}
\end{lemma}
\begin{proof}
	The explicit solution of Equation~\eqref{eq:K1-system-eps1} for $\varepsilon_1$ reads
	\begin{equation}%
		\label{eq:k0:eps1-explicit}
		\varepsilon_1(t) = \frac{2 \varepsilon_1(0)}{2 - 3 \sqrt{2} \varepsilon_1(0)t},
	\end{equation}
	where $\varepsilon_1(0)$ denotes an appropriately chosen initial value for $\varepsilon_1$ in
	$\SigmaInOne$.  Solving the equation $\varepsilon_1(T_1) = \delta$ for $T_1$ results
	in~\eqref{eq:k0-2:trans-time}, as stated. Note that the denominator
	in~\eqref{eq:k0:eps1-explicit} remains strictly positive for all $t\in [0, T_1]$.
\end{proof}
\begin{remark}
	We refer to the time variable by $t$ throughout for simplicity of notation, even though we
	consider different systems in the three coordinate charts $K_j$, with $j=1,2,3$, as well as
	multiple parametrisations of the same system in some cases. 
\end{remark}

To give a more complete description of the geometry and, in particular, of the steady state
structure, we proceed as follows. Setting $r_1 = 0 = \varepsilon_1$ in~\eqref{eq:K1-system}, we find
the singular system
\begin{subequations}%
	\label{eq:K1-r-epsilon-zero}%
	\begin{align}
		u_{1, 1}' &= -1 + 2^{-1/2} u_{1, 1}^2 + 2^{-1/2} \sum_{j = 2}^{k_0} u_{j, 1}^2, \\
		r_1' &= 0, \\
		u_{k, 1}' &= -v_{k, 1} + 2^{1/2} u_{1, 1} u_{k, 1} + \sum_{i, j = 2}^{k_0} u_{i, 1} u_{j, 1}, \\
		v_{k, 1}' &= 0, \\
		\varepsilon_1' &= 0,
	\end{align}
\end{subequations}
from which we see that the hyperplanes $\{ r_1 = 0 \}$ and $\{\varepsilon_1 = 0\}$ are invariant, as
is their intersection. An application of the implicit function theorem shows that lines of steady
states emanate from $p_{a}^{k_0}$ and $p_{r}^{k_0}$, respectively, for $u_{1, 1}$ close to $\pm
2^{1/4}$ and $u_{k, 1}$ and $v_{k, 1}$ small, with $2 \le k \le k_0$. Locally, around $p_a^{k_0}$,
these steady states will inherit the stability of $p_{a}^{k_0}$, which we will make use of in the
estimates in the following subsection. For $k_0 = 2$, the geometry is exemplified in
Figures~\ref{fig:chart-1-k0-2-steady-states} and~ \ref{fig:chart-1-k0-2-eigenvalues}, in which case
$p_a^{2}$ and $p_r^{2}$ are connected by curves of steady states which can be calculated explicitly
from~\eqref{eq:K1-r-epsilon-zero}; see \autoref{fig:chart-1-k0-2-steady-states}. The linearisation
around those states has one zero eigenvalue and two non-trivial eigenvalues $\ell_1$ and $\ell_2$
which depend on the $u_{1, 1}$-coordinate only; these eigenvalues are plotted
in~\autoref{fig:chart-1-k0-2-eigenvalues}.

The geometry for general $k_0$ will be similar, in that $p^{k_0}_a$ and $p^{k_0}_r$ will not be
isolated, with steady states lying in the plane $\{ r_1 = 0 = \varepsilon_1 \}$ that are neutral in
the $v_{k, 1}$-directions and of varying stability in $u_{1, 1}$ and $u_{k, 1}$, for $2 \le k \le
k_0$.  States that are close to the point $p^{k_0}_a$ will be stable in the latter directions, while
those close to $p^{k_0}_r$ will be unstable in the same directions; in between, there will be steady
states of saddle type. These statements are a direct consequence of the implicit function theorem,
applied to the vector field in~\eqref{eq:K1-r-epsilon-zero}.  It is unclear whether a curve of
steady states that connects $p^{k_0}_a$ and $p^{k_0}_r$  will exist for general $k_0$, as is the
case for $k_0=2$.

Furthermore, lines of equilibria are found emanating from each steady state in $\{ r_1 = 0 =
\varepsilon_1 \}$, as can again be seen from the implicit function theorem.  These lines locally
inherit the stability of the corresponding steady states they are based on. 

\begin{remark}
	Typically, steady states in the subspace equivalent to $\{ r_1 = 0 = \varepsilon_1\}$ after
	blow-up can be viewed as intersections of critical manifolds with the blow-up locus $\{\bar
	r=0\}$~\cite{KruSzm3}. However, that is not the case here, as the rescaling of the spatial
	domain by $\varepsilon$ in \eqref{eq:a_and_A} alters the fast-slow structure of the original
	Equation~\eqref{eq:galerkin_system_temp}. If the parameter $a$ is blown up as
	in~\cite{EngelKuehn1} instead, the correspondence with the flow pre-blow-up would be retained;
	however, the resulting dynamic boundary value problem poses different technical challenges, as
	detailed there.
\end{remark}

The existence of non-hyperbolic steady states near $p^{k_0}_a$ that are attracting in the directions
of $u_{1, 1}$ and $u_{k, 1}$ for $k=2,\dots,k_0$ implies the following result.
\begin{lemma}%
	\label{lemma:K1-center-manifold}
	For sufficiently small $\rho$, $\delta$, $\Cin_{u_{1, 1}}$, $\Cin_{u_{k, 1}}$, and $\Cin_{v_{k,
	1}}$, there exists an attracting, $(k_0 + 2)$-dimensional centre manifold $M_{k_0, 1}$ at
	$p^{k_0}_a$ in \eqref{eq:k1_mu}. The manifold $M_{k_0, 1}$ is given as a graph over $\big(u_{1,
	1}, r_1, v_{k, 1}, \varepsilon_1^{1/3}\big)$, where $2 \le k \le k_0$. In particular, for
	initial conditions close to $p^{k_0}_a$, solutions of \eqref{eq:k1_mu} satisfy $u_{1, 1}(t) < 0$
	for $t \in [0, T_1]$.
\end{lemma}
\begin{proof}
	The statements follow from centre manifold theory and~\autoref{lemma:K1:steady-states}.
\end{proof}

The centre manifold argument in \autoref{lemma:K1-center-manifold} implies that if $u_{1,1}$ is
close to $-2^{1/4}$ initially, then it will remain close throughout the transition through chart
$K_1$; in particular, $u_{1,1}$ will remain negative. To obtain corresponding estimates for the
remaining variables $u_{k, 1}$ and $v_{k, 1}$, with $k=2,\dots,k_0$, we combine the classical
variation of constants formula with a fixed point argument.


\begin{figure}
	\centering
	\begin{subfigure}[t]{0.5\textwidth}
		\centering
		\begin{overpic}[scale=0.5]{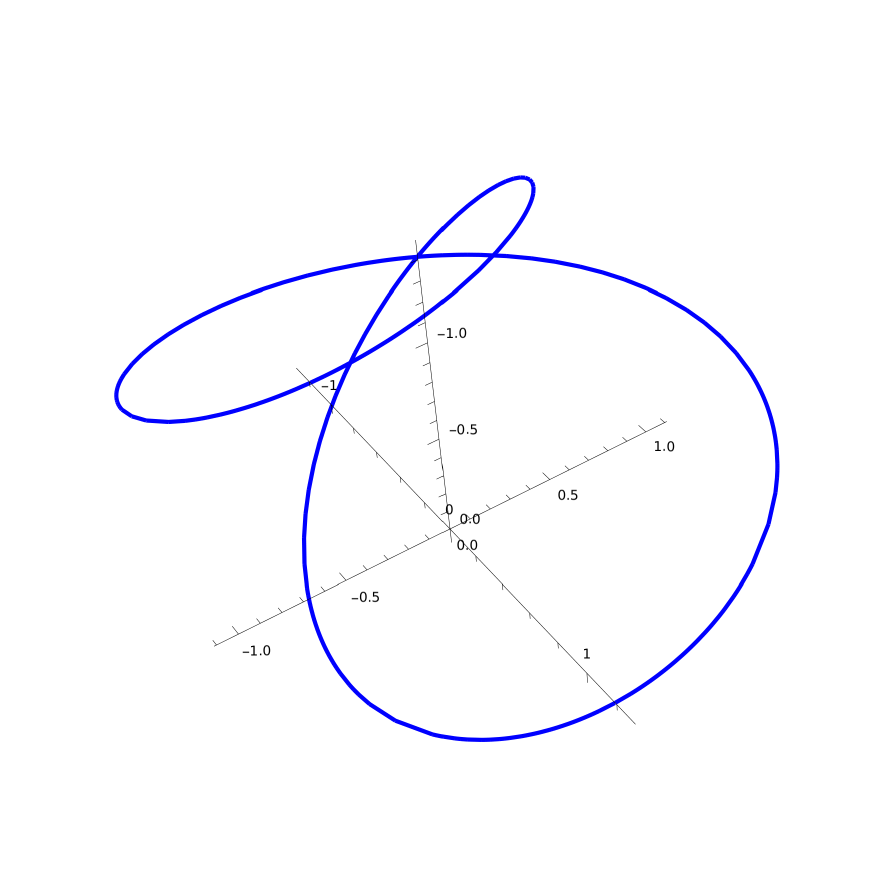}
			\put(70, 13){$u_{1, 1}$}
			\put(20, 20){$v_{2, 1}$}
			\put(43, 75){$u_{2, 1}$}
			\put(30, 58){$p_a^{k_0}$}
		\end{overpic}
		\caption{Steady states of~\eqref{eq:K1-r-epsilon-zero} when $k_0 = 2$.}%
		\label{fig:chart-1-k0-2-steady-states}
	\end{subfigure}%
	\hfill
	\begin{subfigure}[t]{0.5\textwidth}
		\centering
		\begin{overpic}[scale=0.5]{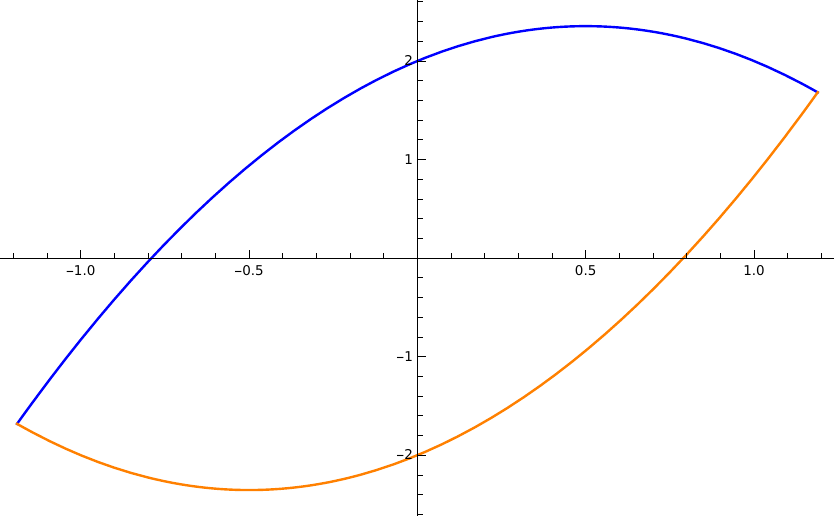}
			\put(100, 25){$u_{1, 1}$}		
			\put(20, 40){\textcolor{blue}{$\ell_1$}}
			\put(72, 15){\textcolor{orange}{$\ell_2$}}
		\end{overpic}
		\caption{The two non-trivial eigenvalues $\ell_1$ and $\ell_2$.}%
		\label{fig:chart-1-k0-2-eigenvalues}
	\end{subfigure}
	\caption{Steady state structure of~\eqref{eq:K1-r-epsilon-zero} for $k_0=2$. (a) The principal
	steady states $p^{2}_a$ and $p^{2}_r$ are connected by a pair of symmetric curves of steady
	states that are parametrised by $u_{1, 1}$. (b) The two non-trivial eigenvalues $\ell_1$ and
	$\ell_2$ of the linearisation about these steady states are plotted against $u_{1, 1}$.}
\end{figure}


\begin{figure}
	\centering
	\begin{overpic}[scale=0.50, percent]{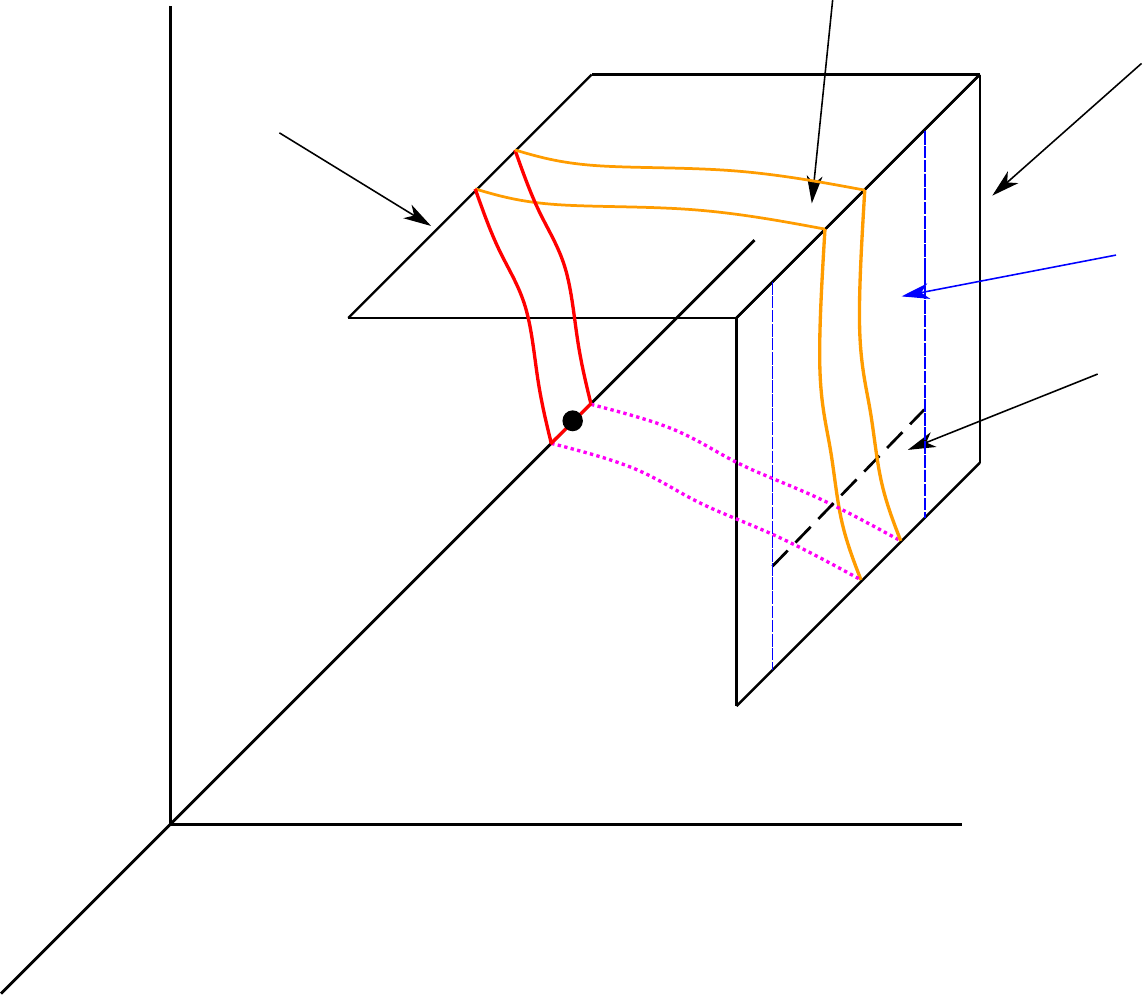}
		\put(15, 90){$\varepsilon_1$}
		\put(5, 0){$u_{1,1}, u_{2, 1}, \dots, u_{k_0, 1}$}
		\put(90, 13){$r_1, v_{2,1}, \dots, v_{k_0, 1}$}
		\put(20, 78){$\SigmaOutOne$}
		\put(70, 89){$M_{k_0, 1}$}
		\put(100, 84){$\SigmaInOne$}
		\put(100, 64){\textcolor{blue}{$R_1$}}
		\put(98, 55){$I(\varepsilon_1)$}
		\put(48, 53){\small$p^{k_0}_a$}
	\end{overpic}
	\caption{The dynamics in $K_1$ is organised around the attracting centre manifold $M_{k_0, 1}$,
	which is anchored in a curve of steady states in the subspace $\{r_1=0=\varepsilon_1\}$, one of which is
	$p^{k_0}_a$. The transition map $\Pi_1$ is defined on the subset $R_1 \subset \SigmaInOne$
	around the intersection $\SigmaInOne \cap M_{k_0, 1}$; slices of $R_1$ with $\varepsilon_1$
	constant, denoted by $I(\varepsilon_1)$, will be mapped to slices with $\varepsilon_3$ constant
	in chart $K_3$. Since $M_{k_0, 1}$ is a graph over $\big(u_{1, 1}, r_1, v_{k, 1},
	\varepsilon_1^{1/3}\big)$, see~\autoref{lemma:K1-center-manifold}, it is illustrated as having
	``thickness''.}
	\label{fig:chart1-1}
\end{figure}

\begin{lemma}%
\label{lemma:K1-estimates}
For $2 \le k \le k_0$ and $u_k(0)$ and $v_{k}(0)$ satisfying \eqref{eq:initial-data}, solutions of
	Equation~\eqref{eq:K1-system} satisfy the estimates
	\begin{equation}
		\label{eq:k0-2:uk1-estimates}%
		\begin{aligned}
			\left\vert u_{k, 1}(t) \right\vert &\le  \frac 1 {\rho}|u_{k}(0)| 
				+ \frac{8a^2 \rho}{|b_k|}  \Big[\sigma_u + \sigma_v \varepsilon_1(0)^{2/3} \rho^2\delta^{2/3}\Big(1 + \frac{8a^2}{|b_k|}\Big)\Big]
		 \end{aligned}
	\end{equation}
	and
	\begin{equation}%
		\label{eq:k0-2:vk1-estimates}
		\left\vert v_{k, 1}(t) \right\vert \le
		\frac{\delta^{2/3}}{\varepsilon_1(0)^{2/3} \rho^2} \big\vert v_{k}(0) \big\vert +
		\varepsilon_1(0)^{2/3} \delta^{2/3} \frac{8a^2 \rho^2}{|b_k|} \sigma_v \leq
		\varepsilon_1(0)^{2/3} \rho^2 \delta^{2/3}  \Big(C_{k, v_0} + \frac{8a^2}{|b_k|}\sigma_v\Big)
		\end{equation} for  all $t \in [0,T_1]$ and some $\kappa \leq \sigma_u, \sigma_v <1$, where
		$T_1$ is the transition time determined in~\eqref{eq:k0-2:trans-time} and $\kappa>0$ is as
		in~\eqref{eq:initial-data-L2}. 
\end{lemma}
\begin{proof}
We first derive the estimates for $v_{k,1}$. Application of the variation of constants formula
to~\eqref{eq:K1-system-vk1} yields
\begin{align}%
	\label{estim_vk_11}
	v_{k, 1}(t) &=  \exp \Big( \int_0^t V_{k,1}(s) \,ds \Big) v_{k, 1}(0) 
		 + \int_0^t \exp\Big(\int_s^t V_{k,1}(\tau) d\tau\Big) \varepsilon_1(s) H_{k,1}^v ds,
\end{align}
where $V_{k,1}(s) = 2^{1/2} \varepsilon_1(s) 
		+ \frac{b_k}{4 A^2} r_1^3(s) \varepsilon_1^{4/3}(s)$. 
Equation~\eqref{eq:K1-system-r1} can be solved explicitly for $r_1$ to give 
\begin{equation}%
	\label{eq:k0:r1-explicit}
	r_1(t) = 2^{-1/3} \rho {\big( 2 - 3 \sqrt{2} \varepsilon_1(0) t \big)}^{1/3},
\end{equation}
where $\varepsilon_1(0)$ denotes the initial value for $\eps_1$ and $r_1(0)=\rho$.  Note that, due
to $b_k < 0$, the second term in $V_{k,1}(s)$ is negative for all $s \in [0, T_1]$. Combination of
the above expression with the explicit solutions for $\varepsilon_1(t)$ and $r_1(t)$ in
\eqref{eq:k0:eps1-explicit} and \eqref{eq:k0:r1-explicit}, respectively, then implies
\begin{align}%
	\label{estim_vk_12} 
	|v_{k, 1}(t)| &\le \frac 2 3 \exp ( \int_0^t \frac{\varepsilon_1(0)} {\phi_1(s)}\,ds ) |v_{k,
	1}(0)| \\ \notag
	&+ \Big(\frac{\sqrt{2}} 3\Big)^{\frac 1 3} \rho^2 \frac{ e^{-\alpha \phi_1(t)^{2/ 3}}}{
		\phi_1(t)^{2/ 3}} \int_0^t \varepsilon_1(0) \phi_1(s)^{1/ 3} e^{\alpha  \phi_1(s)^{2/3}}
		|\widetilde H_{k,1}^v(s)| ds,
\end{align}
where 
$\phi_1(s) = \sqrt{2}/3-\varepsilon_1(0) s$, $\alpha =(3/\sqrt{2})^{2/3} b_k\rho^2/(4 a^2
\sqrt{2})$, and $H_{k,1}^v(t) = r_1(t)^2 \widetilde H_{k,1}^v(t)$. Evaluating the integrals in
\eqref{estim_vk_12}, we obtain
\begin{equation}%
	\label{eq:k0-2:v21-estimates}
	\left\vert v_{k, 1}(t) \right\vert 
		\le \frac{\delta^{2/3}}{\varepsilon_1(0)^{2/3}}\Big( \vert v_{k, 1}(0) \vert  +
		\frac{ 8 a^2}{|b_k|} \sup\limits_{[0,T_1]} |\widetilde H^v_{k,1}(t)|\Big)\qquad
		\text{ for all }  t\in [0, T_1].
\end{equation}
Recall that the term $\widetilde H^v_{k,1}$ is at least quadratic in $v_{k,1}$, with $k=2, \ldots, k_0$.
To estimate $u_{k,1}$, we first rewrite~\eqref{eq:K1-system-uk1} in the form
\[	
	u_{k, 1}^\prime = \Big(F_1 + \frac{b_k}{4 A^2} \varepsilon_1^{1/3} + 2^{1/2} u_{1, 1}\Big) u_{k, 1}
		- v_{k, 1} + M_k(u_{2, 1}, \dots, u_{k_0, 1}) + H_{k,1}^u
\]
for $2 \le k \le k_0$, where
\[
	M_k(u_{2, 1}, \dots, u_{k_0, 1}) := \sum_{i, j = 2}^{k_0} \eta^k_{i, j} u_{i, 1} u_{j, 1}.
\]
Application of the variation of constants formula yields 
\begin{multline}%
	\label{eq:K1-uk1-estim}
	u_{k, 1}(t) = \exp \left( \int_0^t  U_{k,1}(s) ds  \right) u_{k, 1}(0) 
		- \int_0^t \exp \left( \int_s^t U_{k,1}(\tau)d\tau \right) v_{k, 1}(s) ds \\
	+ \int_0^t \exp \left( \int_s^t U_{k,1}(\tau)d\tau \right) \Big(M(u_{2, 1}(s), \dots, u_{k_0, 1}(s))+ H_{k,1}^u\Big) ds,
\end{multline}
where $U_{k,1}(s) := 2^{-1/2} \varepsilon_1(s) + \frac{b_k}{4 A^2} \varepsilon_1^{1/3}(s) + 2^{1/2} u_{1, 1}(s)$.

Due to~\autoref{lemma:K1-center-manifold}, we have $u_{1, 1}(s) < 0$ for $s \in [0, T_1]$; hence, in
the following estimates, we replace $U_{k,1}(s)$ with $\widetilde U_{k,1}(s) := 2^{-1/2}
\varepsilon_1(s) + \frac{b_k}{4 A^2} \varepsilon_1^{1/3}(s)$, as $\exp(\int_s^t U_{k,1}(\tau) d\tau)
\leq \exp(\int_s^t \widetilde U_{k,1}(\tau) d\tau)$ for all $0\leq s < t\leq T_1$. Direct integration gives
\begin{equation} 
	\begin{aligned}
		0&\leq \mathcal I_1(t) = \exp \left( \int_0^t \widetilde U_{k,1}(s) ds  \right) \\ 
			&= \frac 1{\big[1 - (3/\sqrt{2}) \varepsilon_1(0) t\big]^{1/3}} \exp( \frac{ b_k}{4A^2 \sqrt{2}} 
				\Big[ \Big( \frac 1 { \varepsilon_1(0)}\Big)^{\frac 2 3} 
			-  \Big( \frac 1 { \varepsilon_1(0)} - \frac 3 { \sqrt{2}} t\Big)^{\frac 2 3} \Big])
			\leq 1,
	\end{aligned}
\end{equation} 
since $\mathcal I_1(t)$ is a non-increasing function for $\delta \leq \pi^3 2^{-5/4}/(\rho^{3/2} \varepsilon_1(0)^{1/2}) =\pi^3 2^{-5/4}/\varepsilon^{1/2}$,
and 
\[
	\mathcal I_1(T_1) = \frac{ \delta^{1/3}}{\varepsilon_1(0)^{1/3}} \exp(\frac{ b_k}{4A^2\sqrt{2}}
	\Big[ \Big( \frac 1 { \varepsilon_1(0)}\Big)^{\frac 2 3} -  \Big( \frac 1 {\delta}\Big)^{\frac 2 3}
	\Big]) \leq \exp(\frac{ b_k}{8 A^2\sqrt{2} } \Big( 1- \frac 1 {\alpha }\Big) \Big( \frac 1 {
	\varepsilon_1(0)}\Big)^{\frac 2 3} ) 
\]
for $\delta \geq \alpha \varepsilon_1(0)>0$ with $\alpha\geq 1$.  Next, we have that 
\begin{equation*}
\begin{aligned}
	0 &\leq \mathcal I_2(t) = \int_0^t \exp \Big( \int_s^t \widetilde U_{k,1}(\tau) d\tau \Big) ds = \varepsilon_1(t)^{\frac{1}{3}} \frac{ 4 A^2 }{|b_k|}
		\bigg[ \Big( \frac{1}{\varepsilon_1(0)} - \frac{3}{\sqrt{2}} t \Big)^{2/3}  
		\\
  & \quad  - \left( \frac 1{ \varepsilon_1(0)}\right)^{\frac 2 3}
		\exp \bigg( \frac{1}{\sqrt{2}} \frac{ b_k}{4A^2}
		\bigg[\Big( \frac 1{ \varepsilon_1(0)} \Big)^{\frac 2 3} 
		- \Big( \frac{1}{\varepsilon_1(0)} - \frac{3}{\sqrt{2}} t \Big)^{\frac 23} \bigg] \bigg) \bigg] \\
	& \quad+ \varepsilon_1(t)^{\frac 1 3}\frac{ 16 A^4 \sqrt{2}}{|b_k|^2}
	\bigg[ 1- \exp\bigg(\frac{1}{\sqrt{2}} \frac{ b_k}{4A^2}\bigg[ \Big( \frac 1{ \varepsilon_1(0)}
		\Big)^{\frac{2}{3}} - \Big( \frac 1{ \varepsilon_1(0)} - \frac{3} {\sqrt{ 2}} t
		\Big)^{\frac{2}{3}} \bigg]\bigg) \bigg] \\
	&\leq \frac {8\rho a^2}{|b_k|},
	\end{aligned}
\end{equation*}
where we have used $A^2 = \varepsilon^{1/3} a^2 = \rho \varepsilon_1(0)^{1/3}a^2$ and 
\[
\begin{aligned} 
\calI_2(T_1) & = \delta^{-1/3} \frac{ 4 A^2 }{|b_k|} \bigg[ 1 - 
\Big( \frac \delta{ \varepsilon_1(0)}\Big)^{2/3}
\exp(\frac 1 { \sqrt{2}} \frac{ b_k}{4A^2} \Big[\Big( \frac 1{ \varepsilon_1(0)}\Big)^{2/3} - \frac 1{ \delta^{2/3}} \Big]) \bigg]
\\
& + \delta^{1/3}
\frac{ 16 A^4 \sqrt{2}}{|b_k|^2}
\bigg[1- \exp(\frac 1 { \sqrt{2}} \frac{ b_k}{4A^2}\Big[ \Big( \frac 1{ \varepsilon_1(0)} \Big)^{\frac 2 3} -  \frac 1 { \delta^{2/3}}\Big]) \bigg]
\leq \frac 1{ \delta^{1/3} } \frac{ 4 A^2 }{|b_k|}\Big[1 + \frac{ 4 A^2 \sqrt{2}}{|b_k|} \Big].
\end{aligned}
\]
Here, we again have $0<\alpha\varepsilon_1(0) \leq \delta <1$, with $\alpha \geq 1$.

To complete the estimates, we shall use a fixed point argument and define the set
\[
	\begin{multlined}
		\calB_1 = \bigg\{(\tilde u_{2, 1}, \dots, \tilde u_{k_0, 1} , \tilde v_{2, 1}, \dots, \tilde v_{k_0, 1} )
			: \tilde u_{k, 1}, \tilde v_{k,1} \in  C[0, T], 2 \le k \le k_0, \\	
		 \text{with }\sup_{[0,T_1]}|\tilde u_{k, 1}(t)| \le  C_{k,u}, 
		 	\sup_{[0,T_1]}|\tilde v_{k, 1}(t)| \le C_{k,v}, \\ 
		 \sum_{k=2}^{k_0} C_{k,u}^2 \leq \tilde \sigma_u,\text{ and } 
		 	\sum_{k=2}^{k_0} C_{k,v}^2 \leq \tilde \sigma_v \varepsilon_1(0)^{4/3} \bigg\},
	\end{multlined}
\]
where $\tilde \sigma_u, \tilde \sigma_v \leq 1$.

Considering $M_k(\tilde u_{2, 1}, \dots, \tilde u_{k_0, 1})$ and 
\[
	H_{k,1}^l = H_{k,1}^l(\tilde u_{2,1},\ldots,\tilde u_{k_0,1}, \tilde v_{2,1}, \ldots, \tilde v_{k_0,1}),
\] 
with $l=u,v$, in \eqref{estim_vk_11} and \eqref{eq:K1-uk1-estim}
for $(\tilde u_{2, 1}, \dots, \tilde u_{k_0, 1} , \tilde v_{k,1}, \ldots, \tilde v_{k_0, 1})\in
\mathcal B_1 $, we obtain  a map $\mathcal N_1$ given by $\mathcal N_1(\tilde u_{2,1}, \ldots,
\tilde u_{k_0,1}, \tilde v_{2,1}, \ldots, \tilde v_{k_0,1}) = ( u_{2, 1}, \dots, u_{k_0, 1}, v_{2,
1}, \dots, v_{k_0, 1})$. Solutions of \eqref{estim_vk_11} and \eqref{eq:K1-uk1-estim}
correspond to the fixed points of $\mathcal N_1$.  

We shall show that $\mathcal N_1 : \mathcal B_1 \to \mathcal B_1$. Our assumptions on the initial conditions, together with \eqref{eq:k0-2:v21-estimates}, yield 
\begin{equation}%
	\label{estim_vk_22}
	\left\vert v_{k, 1}(t) \right\vert 
		\le \delta^{2/3} \varepsilon_1(0)^{2/3}\Big( \rho^2 C_{k, v_0} +
		\frac{ 8 a^2 \rho^2}{|b_k|} \sigma_v \Big)\qquad
		\text{ for all }  t\in [0, T_1], 
\end{equation}
where we have used $ | \widetilde H^v_{k,1}(t)| \leq C_1\rho^2 \sum_{k=2}^{k_0} |\tilde
v_{k,1}|^2 \leq \rho^2 C_2 \varepsilon_1(0)^{4/3} \tilde \sigma_v \leq \rho^2 \varepsilon_1(0)^{4/3}
\sigma_v$.
Then, 
\[
    |u_{k,1}(t)| \leq |u_{k,1}(0)| + 
    \frac{ 8 a^2\rho}{|b_k|}\big( C_3 \varepsilon_1(0)^{2/3} + \sigma_u\big) \qquad
		\text{ for all }  t\in [0, T_1], 
\]
where $|M_k + H_{k,1}^u|\leq C_4  \sum_{k=2}^{k_0} |\tilde u_k|^2 \leq C_5 \tilde \sigma_u =  \sigma_u$.

Thus, for $0<\rho<1$ and $0<\sigma_u,\sigma_v<1$, we obtain that $\mathcal N_1 : \mathcal B_1 \to
\mathcal B_1$, which implies the estimates in \eqref{eq:k0-2:uk1-estimates} and
\eqref{eq:k0-2:vk1-estimates}. 
%
\end{proof}
\begin{remark}
	Note that if $H^v = 0$, then it is sufficient to consider $|v_k(0)| \leq C_{k, v_0}
	\varepsilon^{1/2}$ and $|u_k(0)| \leq C_{k, u_0}$. For more general higher-order terms of the form
	\[ 
		H^v_k= \calO (u_i u_j, v_i v_j, v_1v_k, v_i v_j),
	\]	
	with $i,j=2, \ldots, k_0$, we would have to assume that $|u_{k}(0)| \le  C_{k,u_0}
	\varepsilon^{2/3}$. Then, in the definition of $\mathcal B_1$, we would consider
	$\sum_{k=2}^{k_0}  C_{k,u}^2\leq  \varepsilon^{4/3} \tilde \sigma_u$, which would imply
	\[
	|u_{k,1}(t)| \leq \varepsilon_1(0)^{2/3} \rho^{2} C_{k, u_0} + 
	   \varepsilon_1(0)^{2/3} \frac{ 8 a^2\rho}{|b_k|}\big(\sigma_u + \sigma_v\big).
	\]
\end{remark}

Given the above estimates, the transition map $\Pi_1$ in chart $K_1$ will be defined on the set $R_1
\subset \SigmaInOne$, which is given by 
\begin{multline}%
	\label{eq:def_R}
	R_1 := \Big\{ (u_{1, 1}, r_1, u_{k, 1}, v_{k, 1}, \varepsilon_1) 
		:|u_{1, 1} + 2^{1/4}| \le \Cin_{u_{1, 1}}, r_1 = \rho,  \\
		 |u_{k, 1}| \le \Cin_{u_{k, 1}} , 
		 |v_{k, 1}| \le \Cin_{v_{k, 1}} \varepsilon_1^{4/3} \text{ for } k = 2, \ldots, k_0, \text{ and }\varepsilon_1 \in [0, \delta] \Big\};
\end{multline}
see \autoref{fig:chart1-1}. The set $R_1$ is precisely the set $\Rin \subset \DeltaIn$, transformed
into the coordinates of chart $K_1$. For $\varepsilon_1 \in [0, \delta]$ fixed, we also define the
slices $I(\varepsilon_1)\subset R_1$ as
\begin{equation}%
	\label{eq:def_Imu}
	I(\varepsilon_1) := \left\{ (u_{1, 1}, r_1, u_{k, 1}, v_{k, 1}, \varepsilon_1) \in R_1 \; : 
		\; \varepsilon_1\in \left[ 0, \delta \right]\text{ fixed} \right\}.
\end{equation}
These slices will be useful when combining the transition through chart $K_1$ with those through
charts $K_2$ and $K_3$, as $I(\varepsilon_1)$ will be mapped to sets with $\varepsilon_3$ constant
in an appropriately defined section $\SigmaOutThree$. 

We summarise our findings on the transition through chart $K_1$, and on the corresponding map $\Pi_1$.
\begin{prop}%
	\label{prop:K1:transition-map}
	The transition map $\Pi_1 : R_1 \rightarrow \SigmaOutOne$ is well-defined. For 
	\[ 
		(u_{1, 1}, \rho, u_{k, 1}, v_{k, 1}, \varepsilon_1)\in R_1, 
			\quad\text{with } k = 2, \dots, k_0, 
	\]
	denote
	\begin{equation}
		\Pi_1(u_{1, 1}, \rho, u_{k, 1}, v_{k, 1}, \varepsilon_1) 
			= (\uoneoutone, \rout_1, \ukoutone, \vkoutone, \delta).
	\end{equation}
	Then, the following estimates hold:
	\begin{subequations}
		\begin{align}
			|\uoneoutone + 2^{1/4} | & \le \Cout_{u_{1, 1}}, \label{eq:prop-K1-u1} \\
			\rout_1 & \in [0, \rho], \label{eq:prop-K1-r} \\
			|\ukoutone| & \le \Cout_{u_{k, 1}},\quad\text{and} \label{eq:prop-K1-uk} \\
			|\vkoutone| & \le \Cout_{v_{k, 1}} \delta^{2/3}, \label{eq:prop-K1-vk}
		\end{align}
	\end{subequations}
	where $\Cout_{u_{k, 1}}$, $\Cout_{u_{k, 1}}$, and $\Cout_{v_{k, 1}}$ are appropriately chosen
	constants. Furthermore, the restriction $\Pi_1 |_{I(\varepsilon_1)}$ is a contraction, with rate bounded by 
	$C \exp \left( c T_1 \right)$, where $C > 0$ and $-2^{3/4} < c < 0$. 
\end{prop}
\begin{proof} 
	The estimates in \eqref{eq:prop-K1-uk} and \eqref{eq:prop-K1-vk} follow directly from the
	definition of $R_1$ in~\eqref{eq:def_R} and~\autoref{lemma:K1-estimates},
	while~\eqref{eq:prop-K1-r} is immediate from the observation that $r_1(t)$ is decreasing,
	by~\eqref{eq:k1_mu_r1}. Finally,~\eqref{eq:prop-K1-u1} and the stated contraction property are
	due to~\autoref{lemma:K1-center-manifold} and the existence of the attracting centre manifold
	$M_{k_0, 1}$.
\end{proof}
%
%
%
%
%

\subsection{Chart $K_2$}

As will become apparent, the dynamics of \eqref{eq:blowup_original} in chart $K_2$ can be seen as a
regular perturbation of the planar subsystem for the first two modes $\{u_1,v_1\}$,
after transformation to $K_2$. In particular, for $r_2 = 0$, that subsystem reduces to the
well-studied Riccati equation \cite{MisRoz}. As the requisite analysis is similar to that in the
corresponding rescaling chart for the singularly perturbed planar fold \cite{KruSzm3}, we merely
outline it here.

In chart $K_2$, the blow-up transformation in \eqref{eq:blowup_transform} reads
\[
	u_1 = r_2 u_{1, 2}, \quad v_1 = r_2^2 v_{1, 2}, \quad
	u_k = r_2 u_{k, 2}, \quad v_k = r_2^2 v_{k, 2},
	\quad\text{and}\quad \varepsilon = r_2^3;
\]
in particular, the variables $u_k$ and $v_k$ ($1\leq k\leq k_0$) are rescaled with powers of $r_2=\varepsilon^{1/3}$, which justifies the terminology.

Substitution of the above transformation into \eqref{eq:blowup_original}  and 
desingularisation with a factor of $r_2$ gives
\begin{subequations}%
	\label{eq:blowup_k2}
	\begin{align}
		u_{1, 2}' &= - v_{1, 2} + 2^{-1/2} u_{1, 2}^2 
			+ 2^{-1/2} \sum_{j = 2}^{k_0} u_{j, 2}^2 + H^u_{1, 2}, \\
		v_{1, 2}' &= - 2^{1/2}, \\
		u_{k, 2}' &= \frac{b_k}{4 A^2} u_{k, 2} - v_{k, 2} + 2^{1/2} u_{1, 2} u_{k, 2}
			+ \sum_{i, j = 2}^{k_0} \eta^k_{i, j} u_{j, 2} u_{i, 2} + H^u_{k, 2}, \label{eq:blowup_k2:uk2} \\
		v_{k, 2}' &= \frac{b_k}{4 A^2} r_2^3 v_{k, 2} + H^v_{k, 2}, \label{eq:blowup_k2:vk2}\\
		r_2' &= 0
	\end{align}
\end{subequations}
for $2\leq k \leq k_0$, where 
\begin{align*}
	H^u_{1, 2} &= \calO \left( r_2 \right), \\
	H^u_{k, 2} &= \calO \left( r_2 u_{1, 2} v_{k, 2}, r_2 u_{k, 2} v_{1, 2},
		r_2 u_{i, 2} u_{j, 2}, r_2 u_{1, 2}^2 u_{k, 1}, r_2 u_{1, 2} u_{i, 2} u_{j, 2},
		r_2 u_{i, 2} u_{j, 2} u_{l, 2} \right),\quad\text{and} \\
	H^v_{k, 2} &= \calO \left(
	r_2^4 v_{1, 2} v_{k, 2}, r_2^4 v_{i, 2} v_{j, 2} \right),
\end{align*}
with $2 \le i, j, l \le k_0$.


The plane $\{ u_{k, 2} = 0 = v_{k, 2} : 2 \le k \le k_0 \} \cap \{r_2=0\}$ is
invariant under the flow of Equation~\eqref{eq:blowup_k2}; on that plane, \eqref{eq:blowup_k2} reduces to
\begin{subequations}%
	\label{Riccati_1}
	\begin{align}
		u_{1, 2}' &= - v_{1,2} + 2^{-1/2} u^2_{1, 2}, \\
		v_{1, 2}' &= - 2^{1/2},
	\end{align}
\end{subequations}
with $u_{1, 2}, v_{1, 2} \in \bbR$, which is a Riccati equation that corresponds to the one found
in~\cite[Proposition 2.3]{KruSzm3}, up to a rescaling. Correspondingly, we have the following result.
\begin{prop}%
	\label{prop:riccati}
	The Riccati equation in~\eqref{Riccati_1} has the following properties:
	\begin{enumerate}
		\item Every orbit has a horizontal asymptote $v_{1, 2} = v^{\infty}_{1, 2}$, where 
			$v^{\infty}_{1, 2}$ depends on the orbit, such that $u_{1, 2} \rightarrow\infty$ as
			$v_{1, 2}$ approaches $v^{\infty}_{1, 2}$ from above.
		\item There exists a unique orbit $\gamma_2$ which can be parametrised as $(u_{1, 2},
			s(u_{1, 2}))$, with $u_{1, 2} \in \bbR$, which is asymptotic to the left branch of the parabola 
			$\{-v_{1, 2} + 2^{-1/2} u_{1, 2}^2 = 0\}$ for $u_{1, 2} \rightarrow -\infty$. The orbit
			$\gamma_2$ has a horizontal asymptote $v_{1, 2} = - \Omega_0 < 0$ such that 
			$u_{1, 2} \rightarrow \infty$ as $v_{1, 2}$ approaches $-\Omega_0$ from above, where $\Omega_0$ is a positive constant that is defined as in \cite{KruSzm3}.
		\item The function $s(u_{1, 2})$ has the asymptotic expansions
			\begin{equation}%
				\label{eq:gamma2_minus_inf}
				s(u_{1, 2}) = 2^{-1/2} u_{1, 2}^2 + \frac{2^{-1/2}}{u_{1, 2}} + \calO \left( \frac{1}{u_{1, 2}^4} \right)
				\quad \text{as } u_{1, 2} \rightarrow - \infty
			\end{equation}
			and 
			\begin{equation}%
				\label{eq:gamma2_plus_inf}
				s(u_{1, 2}) = - \Omega_0 + \frac{2^{1/2}}{u_{1, 2}} + \calO \left( \frac{1}{u_{1, 2}^3} \right)
				\quad \text{as } u_{1, 2} \rightarrow\infty.
			\end{equation}
		\item All orbits to the right of $\gamma_2$ are backward asymptotic to the right branch of
			the parabola $\{-v_{1, 2} + 2^{-1/2} u_{1, 2}^2 = 0\}$.
		\item All orbits to the left of $\gamma_2$ have a horizontal asymptote $v_{1, 2} = v_{1,
			2}^{-\infty} > v_{1, 2}^{\infty}$, where $v_{1, 2}^{-\infty}$ depends on the orbit,
			such that $u_{1, 2} \rightarrow -\infty$ as $v_{1, 2}$ approaches $v_{1, 2}^{-\infty}$
			from above.
	\end{enumerate}
\end{prop}

If we transform the orbit $\gamma_2$ to chart $K_1$, we find that 
\begin{equation}%
	\label{eq:gamma-1-def}
	\gamma_1:=\kappa_{12}^{-1}(\gamma_2) = \left\{\left( u_{1, 2} {s(u_{1, 2})}^{-1/2}, 0, 
		\bzero, \bzero, {s(u_{1, 2})}^{-3/2} 
		\right) \right\},
\end{equation}
where $\bzero$ denotes the zero vector in $\mathbb{R}^{k_0-1}$.

In fact, expanding~\eqref{eq:gamma-1-def} in a power series as $u_{1, 2} \rightarrow - \infty$, we obtain
\begin{equation}
	\gamma_1
		= \left\{\left( -2^{1/4} + \frac{2^{-3/4}}{u_{1, 2}^3} + \calO \left(\frac{1}{u_{1, 2}^6} \right), 
			0, \bzero, \bzero, - \frac{2^{-3/4}}{u_{1, 2}^3} + \calO \left( \frac{1}{u_{1, 2}^6} \right)
			\right) \right\},
\end{equation}
which implies that $\gamma_1$ approaches the steady state $p^{k_0}_a$ in chart $K_1$,
tangent to the vector $(-1, 0, \bzero, \bzero, 1)$. 

Similarly, for $u_{1, 2} > 0$, we can transform $\gamma_2$ to the coordinates in chart $K_3$ via
\begin{equation}
	\begin{aligned}
		\gamma_3 &:=\kappa_{23}(\gamma_2) 
			= \left\{ \left( 0, u_{1, 2}^{-2} s(u_{1, 2}), 0, 0, u_{1, 2}^{-3} \right) \right\} \\  
			&= \left\{ \left(0, - \frac{\Omega_0}{u_{1, 2}^2} + \frac{2^{1/2}}{u_{1, 2}^3} 
				+ \calO \left( \frac{1}{u_{1, 2}^5} \right) , \bzero , \bzero,
				\frac{1}{u_{1, 2}^{3}}\right) \right\} \\
			&= \left\{ \left( 0, -\Omega_0 \varepsilon_3^{2/3} + 2^{1/2} \varepsilon_3 + \calO
			\big(\varepsilon_3^{5/3}\big) \right), \bzero, \bzero, \varepsilon_3 \right\},
	\end{aligned}
\end{equation}
which shows that, as $u_{1, 2} \rightarrow \infty$  or, equivalently, as $\varepsilon_3 \rightarrow
0$, $\gamma_3$ approaches the origin in chart $K_3$ tangent to the vector $(0, 1, \bzero, \bzero, 0)$.

To determine the transition map for chart $K_2$, we first transform the exit section $\SigmaOutOne$
from chart $K_1$ to the coordinates of $K_2$, applying the change of coordinates $\kappa_{12}$ in~\eqref{eq:change_K1_K2}, which will yield an entry section $\SigmaInTwo$
for the flow in $K_2$:
\[
	\SigmaInTwo := \big\{ (u_{1, 2}, v_{1, 2}, u_{k, 2}, v_{k, 2}, r_2)
	: v_{1, 2} = \delta^{-2/3} \big\}.
\]
In addition, the orbit $\gamma_2$ intersects that section in a single point $q_0$, so that 
\begin{equation}%
	\label{eq:q0-def}
	\gamma_2 \cap \SigmaInTwo = \{ q_0 \}.
\end{equation}
The coordinates of $q_0$ satisfy $u_{k, 2} = 0 = v_{k, 2}$ and $r_2 = 0$. 
We also define the exit section  
\begin{equation}
	\SigmaOutTwo := \big\{ (u_{1, 2}, v_{1, 2}, u_{k, 2}, v_{k, 2}, r_2) : u_{1, 2} = \delta^{-1/3} \big\}.
\end{equation}
The resulting geometry is illustrated in~\autoref{fig:chart-K2-1}.
To define the transition map $\Pi_2$ in $K_2$, we consider initial conditions in a small neighbourhood $R_2$ around the point
$q_0$.
\begin{lemma}%
	\label{lemma:K2-uk-stability}
	The invariant set $\{ u_{k, 2} = 0 = v_{k, 2} : 2 \le k \le k_0 \}$ is linearly stable under the flow of \eqref{eq:blowup_k2} if 
	\begin{equation}%
		\label{eq:delta-upper-lower}
		\frac{8^3 a^6}{\pi^6} \varepsilon_0 < \delta < \frac{\varepsilon_0}{\rho^3}. 
	\end{equation}
\end{lemma}
\begin{proof}
	Differentiation of~\eqref{eq:blowup_k2:uk2} with respect to $u_{k, 2}$ shows that for linear stability, we require 
	\begin{equation}%
		\label{eq:k2_condition_on_stability}
		\frac{b_k}{4 A^2} + 2^{1/2} u_{1, 2}(t) < 0 \quad\text{or, more strongly,} \quad 
		u_{1, 2}(t) < \frac{\pi^2}{8 A^2}
	\end{equation}
	for $2 \le k \le k_0$, as $b_k$ is negative and decreasing with
	$k$; recall~\eqref{eq:def-bk}. Given that $u_{1, 2}(T_2) = \delta^{-1/3}$ in $\SigmaOutTwo$, where $T_2$ denotes the (finite) transition time of the orbit $\gamma_2$ between $\SigmaInTwo$ and $\SigmaOutTwo$, 
    it is sufficient to have
	\begin{equation}%
		\label{linear_stab_u12}
		\delta^{-1/3} < \frac{\pi^2}{8 a^2 \varepsilon^{1/3}},
	\end{equation}
	where we have made use of $A = a \varepsilon^{1/6}$. We can simplify the last inequality to 
	\begin{equation}%
		\delta > \frac{8^3  a^6}{\pi^6} \varepsilon;
	\end{equation}
	moreover, since $\varepsilon\in[0,\varepsilon_0)$, it is sufficient to assume 
	\begin{equation}%
		\label{eq:delta-lower}
		\delta > \frac{8^3 a^6}{\pi^6} \varepsilon_0,
	\end{equation}
	which places a lower bound on $\delta$. 
	Finally, the upper bound in the statement of the lemma follows from the definition of $\delta$
	in chart $K_1$.
\end{proof}

\begin{remark}
	The linear stability condition in \eqref{eq:delta-upper-lower} can be satisfied by restricting
	$\varepsilon_0$ on the left-hand-side of the condition so that $\delta$ can be chosen
	sufficiently small for the analysis in charts $K_1$ and $K_3$ to hold, and by then choosing
	$\rho$ small enough for the upper bound on the right-hand side to be satisfied.
\end{remark}

\begin{prop}%
	\label{prop:K2_transition_map}
	The transition map $\Pi_2 : \SigmaInTwo \rightarrow \SigmaOutTwo$ is well-defined in
	a neighbourhood of the point $q_0$, see~\autoref{fig:chart-K2-1}, which maps
	diffeomorphically to a neighbourhood of  $\Pi_2(q_0)$, where
	\[
		\Pi_2(q_0) = \left( \delta^{-1/3}, -\Omega_0 + 2^{1/2} \delta^{1/3} +
		\calO(\delta), \bzero, \bzero, 0 \right). 
	\]
	Moreover, $|u_{k, 2}|$ and $|v_{k, 2}|$ are non-increasing under $\Pi_2$.
\end{prop}
\begin{proof}
	Given~\autoref{lemma:K2-uk-stability}, the system in~\eqref{eq:blowup_k2} can be considered as a regular
	perturbation of the Riccati equation~\eqref{Riccati_1} in a sufficiently small neighbourhood of $q_0$. Then, the  assertions of the proposition follow from~\autoref{prop:riccati} and regular perturbation theory~\cite{Verhulst, MisRoz}.
\end{proof}
\begin{figure}
	\centering
	\begin{overpic}{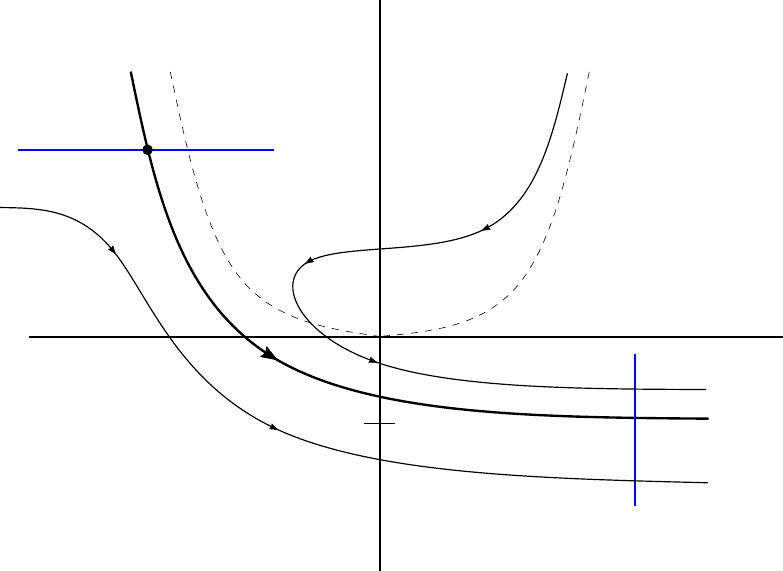}
		\put(98, 27){$u_{1, 2}$}	
		\put(50, 72){$v_{1, 2}$}	
		\put(18, 42){\scalebox{1.5}{$\gamma_2$}}
		\put(16, 51){$q_0$}
		\put(0, 49){\textcolor{blue}{\scalebox{1.5}{$\SigmaInTwo$}}}
		\put(82, 15){\textcolor{blue}{\scalebox{1.5}{$\SigmaOutTwo$}}}
		\put(76, 58){$v_{1, 2} = 2^{-1/2} u_{1, 2}^2$}
		\put(41, 18){$-\Omega_0$}
	\end{overpic}

	\caption{The dynamics in chart $K_2$ on the invariant plane $\{u_{k, 2} = 0 = v_{k, 2}\}\cap
	\{r_2 = 0\}$. For suitably chosen initial conditions and sufficiently small
	$r_2=\varepsilon^{1/3}$, the general dynamics of \eqref{eq:blowup_k2} is a regular perturbation
	of the dynamics on that plane.}%
	\label{fig:chart-K2-1}
\end{figure}

We can also derive estimates on the higher-order modes $\{ u_{k, 2}, v_{k, 2} \}$ during the
transition from $\SigmaInTwo$ to $\SigmaOutTwo$. Consider a point 
\[
    q_1 = \left( u_{1, 2}(0), \delta^{-2/3}, u_{k, 2}(0), v_{k, 2}(0), r_2(0) \right),
\] with $2 \le k \le k_0$, close to $q_0$.
Since orbits of the full system, Equation~\eqref{eq:blowup_k2}, are regular perturbations of the
orbit $\gamma_2$, the transition time $T_2(q_1)$ for the orbit initiated in $q_1$ will be
equal, to leading order, to $T_2$, the transition time for $\gamma_2$,
\begin{equation}%
    \label{eq:K2-transition-time-q1}
	T_2(q_1) = T_2 + \calO \left( u_{k, 2}(0), v_{k, 2}(0), r_2(0) \right).
\end{equation}
The lower bound on $\delta$ in \eqref{eq:delta-upper-lower} then yields the following
estimates on $u_{k,2}$ and $v_{k,2}$.
\begin{lemma}%
	\label{lemma:estimates-K2}
    For any $t \in \left[0, T_2(q_1)\right]$, the following estimates hold:
	\begin{align}%
        \label{eq:K2-estimate-uk}%
	    |u_{k,2}(t)| &\leq \exp(\frac{b_k}{16a^2 r_2(0)} t) |u_{k,2}(0)| + 
		\frac{16a^2 r_2(0)}{|b_k|}  \Big[|v_{k,2}(0)|+\Big(1+ \frac{ 4 a^2 r_2(0)^2 }{|b_k|}\Big)  \sigma \Big]\quad\text{and}  \\
		\label{eq:K2-estimate-vk}
		|v_{k,2}(t)| &\le \exp(\frac{b_k}{4a^2} r_2^2(0) t) |v_{k,2}(0)|+  
		\frac{ 4 a^2 r_2(0)^2 \sigma}{|b_k|},
	\end{align}
	for some constant $C>0$ and $
	0<\kappa \leq \sigma <1$, where $\kappa$ is as in \eqref{eq:initial-data-L2}. 
\end{lemma}
\begin{proof}
For $(\tilde u_{1,2}, \ldots, \tilde u_{k_0,2}, \tilde v_{1,2}, \ldots \tilde v_{k_0,2})$ in 
\[
	\begin{multlined}
	\calB_2 =  \bigg\{(\tilde u_{1, 2}, \dots, \tilde u_{k_0, 2} , \tilde v_{1, 2}, \dots, \tilde v_{k_0, 2} )
		 :   \tilde u_{k, 2}, \tilde v_{k,2} \in  C[0, T_2(q_1)], 
			\text{ with }  \\	
		 \sup_{[0,T_2(q_1)]}|\tilde u_{k, 2}(t)| \le  C_{k}, \; 
			  \sup_{[0,T_2(q_1)]}|\tilde v_{k, 2}(t)| \le  C_k \text{ for } 1 \le k \le k_0,  
			   \text{ and } \sum_{k=1}^{k_0} C_k^2 \leq \tilde\sigma \bigg\},
	\end{multlined}
\]
consider $M_k(\tilde u_{2,2}, \ldots \tilde u_{k_0, 2}) = \sum_{i,j=2}^{k_0} \eta_{i,j}^k \tilde
u_{j,2}\tilde u_{1,2}$ and the higher-order terms $H_{k,2}^u= H_{k,2}^u(\tilde u_{1,2}, \ldots,
\tilde u_{k_0,2}, \tilde v_{1,2}, \ldots \tilde v_{k_0,2})$ and $H_{k,2}^v= H_{k,2}^v(\tilde
u_{1,2}, \ldots, \tilde u_{k_0,2}, \tilde v_{1,2}, \ldots \tilde v_{k_0,2})$.  Thus, we define a
map $\mathcal N_2$ via $(\tilde u_{1,2}, \ldots, \tilde u_{k_0,2}, \tilde v_{1,2}, \ldots
\tilde v_{k_0,2}) \mapsto (u_{1,2}, \ldots,  u_{k_0,2},  v_{1,2}, \ldots v_{k_0,2})$, where $(u_{1,2},
\ldots,  u_{k_0,2},  v_{1,2}, \ldots v_{k_0,2})$ are solutions of \eqref{eq:blowup_k2}.  To
obtain the estimates stated in the lemma, we shall show that $\mathcal N_2 : \mathcal B_2 \to \mathcal B_2$.
As $r_2(t)$ is constant in chart $K_2$, from \eqref{eq:blowup_k2:vk2} we conclude
\begin{equation*}
\begin{aligned}
	|v_{k,2}(t)| &\le \exp(\frac{b_k}{4A^2} r_2^3(0) t) |v_{k,2}(0)| + 
	\exp(\frac{b_k}{4A^2} r_2^3(0) t)\int_0^t \exp(-\frac{b_k}{4A^2} r_2^3(0) s) |H_{k,2}^v| ds\\
	&\le \exp(\frac{b_k}{4A^2} r_2^3(0) t) |v_{k,2}(0)| +  
	\frac{4 A^2 r_2(0)}{|b_k|}\sigma
\end{aligned}
\end{equation*}
for all $t \in \left[ 0, T_2(q_1) \right]$, where 
$|H^v_{k,2}|\leq C r_2(0)^4 \tilde \sigma \leq r_2(0)^4 \sigma$.
Applying the variation of constants formula to \eqref{eq:blowup_k2:uk2}, we find
\begin{equation*}
	\begin{aligned}
		|u_{k,2}(t)| &\le \exp\bigg(\frac{(2-\sqrt{2})b_k}{8A^2} t\bigg) |u_{k,2}(0)| 
	+ \frac{8A^2}{(2-\sqrt{2})|b_k|} \bigg( \sup_{t \in \left[ 0, T_2(q_1) \right]}
		|v_{k,2}(t)| + C \tilde\sigma \bigg) \\
	&\leq  \exp(\frac{b_k}{16A^2} t) |u_{k,2}(0)|\
		+ \frac{16 A^2}{|b_k|} \bigg[ |v_{k, 2}(0)| + \Big(1+ \frac{ 4 a^2 r_2(0)^2 }{|b_k|}\Big)\sigma \bigg]
	\end{aligned}
\end{equation*}
for all $t \in \left[ 0, T_2(q_1) \right]$. Thus, for appropriately chosen $0<r_2(0)<1$ and
$0<\sigma<1$, we obtain that $\mathcal N_2: \mathcal B_2 \to \mathcal B_2$, as claimed, which
implies \eqref{eq:K2-estimate-uk} and \eqref{eq:K2-estimate-vk}. 
\end{proof}
\begin{remark}
Since $A = a \varepsilon^{1/6} = a {\left( r_2(0) \right)}^{1/2}$, the first term
in~\eqref{eq:K2-estimate-uk} is equal to $\exp\big(-\frac{c}{r_2(0)}\big) |u_{k, 2}|$ at $t =
T_2(q_1)$, with $c > 0$ a constant, while the second term has the form of an $\calO \left(
r_2(0) \right)$-correction.

The estimate for $v_{k, 2}(t)$ in~\eqref{eq:K2-estimate-vk} implies the bound $\exp \left( -c r_2^2(0) \right)|v_{k, 2}(0)| \approx |v_{k, 2}(0)|$ at $t = T_2(q_1)$ for small
$r_2(0)$, as considered here.

Taking more general higher-order terms of the form $H^v=H^v(u^2, uv, v^2)$ in
\eqref{eq:initial_problem}, we would find $H^v_{k,2}$ to be of the order $\calO(r_2(0)^2)$; then,  
the second term in \eqref{eq:K2-estimate-vk} would read $4a^2
\sigma/|b_k|$, which is uniformly bounded in $r_2(0)$ and $k$. 
\end{remark}

%
%
%
%


\subsection{Chart $K_3$}

In chart $K_3$, the blow-up transformation in \eqref{eq:blowup_transform} reads
\[
	u_1 = r_3, \quad v_1 = r_3^2 v_{1, 3}, \quad
	u_k = r_3 u_{k, 3}, \quad v_k = r_3^2 v_{k, 3}, \quad\text{and}\quad
	\varepsilon = r_3^3 \varepsilon_3.
\]
After desingularising by dividing out a factor of $r_3$ from the resulting vector field, we obtain
\begin{subequations}%
	\label{eq:blowup_k3}
	\begin{align}
		r_3' &= F_3 r_3, \\
		v_{1, 3}' &= -2 F_3 v_{1, 3} - 2^{1/2} \varepsilon_3, \\
		u_{k, 3}' &= \left( - F_3 + \frac{b_k}{4 A^2} \varepsilon_3^{1/3} + 2^{1/2} \right) u_{k, 3}
			- v_{k, 3} + \sum_{i, j = 2}^{k_0} \eta^k_{i, j} u_{i, 3} u_{j, 3} 
			+ H^u_{k, 3}, \\
		v_{k, 3}' &= \left( -2 F_3 + \frac{b_k}{4 A^2} r_3^3 \varepsilon_3^{4/3} \right) v_{k, 3}
			+ \varepsilon_3 H^v_{k, 3}, \\
		\label{eq:blowup_k3:eps_3}
		\varepsilon_3' &= -3 F_3 \varepsilon_3,
	\end{align}
\end{subequations}
where
\[
	F_3 = F_3(r_3, v_{1, 3}, u_{k, 3}, v_{k, 3}, \varepsilon_3) = - v_{1, 3} + 2^{-1/2} + 2^{-1/2} \sum_{j = 2}^{k_0} u_{j, 3}^2 + H^u_{1,3},
\]
with
\begin{align*}
	H^u_{1, 3} &= \calO \left( r_3 \varepsilon_3, r_3^2 v_{1, 3}^2, r_3^2 v_{j, 3}^2, r_3 v_{1,3},
		r_3 u_{j, 3} v_{j, 3}, r_3 u_{j, 3}^2, r_3 u_{i, 3} u_{j, 3} u_{l, 3} \right), \\
	H^u_{k, 3} &= \calO \left( r_3^2 v_{1, 3} v_{k, 3}, r_3^2 v_{i, 3} v_{j, 3}, r_3 v_{k, 3},\right. \\
 &\qquad\qquad \left.
		r_3 u_{k, 3} v_{1, 3}, r_3 u_{i, 3} v_{j, 3}, r_3 u_{k, 3}, r_3 u_{i, 3} u_{j, 3}, 
		r_3 u_{i, 3} u_{j, 3} u_{l, 3} \right),\quad\text{and} \\
	H^v_{k, 3} &= \calO \left(
	r_3^4 v_{1, 3} v_{k, 3}, r_3^4 v_{i, 3} v_{j, 3} \right)
\end{align*}
for $2 \le i, j, l \le k_0$.

As in $K_1$, we can rewrite \eqref{eq:blowup_k3} in the form
\begin{subequations}%
	\label{eq:blowup_k3_mu}
	\begin{align}
		r_3' &= F_3 r_3, \\
		v_{1, 3}' &= -2 F_3 v_{1, 3} - 2^{1/2} \varepsilon_3, \\
		u_{k, 3}' &= \left( - F_3 + \frac{b_k}{4 A^2} \varepsilon_3^{1/3} + 2^{1/2} \right) u_{k, 3}
			- v_{k, 3} + \sum_{i, j = 2}^{k_0} \eta^k_{i, j} u_{i, 3} u_{j, 3} 
			+ H^u_{k, 3}, \\
		v_{k, 3}' &= \left( -2 F_3 + \frac{b_k}{4 A^2} r_3^3 \big(\varepsilon_3^{1/3}\big)^4 \right) v_{k, 3}
			+ \varepsilon_3 H^v_{k, 3}, \\
		\label{eq:blowup_k3_mu:mu3}
		\big( \varepsilon_3^{1/3} \big)'  &= - F_3 \varepsilon_3^{1/3}
	\end{align}
\end{subequations}
for $2\leq k \leq k_0$.

As mentioned already, the portion $\gamma_3 := \kappa_{23}(\gamma_2)$ of the orbit $\gamma_2$ from
chart $K_2$ with $u_{1, 2} > 0$, transformed to $K_3$, has the expansion
\[
	\gamma_3=\left( 0, - \Omega_0 \varepsilon_3^{2/3} + 2^{1/2} \varepsilon_3 + \calO \big( \varepsilon_3^{5/3}
	\big), \bzero, \bzero, \varepsilon_3 \right) 
\]
as $\varepsilon_3 \rightarrow 0$. Thus, we see that $\gamma_3$ approaches the origin in chart $K_3$.
Hence, it follows that the centre manifold $M_{k_0, 1}$ from chart $K_1$ passes through a neighbourhood of the origin,
which is a hyperbolic steady state for~\eqref{eq:blowup_k3_mu}.

Let $w_k := (0, 0, \dots, 1, \dots, 0)$, with $2 \le k \le k_0$, denote the vector with $k_0 - 1$
entries which are all equal to $0$ except for the $(k - 1)$-th entry, which equals $1$. With that notation, a direct calculation shows the following result.
\begin{lemma}%
	\label{lemma:K3-eigenvalues-1}
	The origin is a hyperbolic steady state of Equation~\eqref{eq:blowup_k3_mu}, with the following eigenvalues and eigenvectors in the corresponding linearisation:
	\begin{itemize}
		\item the simple eigenvalue $\frac{\sqrt{2}}{2}$ with eigenvector $(1, 0, \bzero, \bzero, 0)$, corresponding to $r_3$;
		\item the simple eigenvalue $-\sqrt{2}$ with eigenvector $(0, 1, \bzero, \bzero, 0)$, corresponding to $v_{1, 3}$;
		\item the eigenvalue $\frac{\sqrt{2}}{2}$ with multiplicity $k_0-1$ and eigenvectors $(0, 0, w_k, \bzero, 0)$, corresponding to $u_{k, 3}$ ($2 \le k \le k_0$);
		\item the eigenvalue  $-\sqrt{2}$ with multiplicity $k_0-1$ and eigenvectors $(0, 0, \frac{\sqrt{2}}{3} w_k, w_k, 0)$, corresponding to $v_{k, 3}$ ($2 \le k \le k_0$); and
		\item the simple eigenvalue $-\frac{\sqrt{2}}{2}$ with eigenvector $(0, 0, \bzero, \bzero, 1)$, corresponding to
			$\varepsilon_3^{1/3}$. 
	\end{itemize}
\end{lemma}
\begin{remark}
	Since 
	\[
		- \frac{\sqrt{2}}{2} = - \sqrt{2} + \frac{\sqrt{2}}{2},
	\]
	the eigenvalues of \eqref{eq:blowup_k3_mu} are in resonance. Potential second-order resonant
	terms are $r_3 v_{1, 3}$, $r_3 v_{k, 3}$, and $u_{i, 3} v_{j, 3}$.  While resonances are also
	observed in the singularly perturbed planar fold~\cite{KruSzm3}, the resonant terms differ,
	which is due to the  formulation of the governing equations in chart $K_3$ in terms of
	$\varepsilon_3^{1/3}$. Furthermore, the higher dimensionality of \eqref{eq:blowup_k3_mu}
    allows for a richer resonance structure which may be explored in future work.
\end{remark}

The entry section $\SigmaInThree$ in chart $K_3$, which is obtained by transformation of the exit section $\SigmaOutTwo$ from $K_2$, is given by
\begin{equation}
	\SigmaInThree = \{ (r_3, v_{1, 3}, u_{k, 3}, v_{k, 3}, \mu_3) : \varepsilon_3 = \delta \}, 
\end{equation}
where we consider the set of initial conditions
\begin{multline}%
	\label{eq:k3-R3-definition}
	R_3 = \{(r_3, v_{1, 3}, u_{k, 3}, v_{k, 3}, \varepsilon_3)\; | \; r_3 \in [0, \rho], v_{1,3} \in
		[-\beta, \beta], \\ 
	|u_{k, 3}| \le C_{u_{k, 3}}, |v_{k, 3}| \le C_{v_{k, 3}} \text{ for } 2 \le k \le k_0, \text{ and }\varepsilon_3 = \delta \} \subset \SigmaInThree.
\end{multline}
Here, $\beta, \Cin_{u_{k, 3}}$, and $\Cin_{v_{k, 3}}$, for $2 \le k \le k_0$, are appropriately defined small constants. We also introduce
the exit chart
\begin{equation}
	\SigmaOutThree := \left\{ (r_3, v_{1, 3}, u_{k, 3}, v_{k, 3}, \varepsilon_3)
		: r_3 = \rho \right\}.
\end{equation}
Our aim is to describe the transition map $\Pi_3 : R_3 \rightarrow \SigmaOutThree$. Therefore, since 
$F_3$ is bounded away from zero near the origin, we can divide the vector field in~\eqref{eq:blowup_k3_mu}
by $F_3$, which results in 
\begin{subequations}%
	\label{eq:k3_F3_divided}
	\begin{align}
		r_3' &= r_3, \\
		v_1' &= -2 v_1 - 2^{1/2} \frac{\big(\varepsilon_3^{1/3}\big)^3}{F_3}, \\
		u_k' &= \bigg( -1 + \frac{b_k}{4 A^2} \frac{\varepsilon_3^{1/3}}{F_3} + \frac{2^{1/2}}{F_3} \bigg) u_k 
			- \frac{1}{F_3} v_k + \frac{1}{F_3} \sum_{i, j = 2}^{k_0} 
				\eta_{i, j}^k u_i u_j + \frac{1}{F_3} H^u_k, \\
		v_k' &= \bigg( -2 + \frac{b_k}{4 A^2} \frac{r_3^3 \big(\varepsilon_3^{1/3}\big)^4}{F_3} \bigg) v_k 
			+ \frac{\big(\varepsilon_3^{1/3}\big)^3}{F_3} H^v_k, \\
		\big( \varepsilon_3^{1/3} \big)'  &= - \varepsilon_3^{1/3},
	\end{align}
\end{subequations}
where the prime denotes differentiation with respect to the new, rescaled, time variable.
Here, we have suppressed the subscript $3$ in \eqref{eq:k3_F3_divided} for convenience of notation, and will do so for the remainder
of the section.

The above rescaling of time by $F_3$ results in the eigenvalues of the linearisation about the origin being rescaled by a factor $2^{-1/2}$. \autoref{lemma:K3-eigenvalues-1} hence now implies the following:
\begin{lemma}%
	\label{lemma:K3-eigenvalues-2}
	The origin is a hyperbolic steady state of Equation~\eqref{eq:k3_F3_divided}, with the following eigenvalues in the corresponding linearisation:
	\begin{itemize}
		\item the simple eigenvalue $1$, corresponding to $r_3$;
		\item the simple eigenvalue $-2$, corresponding to $v_{1}$;
		\item the eigenvalue $1$ with multiplicity $k_0-1$, corresponding to $u_{k}$ ($2 \le k \le k_0$);
		\item the eigenvalue $-2$ with multiplicity $k_0-1$, corresponding to $v_{k}$ ($2\le k \le k_0$); and
		\item the simple eigenvalue $-1$, corresponding to $\varepsilon_3^{1/3}$. 
	\end{itemize}
	The associated eigenvectors are as given in~\autoref{lemma:K3-eigenvalues-1}.
\end{lemma}

To obtain estimates for the transition map $\Pi_3$, we follow a procedure that is analogous to that
in~\cite{KruSzm3} for chart $K_3$. We begin by separating out terms containing $r_3$
in~\eqref{eq:k3_F3_divided}. To that end, we expand
\begin{equation}%
	\label{eq:k3-F3-expansion}
		\frac{1}{F_3(v_1, u_k, r_3)} = G_3(v_1,  u_k) + r_3 J(v_1, u_k, r_3), 
\end{equation}
in a neighbourhood of the steady state at the origin, where 
\begin{equation}
	G_3(v_1, u_k) = \frac{1}{2^{-1/2} -v_1 + 2^{-1/2} \sum_{j = 2}^{k_0}{u_j^2}}
\end{equation}
and $J$ is a smooth function of $v_1, u_k, r_3$ in the same neighborhood.
With the above notation, we can rewrite Equation~\eqref{eq:k3_F3_divided} as stated below.
\begin{lemma}
	For $r_3 \ge 0$ sufficiently small, \eqref{eq:k3_F3_divided} can be written as
	\begin{subequations}%
		\label{eq:k3_F3_divided_rewrite}
		\begin{align}
			r_3' &= r_3, \\
			v_1' &=-2 v_1 - 2^{1/2} \varepsilon_3 G_3 + \varepsilon_3 r_3 J_{v_1}, \\
			u_k' &= \left( -1 + \frac{b_k}{4 A^2} \varepsilon_3^{1/3} G_3 + 2^{1/2} G_3 \right) u_k - G_3 v_k
					+ G_3 \sum_{i, j = 2}^{k_0} \eta_{i, j}^k u_i u_j + r_3 J_{u_k}, \\
			v_k' &= \Big(-2 + \frac{b_k}{4 A^2} \big(\varepsilon_3^{1/3}\big)^4 r_3^3 G_3\Big) v_k + \varepsilon_3 r_3 J_{v_k}, \label{eq:k3_F3_divided_rewrite_vk}\\
			\big( \varepsilon_3^{1/3} \big)' &= - \varepsilon_3^{1/3},
		\end{align}
	\end{subequations}
	where $J_{v_1}(r_3, v_1, u_k, v_k, \varepsilon_3^{1/3})$, $J_{u_k}(r_3, v_1, u_k, v_k,
	\varepsilon_3^{1/3})$, and $J_{v_k}(r_3, v_1, u_k, v_k, \varepsilon_3^{1/3})$ are smooth
	functions.
\end{lemma}
\begin{proof}
	Using~\eqref{eq:k3-F3-expansion} in~\eqref{eq:k3_F3_divided} and collecting $r_3$-dependent terms, we obtain~\eqref{eq:k3_F3_divided_rewrite}. The functions $J_{v_1}$, $J_{u_k}$, and $J_{v_k}$
	are defined by summation and multiplication between the variables in chart $K_3$ and the functions $H^u_{k}$, $H^v_{k}$, and $J$, and are
	hence smooth in their arguments in the neighbourhood of the origin we are considering.
\end{proof}


We now have the following result for the transition map $\Pi_3$.
\begin{prop}%
	\label{prop:K3:transition}
	The transition map $\Pi_3: R_3\rightarrow \SigmaOutThree$ is well-defined. Let $(\rin_3,
	\vonein, \ukin, \vkin, \delta)\in R_3$, as defined in~\eqref{eq:k3-R3-definition}, where
	$k=2,\dots, k_0$, and let $T_3$ be the corresponding transition time between $\SigmaInThree$ and
	$\SigmaOutThree$ under the flow of~\eqref{eq:k3_F3_divided_rewrite}. Then, the map $\Pi_3$ is
	given by
	\[
		\Pi_3(\rin_3, \vonein, \ukin, \vkin, \delta) 
			= \left(\rho, \Pi^{k_0}_{3, v_1}, \Pi^{k_0}_{3, u_k}, \Pi^{k_0}_{3, v_k}, 
				\delta^{1/3} \frac{\rin_3}{\rho} \right),
	\]
	where
	\begin{align}
		\left\vert \Pi_{3, v_1}(\rin_3, \vonein, \ukin, \vkin, \delta) \right\vert
			&\le {\left(\frac{\rin_3}{\rho}\right)}^2 \left[|\vonein| + \Cout_{v_{1,3}}(1 + \rin_3 \log \rin_3)\right], \\
		\left\vert \Pi_{3, u_k}(\rin_3, \vonein, \ukin, \vkin, \delta) \right\vert	
			&\le \Cout_{u_{k, 3}},\quad\text{and} \\
		\left\vert \Pi_{3, v_k}(\rin_3, \vonein, \ukin, \vkin, \delta) \right\vert
			&\le \Cout_{v_{k, 3}},
	\end{align}
	for positive constants $\Cout_{v_{1,3}}$, $\Cout_{u_{k, 3}}$, and $\Cout_{v_{k, 3}}$.
\end{prop}
\begin{proof}
From~\eqref{eq:k3_F3_divided_rewrite}, we have that 
\begin{equation}%
	\label{eq:K3-r-mu-exact}
	r_3(t) = \rin_3 e^{t} \quad \text{and} \quad \varepsilon_3(t) = \delta e^{-3t},
\end{equation}
which gives the transition time 
\begin{equation}
	T_3 = \log \frac{\rho}{\rin_3}
\end{equation}
between $\SigmaInThree$ and $\SigmaOutThree$.
	
For $\sum_{k=2}^{k_0}|\tilde u_k(t)|^2 \leq  \sigma$ and $\sum_{k=2}^{k_0}|\tilde v_k(t)|^2 \leq \sigma$, with $0 < \sigma \leq 1$, 
and  $|v_1(t)|\leq 1/(2\sqrt{2})$ for $t \in [0,T_3]$, consider $J_{v_1} = J_{v_1}(\tilde v_1, \tilde u_k, \tilde v_k)$ and
$J_{l_k}=J_{l_k}(v_1, \tilde u_k, \tilde v_k)$, with $l=u,v$ and $k=2, \ldots, k_0$, as well as $G_3 = G_3(v_1, \tilde u_k)$.
We observe that 
	\begin{equation}\label{estim_G3_1}
		\begin{aligned}
		G_3(v_1(t), \tilde u_k(t)) & \le \frac{1}{2^{-1/2} - | v_1(t)|} = 2^{1/2} 
			+ 2^{1/2}\sum_{n = 1}^{\infty} 2^{\frac{n }{2}} |  v_1(t) |^n \\
			& \leq 2^{1/2} 
			\left( 1+ 2^{1/2}\frac {| v_1(t)|}{ 1- \sqrt{2} | v_1(t)|}\right) \leq 2^{1/2} \big( 1+ 2^{3/2}| v_1(t)|\big)
		\end{aligned}
	\end{equation}
for $|v_1(t)|\leq 1/(2\sqrt{2})$,	and 
	\begin{equation}
		\begin{aligned}
			G_3(v_1(t), \tilde u_k(t)) 
			&\ge \frac{1}{2^{-1/2} + |v_1(t)| + 2^{-1/2}\sigma}=: C_1\geq 1.
		\end{aligned}
	\end{equation}
%
Then, using the boundedness of $G_3$ and $J_{v_k}$, from Equation~\eqref{eq:k3_F3_divided_rewrite_vk} for $v_k$
we obtain directly 
\[
	|v_k(t)|\leq e^{-2t + B_3(t)} |v_k(0)| + C\int_0^t e^{-2(t-s)+ B_3(t) - B_3(s) } \varepsilon_3(s) r_3(s)ds 
		\leq e^{-2t} |v_k(0)| + \frac{C \rho a^2 \sigma}{|b_k|}, 
\]
where $B_3(t)= \rin_3 \delta^{4/3} b_k(1- e^{-3t})/(12A^2)$. 
To determine the asymptotic behaviour of $v_1$, we define a new variable~$z$ by
	\begin{equation}%
		\label{eq:k3_z_v1}
		v_1 = e^{-2t} \left( \vonein + z \right), 
	\end{equation}
	where  for $t = 0$ it follows that $z(0) = 0$.  A direct calculation yields
	\begin{align*}
		v_1' &= e^{-2t} z' - 2 e^{-2t} \left( \vonein + z \right), \\
		-2 v_1 - \delta e^{-3t} G_3 + \rin_3 \delta e^{-2t} J_{v_1} &= e^{-2t} z' 
			- 2 e^{-2t} \left( \vonein + z \right),\quad\text{and} \\
		- 2 e^{-2t} \left( \vonein + z \right) + \delta e^{-3t} G_3 + \rin_3 \delta e^{-2t} J_{v_1} &= 
			e^{-2t} z' - 2 e^{-2t} \left( \vonein + z \right)
	\end{align*}
	or, equivalently,
	\begin{equation}
		z' = -2^{1/2} e^{-t} \delta G_3 + \rin_3 \delta J_{v_1}.
	\end{equation}
	Then, the boundedness of $G_3$ and $J_{v_1}$ implies  
	\begin{equation*}
		|z(t)| \leq C \delta(1 - e^{-t} + \rin_3 t).
	\end{equation*}
	Reverting to the original variable $v_1$ via~\eqref{eq:k3_z_v1}, we find 
	\[
		|v_1(t)|\leq  e^{-2t}\big[|\vonein| + C\delta
		(1 - e^{-t} + \rin_3 t)\big]\le   |\vin_1| + 2 \delta C_2 \leq 1/(2\sqrt{2}) \quad \text{for all } t \in [0,	T_3]
	\]
	with sufficiently small $|\vin_1|$,	and 
	\begin{equation}
		|v_1 (T_3)| \leq 
		{\left(\frac{\rin_3}{\rho}\right)}^2 
		\left[|\vonein| + C\delta\left(1  
			+ \rin_3 \log \frac{\rho}{\rin_3}\right) \right],
	\end{equation}
	which proves the first estimate  stated in the theorem.

	Next, we show that $u_{k}$ remains bounded throughout the transition through chart $K_3$. Once again, we perform an
	estimate using the variation of constants formula
	\begin{multline*}
		u_{k}(t) = \exp\left(\int_0^t U_k(\tau) d \tau\right) \ukin \\
			+ \int_0^t \exp\left(\int_s^t U_k(\tau) d \tau\right) 
				\left( - G_3(v_1, \tilde u_k) v_k + G_3(v_1, \tilde u_k) \sum_{i,j=2}^{k_0}\eta^k_{i, j} \tilde u_i \tilde u_j + r_3 J_{u_k} \right) ds,
	\end{multline*}
	where
	\[
		\begin{aligned}
			U_k(\tau) 
			&= - 1 + \frac{b_k}{4A^2}\delta^{1/3}e^{-\tau} G_3(v_1(\tau), \tilde u_k(\tau)) + 2^{1/2} G_3( v_1(\tau), \tilde u_k(\tau)).
		\end{aligned}
	\]
	Using our assumptions on $\tilde u_k$ and $\tilde v_k$, the estimate for $G_3(v_1, \tilde u_k)$ in
	\eqref{estim_G3_1}, and the fact that $b_k < 0$, we find
	\begin{equation}%
		\label{eq:K3-est-I1}
		\begin{aligned}
			\calI_1 (s)&:= \int_s^t U_k(\tau) d\tau \\
			&= s - t + \frac{b_k}{4A^2}\delta^{1/3} \int_s^t e^{-\tau} G_3(v_1(\tau), \tilde u_k(\tau)) d\tau
				+ 2^{1/2} \int_s^t G_3(v_1(\tau), \tilde u_k(\tau)) d\tau \\
			&\le t - s + \frac{b_k}{4A^2}\delta^{1/3} C_1 \left( e^{-s} - e^{-t} \right) 
				+ 4 \int_s^t |v_1(\tau)| d \tau 
		\end{aligned}
	\end{equation}
	for $0 \le s \le t \le T_3$. To estimate the integral in the last inequality, we observe that
	\begin{equation*}
		\left\vert v_1(\tau) \right\vert \le e^{-2\tau} 
			\big[|\vin_1| + C_2\delta (1  + \tau)\big]
	\end{equation*}
	and write
	\begin{equation*}
		\begin{aligned} 
		\int_s^t |v_1(\tau)| d\tau \leq \frac 14(2|\vin_1| +  3 C_2 \delta) (e^{-2s} - e^{-2t}) + C_2 \frac \delta 2 (e^{-2s}s - e^{-2t}t). 
		\end{aligned}
	\end{equation*}
	The inequality in~\eqref{eq:K3-est-I1} then becomes
	\begin{equation*}
		\calI_1(s) \le t - s + \frac{b_k}{4A^2}\delta^{1/3} C_1 \left( e^{-s} - e^{-t} \right) 
		+(2|\vin_1| +  3 C_2 \delta) (e^{-2s} - e^{-2t}) + 2 C_2 \delta (e^{-2s}s - e^{-2t}t);
	\end{equation*}
thus,
\begin{equation*}
    \exp(\calI_1(0)) \leq   C \exp\Big(t+\frac{b_k}{4 A^2}\delta^{1/3} C_1\big(1 - e^{-t}\big)\Big),
\end{equation*}
where $C= \exp(2|\vin_1| +  5 C_2 \delta)$. For the second term in $u_k(t)$, using that $t\leq e^t$
and $1\leq e^t$ for $t\geq 0$, we have 
\begin{align*}
  &  \int_0^t \exp(\calI_1(s)) ds  \leq  \mathcal I_2(t) \int_0^t \exp\Big( - s + \frac{b_k}{4A^2}\delta^{1/3} C_1  e^{-s}  
+ 	C_3  e^{-2s}  + 2 C_2\delta  e^{-2s}s \Big) ds
\\
& \quad  \leq \mathcal I_2(t)
\int_0^t e^{-s}\exp\Big(\Big( \frac{b_k}{4A^2}\delta^{1/3} C_1    
+ 	C_4 \delta \Big)e^{-s} \Big) ds \\
& \quad = \mathcal I_2(t) \frac{ 4 A^2}{C_1|b_k| \delta^{1/3} - 4 A^2C_4} \Big[ \exp \Big(\Big(\frac{b_k \delta^{1/3}}{4A^2}C_1+ C_4\delta \Big)  e^{-t}\Big) -\exp \Big(\frac{b_k \delta^{1/3}}{4A^2}C_1+ C_4\Big) \Big]\\
& \quad \leq C_5\frac{ 4 A^2}{C_1|b_k| \delta^{1/3} - 4 A^2C_4} e^t
\leq C_6 \frac {a^2}{|b_k|} \varepsilon^{1/3} \frac {\rho}{\rin_3}
\leq \frac{C_7 \rho a^2}{|b_k|},    
\end{align*}
where $C_3 = 2|\vin_1| +  3 C_2 \delta $  and 
\[
	\mathcal I_2(t)=\exp\Big(t  - \frac{b_k}{4A^2}\delta^{1/3} C_1   e^{-t}  
	- C_3  e^{-2t} - 2 C_2 \delta e^{-2t}t\Big).
\]
The estimates for $v_1$ and $v_k$ then yield 
\[
	|u_k(t)| \leq C_1|u_k(0)| + \frac{C_2 \rho a^2}{|b_k|}\Big[|v_k(0)|+\frac{C \rho a^2
	\sigma}{|b_k|}+ \sigma (1+ \rho)\Big]. 
\]
Hence, for sufficiently small $|u_k(0)|$, $|v_k(0)|$, and $\rho$, we find 
\[
	\sum_{k=2}^{k_0} |u_k(t)|^2 \leq \sigma\quad\text{and}\quad \sum_{k=2}^{k_0} |v_k(t)|^2 \leq
	\sigma\quad\text{ for all } t \in [0, T_3]. 
\]
Then, application of a fixed point argument as in chart $K_1$ yields the estimates stated in the theorem. 
%
\end{proof}


\begin{figure}
	\centering
	\begin{overpic}[scale=0.21]{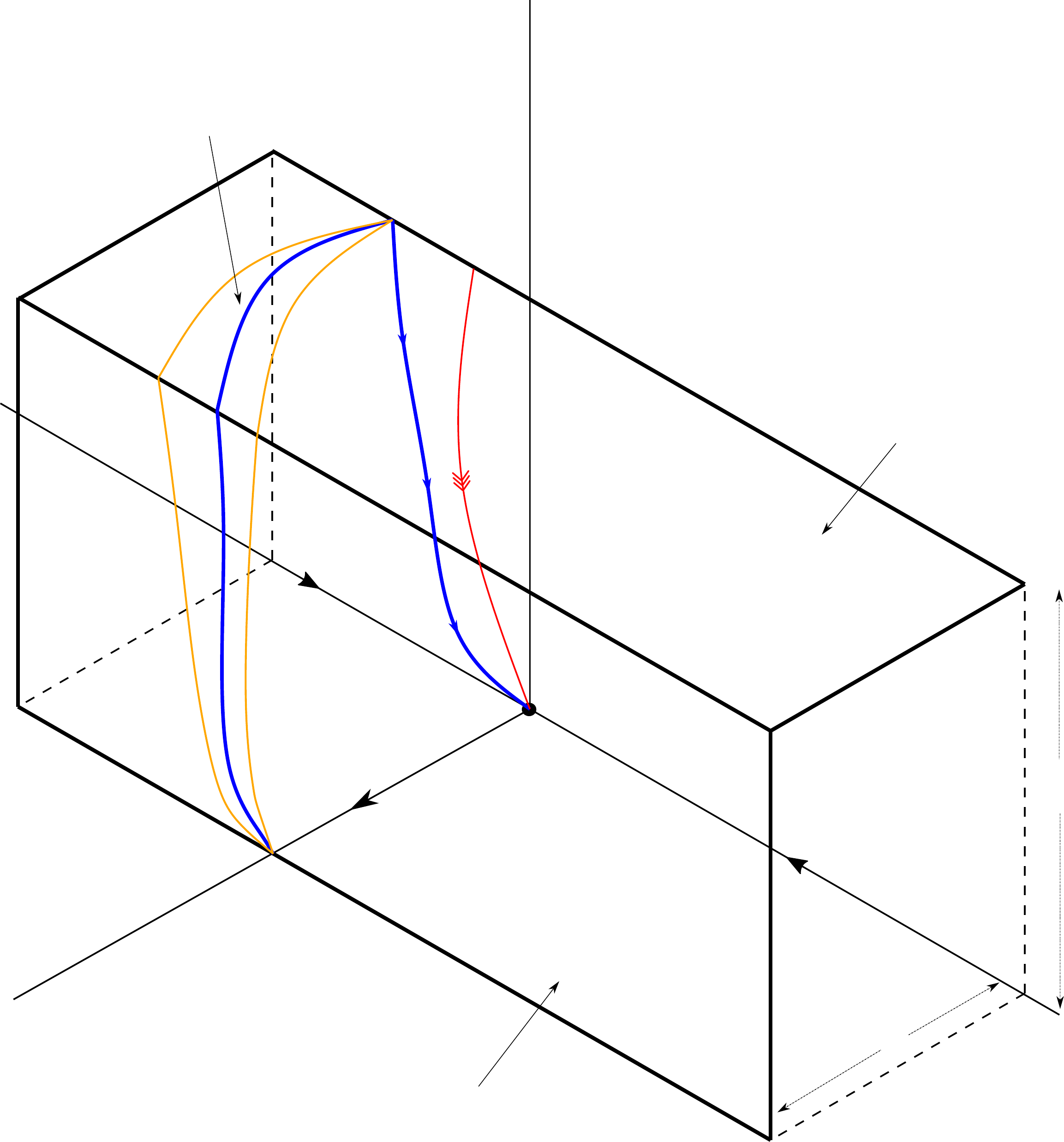}
		\put(46, 33){$q$}	
		\put(91.5, 29.5){$\delta$}
		\put(77.1, 7.5){$\rho$}
		\put(79, 62){$\SigmaInThree$}
		\put(0, 10){$r_3, u_{k, 3}$}
		\put(48, 100){$\varepsilon_3$}
		\put(16, 90){$M_{k_0, 3}$}
		\put(38, 0){$\SigmaOutThree$}
		\put(88, 7){$v_{1, 3}, v_{k, 3}$}
	\end{overpic}
	\caption{Dynamics in chart $K_3$. As the orbit $\gamma_2$ from chart $K_2$, after
	transformation to $K_2$ via $\gamma_3 := \kappa_{23}(\gamma_2)$ (in red), passes through the
	origin $q$, the invariant manifold $M_{k_0, 3}$ contains $q$. The transition map $\Pi_3$ is
	defined in a neighbourhood of the intersection $M_{k_0, 3} \cap \SigmaInThree$.}%
	\label{fig:chart3-1}
\end{figure}
%
%
%
%


\subsection{Proof of main result}%
\label{subsec:transition_map}

Let us now combine the analysis in the three charts $K_1$, $K_2$, and $K_3$ to give the proof
of~\autoref{prop:transition_full}. For $2\leq k \leq k_0$, the initial conditions $(\uin_{1, 1},\rin_1,\uin_{k, 1},\vin_{k, 1},\epsin_1)$ in $K_1$ are assumed to lie in $R_1 \subset
\SigmaInOne$, where $R_1$ is defined  in \eqref{eq:def_R}. 
Applying the transition map $\Pi_1$, see \autoref{prop:K1:transition-map}, we obtain 
\begin{align*}
	\Pi_1 (R_1)= \Big\{ &|\uout_{1, 1} + 2^{1/4}| \le \Cout_{u_{1, 1}}, \
	\rout_1	\in [0, \rho], \ \epsout_1 = \delta, \\ &|\uout_{k, 1}| \le \Cout_{k, 1},\ \text{and}\
		|\vout_{k, 1}| \le \Cout_{v_{k, 1}} \delta^{2/3} \Big\}.
\end{align*}
Transformation of the above set to chart $K_2$ yields
\begin{align*}
	\kappa_{12} \circ \Pi_1 (R_1)= \Big\{ &\uin_{1, 2} = \delta^{-1/3} \uout_{1, 1}, 
		\vin_{1, 2} = \delta^{-2/3},\ |\uin_{k, 2}| \le \delta^{-1/3} |\uout_{k, 1}|,\\
	&|\vin_{k, 2}| \le \delta^{-2/3} |\vout_{k, 1}|,\ \text{and}\ \rin_2 \in [0, \delta^{1/3} \rho] \Big\},
\end{align*}
with $u_{k, 1}$ and $v_{k, 1}$ as above. Since the higher-order modes $\{u_{k, 2}, v_{k, 2}\}$  do not grow in $K_2$, we have 
\begin{equation} 
	\begin{aligned}
		\Pi_2 \circ \kappa_{12} \circ \Pi_1 (R_1) = \Big\{ &\uout_{1, 2} = \delta^{-1/3},\	|\vout_{1, 2} + c_1|\le \Cout_{v_{1,2}}, \\ 
		&|\uout_{k, 2}| \le |\uin_{k, 2}|,\ |\vout_{k, 2}| \le |\vin_{k, 2}|,\ \text{and}\
			\rout_2 \in \left[0, \delta^{1/3} \rho\right] \Big\}.
	\end{aligned}
\end{equation}
Application of the change of coordinates $\kappa_{23}$ yields 
\begin{align*}
	\kappa_{23} \circ \Pi_2 \circ \kappa_{12} \circ \Pi_1 (R) =
	\Big\{ &\rin_3 \in [0, \varepsilon],\ \vin_{1, 3} \in [-\beta, \beta],\ \epsin_3 = \delta,
		|\uin_{k, 3}| = \delta^{1/3}, \\
	&|\uout_{k, 2}| \le \Cin_{u_{k, 3}},\ 
			|\vin_{k, 3}|= \delta^{2/3},\ \text{and}\ |\vout_{k, 2}| \le \Cin_{v_{k, 1}} \Big\},
\end{align*}
where $\beta > 0$ is a small constant.
Finally, we apply the map $\Pi_3$, see~\autoref{prop:K3:transition}, to obtain 
\begin{align*}
	\Pi_3 \circ \kappa_{23} \circ \Pi_2 \circ \kappa_{12} \circ \Pi_1 (R_1)  
		=\Big\{ \rout_3 = \rho,\ \vout_{1, 3},\ \epsout_3 \in [0, \delta],\
			|\uout_{k, 3}| \le \Cout_{u_{k, 3}},\ \text{and}\ 
			|\vout_{k, 3}| \le \Cout_{v_{k, 3}} \Big\}, 
\end{align*}
where $\vout_{1, 3}$ is as in~\autoref{prop:K3:transition}. The result then follows, since
the sections $\SigmaInOne$ and $\SigmaOutThree$ are equivalent to $\DeltaIn$ and
$\DeltaOut$, respectively, transformed into the coordinates of charts $K_1$ and $K_3$, respectively,
and since the systems in~\eqref{eq:galerkin_system_temp} and~\eqref{eq:blowup_original} are
equivalent for $\varepsilon > 0$ sufficiently small. 
%
%
%
%

\section{Conclusions and outlook}\label{concl}

In this work, we have studied, via discretisation, a fast-slow system of partial differential
equations (PDEs) of reaction-diffusion type, Equation~\eqref{eq:initial_problem}, under the
assumption that a fold singularity is present at the origin in the fast kinetics. We have
approximated a family $S_{\varepsilon, \zeta}$ of slow manifolds by their corresponding Galerkin
manifolds $\calC_{\varepsilon, k_0}$, which we have then extended past the singularity by applying
the desingularisation technique known as blow-up \cite{DumortierRoussarie}. Here, it is worth
emphasising that the family $S_{\varepsilon, \zeta}$ is defined via a subsplitting ansatz for the
slow variables; hence, it is not simply a ``generic" perturbation of a classical critical manifold
which is obtained in a finite-dimensional setting by a ``quasi-steady state approximation".

As we have seen, our main result, \autoref{prop:transition_full}, is analogous to what one would
expect in the planar (finite-dimensional) setting \cite{KruSzm3}. While we have shown that the
resulting Galerkin manifolds $\calC_{\varepsilon, k_0}$ approximate the family $S_{\varepsilon,
\zeta}$ away from the fold singularity, a natural next question
concerns the passage to the limit as $k_0 \rightarrow \infty$ near the fold.  In
the normally hyperbolic regime, we know that the limit of $k_0 \rightarrow \infty$ implies a double
limit, with $\varepsilon\rightarrow 0$, as well as a specific scaling law~\cite{EngelHummelKuehn}. 
While it remains open whether that double singular limit is well-defined near non-normally
hyperbolic singularities, our results do lay relevant groundwork. Two conjectures seem plausible
here: (i) an alternative approach may allow one to prove that, potentially under slightly stronger
assumptions, a limiting invariant slow manifold exists uniformly in $\varepsilon\in
(0,\varepsilon_0)$, with $\varepsilon_0>0$ small, as $k_0\to\infty$; or (ii) the limiting object
must diverge for fixed positive $\varepsilon$ as $k_0\to \infty$ due to the coupling between
infinitely many higher-order modes in the Galerkin discretisation. Unfortunately, standard
techniques \cite{kostianko2021,romanov2000} for proving the non-existence of invariant manifolds do
not seem to allow us to verify (ii). However, on the basis of previous work in the normally
hyperbolic regime and the detailed estimates presented here, we can  conjecture that a well-defined
double limit of $\varepsilon \rightarrow 0$ and $k_0\rightarrow \infty$ will exist even near fold
singularities.

We further emphasise that the presence of an additional $2k_0 - 2$ equations after discretisation,
with $k_0$ arbitrarily large, causes several challenges. Thus, a preparatory rescaling of the domain
length is introduced to allow for the application of the blow-up technique; an alternative approach
in previous work on the transcritical and pitchfork singularities \cite{EngelKuehn1} results in a
dynamic boundary value problem. Our rescaling appears natural, since it can be recovered directly
from the original system of PDEs in
\eqref{eq:initial_problem}.
Specifically, taking $u = \varepsilon^{1/3} U$, $v = \varepsilon^{2/3} V$, $t = \varepsilon^{-1/3} \tau$, and 
$x = \varepsilon^{-1/6} X$, which is consistent with our scaling in \eqref{scalling_1}, we obtain 
\begin{equation} \label{rescal}
	\begin{aligned} 
		\partial_\tau U &= \partial_X^2 U - V + U^2 + \varepsilon^p H_u(U,V) && \text{ on } (-a \varepsilon^{1/6}, a \varepsilon^{1/6}),\\
		\partial_\tau V &= \varepsilon \partial_X^2 V - 1 + \varepsilon^q H_v(U,V) && \text{ on } (-a \varepsilon^{1/6}, a \varepsilon^{1/6}),
	\end{aligned} 
\end{equation}
for some $p, q >0$. 
Equation~\eqref{rescal} defines a system of PDEs on a domain
shrinking to the origin as $\varepsilon \to 0$, as is to be expected due to the singular nature of \eqref{eq:initial_problem}. Denoting by $(U_\varepsilon, V_\varepsilon)$
solutions of \eqref{rescal}, and using the boundedness of higher-order terms and the non-positivity
of $U_\varepsilon$ or the boundedness of $U^2_\varepsilon$, which can be achieved by considering a
cut-off function, we obtain the following estimates: 
\[
\begin{aligned}
	\|V_\varepsilon\|^2_{L^\infty(0,T; L^2(\Omega_\varepsilon))} + 
	\varepsilon \|\partial_x V_\varepsilon\|^2_{L^2(0,T; L^2(\Omega_\varepsilon))}
	& \leq C \big( \|V_\varepsilon(0)\|^2_{L^2(\Omega_\varepsilon)} + \varepsilon^{1/6}\big)\quad\text{and} \\
	\|U_\varepsilon\|^2_{L^\infty(0,T; L^2(\Omega_\varepsilon))} + 
	 \|\partial_x U_\varepsilon\|^2_{L^2(0,T; L^2(\Omega_\varepsilon))}
	& \leq C \big(\|U_\varepsilon(0)\|^2_{L^2(\Omega_\varepsilon)} 
		+ \|V_\varepsilon(0)\|^2_{L^2(\Omega_\varepsilon) } + \varepsilon^{1/6}\big),
\end{aligned}
\]
where $\Omega_\varepsilon = (-a \varepsilon^{1/6}, a \varepsilon^{1/6})$ and $C$ is some positive constant independent of $\varepsilon$. These estimates  imply that $U_\varepsilon(\cdot, \varepsilon^{1/6} \cdot)  \rightharpoonup U_0$ in $L^2(0,T; H^1(-a,a))$,
which is independent of $X$, and $V_\varepsilon(\cdot, \varepsilon^{1/6}\cdot) \rightharpoonup V_0$,  $\varepsilon^{1/2}\partial_X V_\varepsilon(\cdot, \varepsilon^{1/6}\cdot) \rightharpoonup W$ in $L^2((0,T)\times(-a,a))$  as $\varepsilon \to 0$, for some $W \in L^2((0,T)\times(-a,a))$. Thus, in
the limit as $\varepsilon \to 0$, we see that $(U_0, V_0)$ satisfies the system of ODEs
\begin{equation} \label{rescal_ODE}
	\begin{aligned} 
		\dv{U}{\tau} &= - V + U^2,\\
		\dv{V}{\tau} &= - 1.
	\end{aligned} 
\end{equation}
Equation~\eqref{rescal_ODE} is precisely the Riccati equation which lies at the heart of the dynamics in our rescaling chart $K_2$.

A consequence of our rescaling of the domain length is, however, that the original fast-slow
structure which is present in the discretised system, Equation~\eqref{eq:galerkin_system_temp}, does
not translate to the blow-up analysis in the three coordinate charts. In particular, there is no
longer a direct correspondence between singular objects in those charts and the layer and reduced
problems pre-blow-up. Since the corresponding flows in the two scalings are equivalent after
``blow-down", the loss of correspondence is merely of technical relevance: while it does entail that
the approach in \cite{wechselberger2012propos} does not apply to~\eqref{eq:blowup_original}, we do
not consider canard dynamics here, as is done there.

As elaborated in~\ref{subsec:example-k0-2}, an additional challenge arises due to the finite-time
blowup which can occur in \eqref{eq:galerkin_system_temp} and which is due to the presence of
additional slow variables $v_{k}$, $2 \le k \le k_0$, after Galerkin discretisation. To avoid
solutions blowing up before they enter a neighbourhood of the singularity at the origin, we defined
an $\varepsilon$-dependent set of initial values $\Rin(\varepsilon) \subset \DeltaIn$, which we
combined with careful estimates for the higher-order modes $\{u_k,v_k\}$ resulting from the
discretisation. We conjecture that this blowup is, in essence, caused by additional fold
singularities that can be reached before the principal singularity at the origin which has been our
focus here. In particular, a future research direction would be the desingularisation of larger
submanifolds where normal hyperbolicity is lost in the Galerkin discretisation; for example, one
could blow up the blue curve in~\autoref{fig:critical-manif-2} or the surface
in~\autoref{fig:critical-manifold-stability-3} in the cases where $k_0=2$ or $k_0=3$, respectively. 

Finally, we briefly place our work into the broader context of singular perturbation problems
arising in an infinite-dimensional context. Firstly, for fast-slow reaction-diffusion systems of the
form in~\eqref{eq:fsstandardPDE}, we have recently gained a better understanding of transcritical
points and generic fold points, including the results presented in this work. In finite dimensions,
such non-hyperbolic points are known to generate only a dichotomy of either fast jumps of
trajectories or an exchange of stability between slow manifolds. Yet, more
degenerate fold points, such as folded nodes or folded saddle-nodes, may generate extremely
complicated local dynamics, including oscillatory patterns, even in fast-slow systems of ODEs. That
classification is likely to become even more complex in the infinite-dimensional setting of (systems
of) PDEs. Secondly, systems of the form in~\eqref{eq:fsstandardPDE} represent one class of
interesting PDEs, where small perturbation parameters and singular limits occur. Other classes
involve fast reaction terms, small diffusion problems, or heterogeneous media with highly
oscillatory coefficients, which all commonly appear in the context of reaction-diffusion systems.
Once one goes beyond reaction-diffusion systems, there are vast classes of PDE-type singular
perturbation problems arising across the sciences. From a mathematical viewpoint, it is immediately
clear that, in any parametrised PDE model, one anticipates possible distinctions between normally
hyperbolic dynamics, where locally a good approximation is achieved by linearisation, and a loss of
normal hyperbolicity along submanifolds in parameter space. Therefore, there is a  need for
developing techniques to tackle a loss of normal hyperbolicity in (systems of) PDEs. Our work is but
one building block towards that  general effort. Last, but not least, we have not yet related
our theoretical approach via Galerkin discretisation with the performance of various numerical
methods for PDEs. We conjecture that there is a link between (a loss of) performance and the
presence of singularities, or non-hyperbolic points, in systems of nonlinear PDEs.

\section*{Acknowledgements}

C.K.\ thanks the VolkswagenStiftung for support via a Lichtenberg Professorship. C.K. and M.E. also thank the DFG
for support via the SFB/TR 109 ``Discretization in Geometry and Dynamics''.
M.E. further acknowledges the support of the DFG through Germany's Excellence Strategy -- The Berlin
Mathematics Research Center MATH+ (EXC-2046/1, project ID 390685689). F.H.~and C.K.~acknowledge
support of the EU within the
TiPES project funded by the European Union Horizon 2020 Research and Innovation
Programme under grant agreement 820970. 

T.Z.\ was supported by the Maxwell Institute Graduate School in Analysis and its
Applications, a Centre for Doctoral Training funded by the UK Engineering and Physical Sciences
Research Council under grant agreement EP/L016508/01, the Scottish Funding Council, Heriot-Watt University, and the
University of Edinburgh.

The authors are grateful to two anonymous reviewers for insightful suggestions and comments that greatly improved the original manuscript.
\appendix

\section{Uniform boundedness and convergence}\label{app:boundconv}

Under appropriate assumptions on initial data, it is possible to show the uniform boundedness of solutions to Equation~\eqref{eq:initial_problem} for sufficiently large times before those solutions reach the singularity at the origin. Uniform boundedness will then imply convergence of the Galerkin discretisation, as shown below. Hence, our finite-dimensional Galerkin manifolds can be interpreted as ``approximately invariant slow manifolds"; the accuracy of the resulting approximation will improve with increasing $k_0$. 

\subsection{Uniform boundedness of solutions}
For simplicity, we first consider the equations in \eqref{eq:initial_problem} without higher-order terms $H^u$ and $H^v$. The parabolic comparison principle for $\hat v_0 \geq v(0,x) \geq \tilde v_0 >0$,  with $x \in (-a,a)$, yields $\hat v \geq v(t,x) \geq \tilde v(t)$, where $\tilde v$ and $\hat v$ satisfy
\[
\begin{aligned} 
\dv{\tilde v}{t} & = - \varepsilon, \quad & \quad\text{with }\tilde v(0) = \tilde v_0, \\
\dv{\hat v}{t} & = - \varepsilon, \quad & \quad\text{with }\hat v(0) = \hat v_0
\end{aligned} 
\]
and, hence, 
\[
\tilde v(t) = \tilde v_0 - \varepsilon t \quad\text{and}\quad \hat v(t) = \hat v_0 - \varepsilon t, \quad \text{with }\tilde v(t) \geq 0\text{ for } t \leq \frac{\tilde v_0} \varepsilon.
\]
Then, for $\tilde u_0 \leq u(0,x) \leq \hat u_0<0$, with $x\in (-a,a)$, we again apply the parabolic comparison principle to obtain that $\tilde u(t) \leq u(t,x) \leq \hat u(t)$, where
\[
\begin{aligned}
 \dv{\hat u}{t} & = - \tilde v + \hat u^2, \quad 
& \quad\text{with }\hat u(0) = \hat u_0, \\
 \dv{\tilde u}{t} & = - \hat v + \tilde u^2, \quad
& \quad\text{with }\tilde u(0) = \tilde u_0. 
\end{aligned}
\]
For $t \leq \tilde v_0/(2 \varepsilon)$, we have $\tilde v \geq \tilde v_0/2$ and can hence estimate $\hat u(t) \leq \bar u(t)$, where 
\[
\dv{\bar u}{t} = - \frac{\tilde v_0}2 + \bar u^2, \quad\text{with } \bar u(0) = \hat u_0,
\]
and
\[
\bar u(t) = \sqrt{ \frac {\tilde v_0} 2} \frac{ \hat u_0 (1+ e^{\sqrt{2\tilde v_0} t}) - \sqrt{\tilde v_0/2}(e^{\sqrt{2\tilde v_0} t}-1)}{ - \hat u_0(e^{\sqrt{2\tilde v_0} t}-1) + \sqrt{\tilde v_0/2} (1+ e^{\sqrt{2\tilde v_0} t})},
\]
which is bounded for all $\tilde v_0>0$ and $t \leq \tilde v_0 /(2 \varepsilon)$. Similarly, we obtain that $\tilde u$ is also uniformly bounded for $\tilde v_0>0$ and $t \leq \tilde v_0 /(2 \varepsilon)$. In sum, we hence have 
\[
\min\big\{\tilde u_0, -\sqrt{\hat v_0}\big\} \leq u(t,x) \leq \max_{t\in [0,\tilde v_0 /(2 \varepsilon)]} \bar u(t) \quad \text{ for } \; 0 \leq t\leq \frac{\tilde v_0}{2 \varepsilon}\text{ and } x \in (-a,a).
\]
Thus, for all $v(0,x) \geq \tilde v_0 >0$ and $u(0,x) \leq \hat u_0<0$, we obtain that solutions of~\eqref{eq:initial_problem}, without higher-order terms $H^u$ and $H^v$, are uniformly bounded for $0\leq t \leq \tilde v_0 /(2 \varepsilon)$.

When considering higher-order terms of the form $H^u(u, v,\varepsilon) = \mathcal O(\varepsilon, uv, v^2, u^3)$ and $H^v(u,v,\varepsilon) = \mathcal O(v^2)$ in \eqref{eq:initial_problem}, for $|u|,|v|\leq 1$, we can assume 
\[
|H^u(u, v,\varepsilon)|\leq \kappa_u(\varepsilon + |uv| + |v|^2 + |u|^3) \quad\text{and}\quad |H^v(u, v,\varepsilon)|\leq \kappa_v |v|^2
\]
for some positive constants $\kappa_u$ and $\kappa_v$. To derive  estimates for the solutions of \eqref{eq:initial_problem}, we apply a fixed point argument: for given $(u^\ast, v^\ast)$ with 
\[
 |u^\ast| \leq \min\Big\{\frac 1 {4\kappa_u}, 1\Big\} \quad\text{and} \quad  0 < v^\ast \leq \min \Big\{ \frac 1 { 2\sqrt{\kappa_v}}, \frac 1 { 4 \kappa_u}, \frac 1{32 \kappa_u^2}, 1\Big\},
\]
we consider $H^v(u^\ast, v^\ast,\varepsilon)$ and $\widetilde H^u(u^\ast, u, v,\varepsilon)$, which is obtained from $H^u(u,v,\varepsilon)$ by replacing the terms of order $uv$ and $u^3$ by $u^\ast v$ and $u^\ast u^2$, respectively. The above assumptions on $u^\ast$ and $v^\ast$ yield
\[
\begin{aligned}
|\widetilde H^u(u^\ast, u, v, \varepsilon)|& \leq \kappa_u (\varepsilon + |u^\ast v|+ v^2 + |u^\ast| u^2) \leq \kappa_u \varepsilon + \frac{|v|}2  + \frac{u^2}4 \quad\text{and} \\
|H^v(u^\ast, v^\ast,\varepsilon)| &\leq \kappa_v |v^\ast|^2 \leq \frac 14,
\end{aligned}
\] 
which ensures 
\[
\begin{aligned}
\dv{\tilde v}{t} & = \varepsilon (-1+ H^v(u^\ast,v^\ast,\varepsilon)) \geq - \frac 5 4 \varepsilon \quad\text{and} 
\\
\dv{\hat v}{t} &= \varepsilon (-1+ H^v(u^\ast,v^\ast,\varepsilon)) \leq - \frac {3\varepsilon} 4.
\end{aligned} 
\]
Then, for initial conditions satisfying
\begin{equation}\label{init_assum_estim}
 -\min \Big\{ \frac 1{4 \kappa_u},1 \Big\} \leq \tilde u_0 < 0 \quad\text{and}\quad \frac{4} 5 \hat v_0  \leq \tilde v_0 \leq \hat v_0 \leq  \min\Big\{\frac  1{2\sqrt{ \kappa_v}}, \frac 1 { 4 \kappa_u}, \frac 1{32 \kappa_u^2}, 1\Big \} 
\end{equation}
and for $0 \leq t \leq 4\tilde v_0/ (5\varepsilon)$, we have 
\[
0 \leq \tilde v_0 - \frac {5\varepsilon} 4   t \leq  \tilde v(t) \leq v(t,x) \leq  \hat v(t) \leq \hat v_0 - \frac{3 \varepsilon} 4 t \leq \min\Big\{\frac  1{2\sqrt{\kappa_v}}, \frac 1 { 4 \kappa_u}, \frac 1 {32 \kappa_u^2}, 1\Big \}. 
\]
For $\hat u$, we obtain
\[
\dv{\hat u}{t} = - \tilde v + \hat u^2 + |\widetilde H^u(u^\ast, \hat u, \hat v, \varepsilon)| \leq - \tilde v_0 + \frac{5 \varepsilon} 4 t +  \frac{\hat v_0}2 - \frac{3 \varepsilon} 8 t  + \kappa_u \varepsilon  + \frac 3 4  \hat u^2, 
\]
which, for $\varepsilon \leq \tilde v_0/(16 \kappa_u)$ and $t \leq 2\tilde v_0/ (7 \varepsilon)$, implies
\[
\dv{\hat u}{t}  \leq - \frac{\tilde v_0}{16} + \frac  5 4 \hat u^2.
\]
In combination with the previous estimates, the fixed point argument  yields uniform boundedness of solutions to \eqref{eq:initial_problem} for $t \leq 2\tilde v_0/ (7 \varepsilon)$ and initial conditions satisfying \eqref{init_assum_estim}. 

\subsection{Convergence of Galerkin discretisation}
The Galerkin approximation $(u_n, v_n)$ for solutions to  \eqref{eq:initial_problem}, with $u_n(t,x) = \sum_{k=1}^n u_k(t) e_k(x)$ and $v_n(t,x) = \sum_{k=1}^n v_k(t) e_k(x)$, satisfies
\begin{equation} \label{Galerkin_approx_PDE}
\begin{aligned}
 \partial_t u_n &= \partial_x^2 u_n - v_n + u_n^2 + H^u (u_n, v_n,\varepsilon) && \text{for }x\in (-a,a)\text{ and }t >0,  \\
 \partial_t v_n &= \varepsilon(\partial_x^2 v_n -1 + H^v(u_n, v_n,\varepsilon)) && \text{for }x\in (-a,a)\text{ and }t >0,\\
&\partial_x u_n(t,x) = 0 = \partial_x v_n(t,x) && \text{for } x = \pm a\text{ and }t>0, \\
& u_n(0,x)=u_{n,0}(x)\quad\text{and}\quad v_n(0,x)=v_{n,0}(x) && \text{for } x\in(-a,a),
\end{aligned}
\end{equation}
where $u_{n,0}$ and $v_{n,0}$ are projections of $u_0$ and $v_0$, respectively, onto the space $V = {\rm span}\{ e_1(x), \ldots, e_n(x)\}$.
Using similar estimates as above and imposing the assumptions on initial conditions in \eqref{init_assum_estim}, we obtain that $u_n(t,x)$ and $v_n(t,x)$ are uniformly bounded in $[0,T]\times [-a, a]$ for $\tilde u_0 \leq u_{n,0}(x)\leq \hat u_0 <0$, $\hat v_0 \geq v_{n,0}(x)\geq \tilde v_0 >0$, and $T \leq 2\tilde v_{0}/ (7 \varepsilon)$.
It follows that we have the a~priori estimates
\[
\begin{aligned}
 \| u_n\|^2_{L^\infty((0,T)\times(-a, a))} + \|u_n\|^2_{L^2(0,T; H^1(-a,a))} + \|\partial_t u_n\|^2_{L^2(0,T; H^{1}(-a,a)^\prime)} & \leq C \quad\text{and} \\
 \| v_n\|^2_{L^\infty((0,T)\times(-a, a))} + \varepsilon\|v_n\|^2_{L^2(0,T; H^1(-a,a))}+ \|\partial_t v_n\|^2_{L^2(0,T; H^{1}(-a,a)^\prime)} &\leq C,
\end{aligned}
\]
with a constant $C>0$ that is independent of $n$, which ensures convergence of $u_n \to u$ weakly in $L^2(0,T; H^1(-a,a))$ and strongly in $L^2((0,T)\times (-a,a))$, as well as of $v_n \to v$ weakly-$\ast$ in $L^\infty(0,T; L^2(-a, a))$ and of $\sqrt{\varepsilon} v_n \to \sqrt{\varepsilon} v$ weakly in $L^2(0,T; H^1(-a,a))$ and strongly in  $L^2((0,T)\times (-a,a))$; see e.g.~\cite{Evans}. Thus, we can pass to the limit as $n \to \infty$ in \eqref{Galerkin_approx_PDE} to conclude that $u$ and $v$ are solutions to the original system in~\eqref{eq:initial_problem}.

Next, considering equations for the differences $u_n - u$ and $v_n-v$ and taking $u_n-u$ and $v_n - v$ as test functions, respectively, we obtain
\[
\begin{aligned}
 \frac 12 \partial_t \|u_n - u\|^2_{L^2(-a,a)} + \|\partial_x(u_n - u)\|^2_{L^2(-a, a)} \leq \frac 12 \|v_n - v\|^2_{L^2(-a, a)} + \frac 12 \|u_n - u\|^2_{L^2(-a, a)} \\
 + \|u_n+u\|_{L^\infty}\|u_n - u\|^2_{L^2(-a, a)}+
h^u(\|u_n\|_{L^\infty},\|u\|_{L^\infty})|u_n - u\|^2_{L^2(-a, a)}\quad\text{and} \\
  \frac 12 \partial_t \|v_n - v\|^2_{L^2(-a, a)} + \varepsilon \|\partial_x(v_n - v)\|^2_{L^2(-a, a)} \leq
 \varepsilon h^v(\|u_n\|_{L^\infty},\|u\|_{L^\infty})|v_n - v\|^2_{L^2(-a, a)},
\end{aligned}
\]
with some smooth functions $h^u$ and $h^v$ representing contributions from the higher-order terms $H^u$ and $H^v$. 
Adding both inequalities, using the uniform boundedness of $u_n$, $u$, $v_n$, and $v$, and applying the Gr\o nwall inequality, we obtain
\[
\begin{aligned} 
& \sup_{(0,T)}  \|u_n - u\|^2_{L^2(-a,a)} + \|\partial_x(u_n - u)\|^2_{L^2((0,T)\times(-a,a))} \\
 & +\sup_{(0,T)}\|v_n - v\|^2_{L^2(-a,a)} + \varepsilon \|\partial_x(v_n - v)\|^2_{L^2((0,T)\times (-a,a))} \\
& \qquad \leq C(T) \big[\|u_n(0)- u(0)\|^2_{L^2(-a,a)} + \|v_n(0)- v(0)\|^2_{L^2(-a,a)} \big],
\end{aligned} 
\]
which ensures the convergence of the Galerkin truncation to the solution of the original problem, Equation~\eqref{eq:initial_problem}, as the approximation of the initial data converges strongly in $L^2(-a,a)$.

\section{Illustrative example: $k_0 = 2$}%
\label{subsec:example-k0-2}

In order to develop intuition for the singular geometry and resulting dynamics of
\eqref{eq:galerkin_system_temp}, it is instructive to examine the simple case where $k_0 = 2$.
For simplicity, let $a = \frac{1}{2}$, and assume that the higher-order terms $H^u_i$ and $H^v_i$ for $i=1,2$
are identically zero. In that case, the system in \eqref{eq:galerkin_system_temp} reads
\begin{subequations}%
	\label{eq:galerkin_system_two}
	\begin{align}
		 u_1^\prime &= -v_1 + 2^{-1/2} u_1^2 + 2^{-1/2} u_2^2, \\
		 v_1^\prime &= -2^{-1/2} \varepsilon, \\
		 u_2^\prime &= -\pi^2 u_2 - v_2 + 2^{1/2} u_1 u_2, \\
		 v_2^\prime &= -\pi^2 \varepsilon v_2,
	\end{align}
\end{subequations}
where the critical manifold $\calC$ is given by the graph
\begin{subequations}%
	\label{eq:critical_manif_two}
	\begin{align}
		v_1 = f_1(u_1, u_2) &:= 2^{-1/2} u_1^2 + 2^{-1/2} u_2^2\quad\text{and} \\
		v_2 = f_2(u_1, u_2) &:= -\pi^2 u_2 + 2^{1/2} u_1 u_2.
	\end{align}
\end{subequations}
Linearisation of the layer problem induced by \eqref{eq:critical_manif_two} for $\varepsilon=0$ about $\calC$ reveals that one eigenvalue is always
negative for any choice of $(u_1, u_2)$, whereas the sign of the other eigenvalue depends on $(u_1, u_2)$, as shown
in~\autoref{fig:critical-manif-2}. The set $\calC_0$, as defined in~\eqref{eq:C_0-definition}, is denoted in red there. To the left of the curve
$u_1 = g(u_2) := \frac{1}{2} \big( \pi^2 - \sqrt{\pi^2 + 4 u_2^2} \big)$ (illustrated in blue), the second eigenvalue
is negative, whereas it is positive to the right of the curve. Normal hyperbolicity is lost on the curve itself. 

\begin{remark}
Similarly, one can visualise the stability properties of the critical manifold $\calC$ in the case
where $k_0 = 3$, which will be given as a graph over $(u_1, u_2, u_3)$;
see~\autoref{fig:critical-manifold-stability-3}. Specifically, the manifold $\calC$ is then
attracting inside the funnel-like region of $(u_1, u_2, u_3)$-space shown in the figure and of
saddle type outside that region. In analogy to the case of $k_0=2$, normal hyperbolicity is lost
on the surface separating those two regions which is now given by an implicit polynomial
expression that can be obtained by application of the Routh-Hurwitz stability criterion. The set $\calC_0$
is again drawn in red. 
\end{remark}

For the particular case when $k_0=2$, it is possible to find explicit formulae for the
initial conditions which will allow us to reach the section $\DeltaIn$ under the flow of~\eqref{eq:galerkin_system_two}. Firstly, $(u_1, u_2)$ must be in the region of the $(u_1,u_2)$-plane
that corresponds to the normally hyperbolic attracting part of the critical manifold $\calC$, see
\autoref{fig:critical-manif-2}. Secondly, by GSPT, we have to be sufficiently close to the
corresponding slow manifold $\calC_\varepsilon$ for $\varepsilon$ sufficiently small, which amounts to a
condition of the form
\begin{equation}%
	\label{eq:condition_v_u}
	\max\{|v_1 - f_1(u_1, u_2)|, |v_2 - f_2(u_1, u_2)|\} < C,
\end{equation}
where $C > 0$ is some suitably chosen constant. 
Thirdly, we also need to impose corresponding restrictions on $(v_1, v_2)$ to ensure that we will
not reach an unstable part of the critical manifold $\calC$ under the slow flow induced by \eqref{eq:galerkin_system_two}. To that end, we first
need to invert the line $u_1 = g(u_2)$, which separates the attracting and saddle-like parts of $\calC$
in the $(u_1,u_2)$-plane, by substituting into \eqref{eq:critical_manif_two} and solving for $v_1$
and $v_2$. The result is now a curve in the $(v_1,v_2)$-plane of the form $v_1 = h(v_2)$, where the
function $h$ is a quadratic polynomial in $v_2$, as shown in blue
in~\autoref{fig:critical-manif-v-2}.
A further, fourth, restriction is given by solving explicitly the $(v_1,v_2)$-subsystem
in~\eqref{eq:galerkin_system_two}, rewritten in terms of the slow time, and by then determining
a relation between $v_1$ and $v_2$ so that the flow reaches the section $\DeltaIn$:
\begin{equation}%
	\label{eq:condition_v1_v2_2} |v_2| \le \xi (v_1) := \Cin_{v_2}
	e^{\frac{\pi^2}{\sqrt{2}} \left( v_1 - \rho^2 \right)};
\end{equation} 
see \autoref{fig:critical-manif-v-2}, in purple, for an illustration.
Initial values for~\eqref{eq:galerkin_system_two} satisfying these four conditions will flow into
the section $\DeltaIn$. 

However, as the flow of \eqref{eq:galerkin_system_two} approaches the
singularity at the origin, for some initial conditions $u_2$ may blow up before $v_1$ becomes negative. The flow will hence have
reached an unstable part of the critical manifold $\calC$. In the next section, we provide an explicit explanation for this blowup in finite time. 

\begin{figure}
	\centering
	\begin{overpic}[scale=0.60]{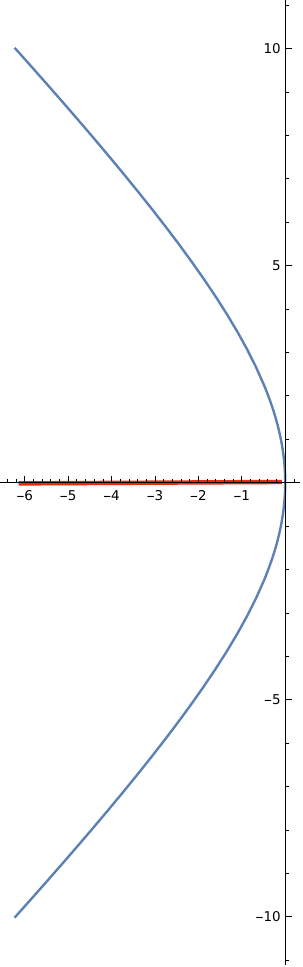}
		\put(32, 49){$u_1$}	
		\put(25, 100){$u_2$}
		\put(0, 60){stable}
		\put(15, 85){saddle}
		\put(5, 95){$u_1 = g(u_2)$}
		\put(0, 45){$\calC_0$}
		\put(15, 15){saddle}
	\end{overpic}
	\caption{Stability properties of the critical manifold $\calC$ which, for $k_0=2$, can be written as
	a graph over $(u_1, u_2)$. A loss of normal hyperbolicity occurs along the curve $u_1 = g(u_2) $ (in blue)
	where one of the two eigenvalues of the linearisation about $\calC$ changes sign. The manifold $\calC_0$ is shown in red.}%
	\label{fig:critical-manif-2}
\end{figure}
\begin{figure}
	\centering
	\begin{overpic}[scale=0.8]{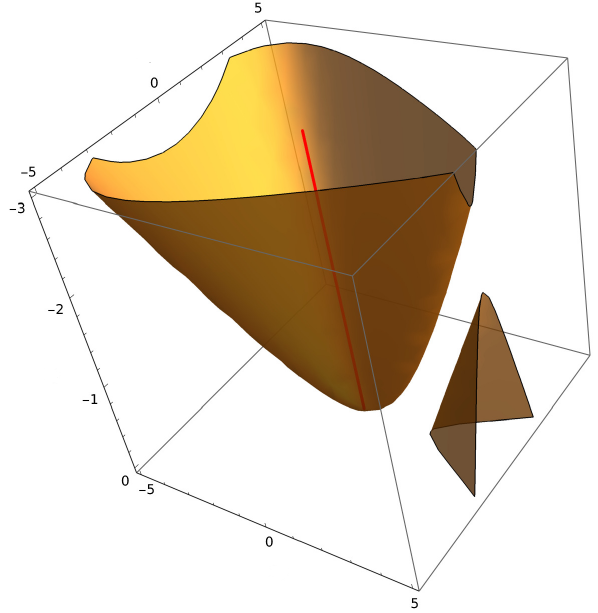}
		\put(0, 40){$u_1$}	
		\put(20, 90){$u_2$}	
		\put(40, 3){$u_3$}
	\end{overpic}

	\caption{The fold surface for $k_0 = 3$. The critical manifold $\calC$ can be written as a
	graph over $(u_1, u_2, u_3)$ and is stable inside the funnel-like region, the boundary of which
	is a surface that is implicitly defined by a polynomial equation in $u_1$, $u_2$, and $u_3$. One
	of the three eigenvalues of the linearisation about $\calC$ changes sign across the surface.}
	\label{fig:critical-manifold-stability-3}
\end{figure}
\begin{figure}
	\centering
	\begin{overpic}[scale=0.7, percent]{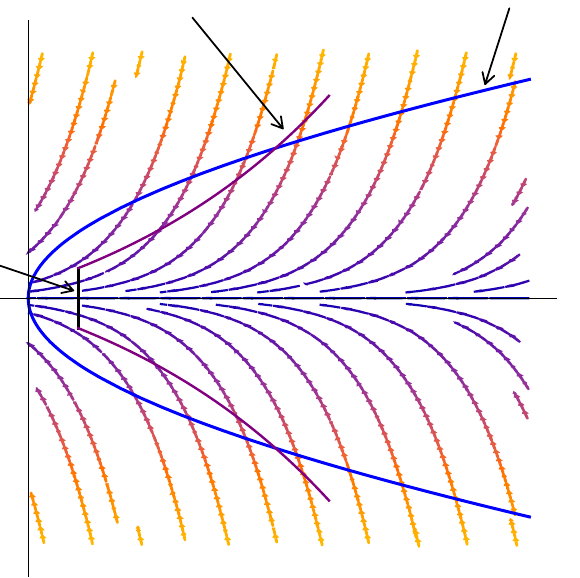}
		\put(0, 100){$v_2$}
		\put(100, 47){$v_1$}
		\put(-8, 53){$\DeltaIn$}
		\put(85, 100){$h(v_2)$}
		\put(29, 99){$\pm \xi(v_1)$}
	\end{overpic}
	\caption{The reduced flow of~\eqref{eq:galerkin_system_two}. The region inside the curve
		$h(v_2)$ (in blue) corresponds to the stable part of the critical manifold $\cal  C$; across that curve,
		one of the eigenvalues of the linearisation about $\calC$ changes sign. Also illustrated are
		$\DeltaIn$ (in black) and $\pm \xi(v_1)$ (in purple); recall~\eqref{eq:condition_v1_v2_2}. The set of
		initial conditions in the $(v_1,v_2)$-plane that reach $\DeltaIn$ is found in the intersection
		of the regions to the right of $h(v_2)$ and $\pm \xi(v_1)$.}%
	\label{fig:critical-manif-v-2}
\end{figure}

\section{Finite-time blowup of solutions}%
\label{subsec:finite-time-blowup}
To motivate the importance of restrictions on the initial data for the Galerkin truncation in~\eqref{eq:galerkin_system_original}, we prove that for some choices of initial conditions, a blowup in $u_1$ can occur before $v_1$ becomes negative already for $k_0=2$. Setting $a=\frac{1}{2}$ and rescaling $u_1$ and $u_2$ by a factor of $2^{-1/2}$, we
obtain the two-dimensional system
\begin{align}
	\begin{aligned}\label{Eq:2x2_Galerkin}
		u_1^\prime &= -v_1+u_1^2+u_2^2, & \quad\text{with } u_1(0)=u_1^0,\\
		v_1^\prime &= -\varepsilon, &\quad\text{with } v_1(0)=v_1^0,\\
		u_2^\prime &= -v_2+u_2(2u_1-\pi^2),\quad &\quad\text{with }u_2(0)=u_2^0,\\
		v_2^\prime &= -\varepsilon \pi^2 v_2,\quad &\quad\text{with }v_2(0)= v_2^0.
	\end{aligned}
\end{align}
It is assumed that $v_1^0>0$. We will show that, for $v_2^0\neq 0$ and $\varepsilon>0$ sufficiently
small, a finite-time blowup will occur in \eqref{Eq:2x2_Galerkin} before $v_1$ changes sign. For the
sake of simplicity and without loss of generality, we may assume that $u_2^0<0$ and $v_2^0>0$; see
also~\autoref{Rem:RestrictionInitialValues2}.
\subsection{Main observation}
Firstly, we establish our main observation on finite-time blowup for solutions of~\eqref{Eq:2x2_Galerkin} when $k_0=2$. Various auxiliary results which are used in the proof are collated in~\ref{app:blowup}.

\begin{prop}\label{Thm:Blowup_before_Blowup}
	Let $u_1^0\in\R$, $u_2^0<0$, and $v_1^0,v_2^0>0$. Then, there exists $\varepsilon>0$ such that the solution of \eqref{Eq:2x2_Galerkin} blows up before
	$t_0=\frac{v_1^0}{\varepsilon}$, i.e., before $v_1$ changes sign. 
\end{prop}
\begin{proof}
	As observed in Proposition~\ref{Prop:RelationInitialValues} and Remark~\ref{Rem:RestrictionInitialValues2} below, without loss of generality, we may assume that
    \[-\pi/2< u_1^0 \leq \pi/4 \quad \text{and} \quad
	v_1^0<\min\left\{\tfrac{\pi^2}{16},\left(\tfrac{e^{-\pi^4/32}v_2^0}{2(\pi+\pi^2)}\right)^2\right\}.
    \]
	By Propositions~\ref{Prop:u_1^0_below} and~\ref{Prop:u_1^0_above}, it follows that $-\pi/2< u_1(t) \leq \pi/4$ for all $t\geq0$ unless there is blowup in finite
	time independent of $\varepsilon>0$. We consider the time interval
	$[0,\frac{v_1^0}{2\varepsilon}]$ in which $v_1$ remains positive. Moreover, we have
	$v_2(t)\in\big[\exp\big(-\tfrac{\pi^2v_1^0}{2}\big)v_2^0,v_2^0\big]$ for all
	$[0,\frac{v_1^0}{2\varepsilon}]$. Since $u_2(t)\leq0$, by
	Remark~\ref{Rem:RestrictionInitialValues2}, and since $-\pi/2< u_1(t) \leq \pi/4$ for all
	$t\geq0$, we find 
	\[
		-2v_2^0-(\pi^2-\tfrac{\pi}{2})u_2 < -v_2+u_2(2u_1-\pi^2) = \partial_tu_2  < -\exp(-\frac{\pi^2v_1^0}{2})v_2^0 -(\pi+\pi^2)u_2
	\]
	in $[0,\frac{v_1^0}{2\varepsilon}]$. Let now $w_u$ and $w_o$ be the solutions of
	\begin{align*}
		 w^\prime_u &= -2v_2^0-(\pi^2-\tfrac{\pi}{2})w_u,\\
		 w^\prime_o &= -\exp\left(-\frac{\pi^2v_1^0}{2}\right)v_2^0 -(\pi+\pi^2)w_o
	\end{align*}
	in $[0,\frac{v_1^0}{2\varepsilon}]$, with $w_u(0)=u_2^0=w_o(0)$.   Lemma~\ref{Lemma:Monotonicity} ensures $w_u \leq u_2 \leq w_o$. 
	Thus, in $[\frac{v_1^0}{4\varepsilon},\frac{v_1^0}{2\varepsilon}]$, we have  
	\begin{align*}
		u_2(t)&\leq w_o(t)=\exp\left(-(\pi+\pi^2)t\right)\left[u_2^0+\frac{1}{(\pi+\pi^2)}\exp\left(-\frac{\pi^2v_1^0}{2}\right)v_2^0\right]
			\\
   & \qquad \qquad  -\frac{1}{(\pi+\pi^2)}\exp\left(-\frac{\pi^2v_1^0}{2}\right)v_2^0\\
			&\leq -\frac{1}{2(\pi+\pi^2)}\exp\left(-\frac{\pi^2v_1^0}{2}\right)v_2^0,
	\end{align*}
	provided $\varepsilon>0$ is sufficiently small. Correspondingly, in
	$\left[\frac{v_1^0}{4\varepsilon},\frac{v_1^0}{2\varepsilon}\right]$ we obtain 
	\begin{equation}%
	\label{Eq:Blowup_argument}
		\begin{aligned} 
			u_1^\prime &= -v_1+ u_1^2+u_2^2\geq
			-v_1^0+\frac{e^{-\pi^2v_1^0}}{4(\pi+\pi^2)^2}{\left(v_2^0 \right)}^2+u_1^2\\
   &\geq -  v_1^0
			+\frac{e^{-\pi^4/16}}{4(\pi+\pi^2)^2}{\left(v_2^0 \right)}^2+u_1^2 > c +u_1^2
		\end{aligned}
	\end{equation}
	for some $c > 0$ due to $v_1^0<\min\left\{\tfrac{\pi^2}{16},\left(\tfrac{e^{-\pi^4/32}v_2^0}{2(\pi+\pi^2)}\right)^2\right\}$. The equation
	\[
		w^\prime = \mu + w^2,
	\]
	with $\mu>0$ constant, experiences blowup for any initial condition at a time $t_0$ that depends on the initial
	condition and on $\mu$, but that is independent of $\varepsilon$. If $\varepsilon>0$ is small enough, then
	the blowup occurs in $\left[0,\frac{v_1^0}{4\varepsilon}\right]$. Thus,
	Lemma~\ref{Lemma:Monotonicity} implies that $u_1$ blows up before time
	$\frac{v_1^0}{2\varepsilon}$; in particular, it blows up before $v_1$ changes sign.
\end{proof}


\subsection{Proof of Proposition~\ref{Thm:Blowup_before_Blowup}}
\label{app:blowup}
The following comparison principle is standard; however, we include it for completeness.  

\begin{lemma}\label{Lemma:Monotonicity}
	Let $f,g\colon[0,\infty)\times \R\to \R$ be such that $f(t,x)> g(t,x)$ for all
	$(t,x)\in[0,\infty)\times \R$, and suppose that $f$ and $g$ are locally Lipschitz continuous.
	Furthermore, let $x_0\in\R$, and let $y_f$ and $y_g$ be the solutions of
	\[
		 y^\prime_f(t) = f(t,y_f(t))\quad\text{and}\quad y^\prime_g(t) = g(t,y_g(t)),
		 	\quad\text{with } y_f(0) = y_g(0).
	\]
	Then, $y_f(t) \geq y_g(t)$ for all $t$ in the intersection of the maximal
	existence intervals of $y_f$ and $y_g$.
\end{lemma}


\begin{prop}\label{Prop:RelationInitialValues}
	If the solution of \eqref{Eq:2x2_Galerkin} exists for a sufficiently long time, then there exists $\eta>0$, independent of $\varepsilon$, but dependent on $v_1^0$ and $v_2^0$, such
	that
	\[
		0<v_1\big(\tfrac{v_1^0-\eta}{\varepsilon}\big)<\min\left\{\tfrac{\pi^2}{16},\tfrac{e^{-\pi^4/16}}{4(\pi+\pi^2)^2}
			\left[v_2\big(\tfrac{v_1^0-\eta}{\varepsilon}\big)\right]^2\right\}.
	\]
\end{prop}
\begin{proof}
	Solving explicitly, we can write $v_1(t)=v_1^0-\varepsilon t$ and $v_2(t)=\exp(-\varepsilon\pi^2t)v_2^0$. Hence,
	\[
		v_1\big(\tfrac{v_1^0-\eta}{\varepsilon}\big)=\eta>0.
	\]
	On the other hand, for $\eta>0$ sufficiently small, we have
	\[
		v_1\big(\tfrac{v_1^0-\eta}{\varepsilon}\big)=\eta<
		\tfrac{e^{-\pi^4/16}}{4(\pi+\pi^2)^2}\exp\big[-2\pi^2(v_1^0-\eta)\big](v_2^0)^2
		=\tfrac{e^{-\pi^4/16}}{4(\pi+\pi^2)^2}\left[v_2\big(\tfrac{v_1^0-\eta}{\varepsilon}\big)\right]^2.
	\]
	Obviously, if $\eta>0$ is small enough, it also holds that
	\[
		v_1\big(\tfrac{v_1^0-\eta}{\varepsilon}\big)=\eta<\tfrac{\pi^2}{16},
	\]
	which shows the assertion.
\end{proof}

 Given Proposition~\ref{Prop:RelationInitialValues}, blowup in \eqref{Eq:2x2_Galerkin} can still
 occur in a time interval of length $\eta/\varepsilon$. Since $\eta$ can be chosen independent of
 $\varepsilon$, that interval can be made arbitrarily large for $\varepsilon$ sufficiently small. In
 particular, if we can now show that solutions of \eqref{Eq:2x2_Galerkin} blow up after a time which
 is independent of $\varepsilon$, then blowup will occur before $v_1$ changes sign if
 $\varepsilon>0$ is small enough. By Proposition~\ref{Prop:RelationInitialValues}, we may assume
 that
 $v_1^0<\min\Big\{\tfrac{\pi^2}{16},\Big[\tfrac{e^{-\pi^4/32}v_2^0}{2(\pi+\pi^2)}\Big]^2\Big\}$.

\begin{prop}\label{Prop:u_1^0_below}
	If the solution of \eqref{Eq:2x2_Galerkin} exists for a sufficiently long time, then there
	exists a time $t_0\geq0 $, independent of $\varepsilon$, such that $u_1(t_0)> -\pi/2$.
\end{prop}
\begin{proof}
	Since we can assume $v_1^0<\tfrac{\pi^2}{16}$, it holds that
	\[
		u^\prime_1= -v_1+u_1^2+u_2^2> -\tfrac{\pi^2}{16} +u_1^2.
	\]
	As long as $u_1^2\leq -\pi/2$, we also have $-\tfrac{\pi^2}{16} +u_1^2>\frac{3\pi^2}{16}$ and,
	hence, $u^\prime_1>\frac{3\pi^2}{16}$, which proves the assertion.
\end{proof}

\begin{prop}\label{Prop:u_1^0_above}
	If $u_1^0>\pi/4$, then solutions of \eqref{Eq:2x2_Galerkin} blow up after a finite time which is independent of $\varepsilon$.
\end{prop}
\begin{proof}
	Since we may assume $v_1^0<\tfrac{\pi^2}{16}$, it holds that 
	\[
		 u^\prime_1= -v_1+u_1^2+u_2^2>-\tfrac{\pi^2}{16} +u_1^2.
	\]
	If $u_1^0>\pi/4$, then the right-hand side in the above expression is positive. It follows from
	\autoref{Lemma:Monotonicity} that blowup occurs after a finite time which is independent of
	$\varepsilon$, as that is the case for the solution of 
	\[
		w^\prime = -\tfrac{\pi^2}{16} +w^2,\quad\text{with } w(0)=u_1^0>\pi/4.
	\]
\end{proof}

\begin{remark}
	\label{Rem:RestrictionInitialValues2}%
	\begin{enumerate}
	\item Propositions~\ref{Prop:u_1^0_below} and~\ref{Prop:u_1^0_above} imply that we may assume $-\pi/2< u_1^0 \leq \pi/4$.
	\item One can also show the following: if solutions to \eqref{Eq:2x2_Galerkin} exist for a long enough time
		and if $\varepsilon>0$ is sufficiently small, then there exists $t_0\geq0$,
		independent of $\varepsilon$, such that $u_2(t)\leq 0$ for all $t\geq t_0$. 
		We simply take $u_2^0<0$ and observe that, hence, $u_2(t)\leq0 $ for all $t\geq0$.
		Note, however, that one has to exchange signs here if $v_2^0<0$.
	\end{enumerate}
\end{remark}
%
We
now  derive an estimate for how small $\varepsilon$ has to be such that we observe blowup
before $v_1$ changes sign. In a first step, we give an explicit expression for $\eta$ -- dependent on $v_1^0$ and $v_2^0$, but independent of  $\varepsilon$ -- that satisfies the estimate in
Proposition~\ref{Prop:RelationInitialValues}. For the sake of simplicity, we will assume that
$v_1^0\in(0,\pi^2/16)$.

\begin{lemma}\label{lemma:eta}
If $\eta$ is chosen as 
\[
\eta= \frac{(v_2^0)^2}{4(\pi+\pi^2)^2}\exp\Big(-\frac{\pi^4}{16}-2\pi^2v_1^0\Big),
\]
then the estimate in Proposition~\ref{Prop:RelationInitialValues} is satisfied.
\end{lemma}
\begin{proof}
	The estimate in Proposition~\ref{Prop:RelationInitialValues} holds true if and only if
	\begin{align*}
		\eta &=v_1^0-\varepsilon \frac{v_1^0-\eta}{\varepsilon}=
		v_1\Big(\frac{v_1^0-\eta}{\varepsilon}\Big)<\tfrac{e^{-\pi^4/16}}{4(\pi+\pi^2)^2}\Big[v_2\big(\tfrac{v_1^0-\eta}{\varepsilon}\big)\Big]^2\\
		&=\tfrac{e^{-\pi^4/16}}{4(\pi+\pi^2)^2} (v_2^0)^2\exp\Big(-2\varepsilon
		\pi^2\frac{v_1^0-\eta}{\varepsilon} \Big)=\tfrac{e^{-\pi^4/16}}{4(\pi+\pi^2)^2}
		(v_2^0)^2\exp\big[-2 \pi^2(v_1^0-\eta) \big].
	\end{align*}
	Multiplication by $e^{-2\pi^2\eta}$ yields
	\[
		\eta\exp(-2\pi^2\eta) < \tfrac{e^{-\pi^4/16}}{4(\pi+\pi^2)^2} (v_2^0)^2\exp(-2 \pi^2v_1^0),
	\]
	which is satisfied if
	\[
		\eta= \frac{(v_2^0)^2}{4(\pi+\pi^2)^2}\exp\big(-\tfrac{\pi^4}{16}-2\pi^2v_1^0\big),
	\]
	as stated in the assertion.
\end{proof}

\begin{remark}\label{Rem:Blowup_Time}
The main argument in the proof of Proposition~\ref{Thm:Blowup_before_Blowup} was that solutions of $w' = \mu + w^2$ blow up in finite time if $\mu>0$. The explicit solution is given by
\[
 w(t)=\sqrt{\mu}\tan\Big(\arctan(\tfrac{w(0)}{\sqrt{\mu}})+\sqrt{\mu}t\Big),
\]
and hence exists until time 
\[
 t=\frac{\pi/2-\arctan(\tfrac{w(0)}{\sqrt{\mu}})}{\sqrt{\mu}}.
\]
In particular, blowup occurs before time $t=\pi/\sqrt{\mu}$. To determine how to choose $\mu$ in
Proposition~\ref{Thm:Blowup_before_Blowup}, we recall Equation~\eqref{Eq:Blowup_argument}, which
allows for
\[
	\mu=-v_1^0+\frac{e^{-\pi^4/16}}{4(\pi+\pi^2)^2}(v_2^0)^2=(e^{2\pi^2\eta}-1)\eta;
\]
here, we have used Lemma~\ref{lemma:eta}.
\end{remark}

\begin{prop}\label{Prop:Choice_epsilon}
	If $\varepsilon<\frac{\eta^2}{2\sqrt{2}}$, then the solution of \eqref{Eq:2x2_Galerkin} blows up
	before $v_1$ changes sign.
\end{prop}
\begin{proof}
 	In the proof of Proposition~\ref{Thm:Blowup_before_Blowup}, blowup is generated in the time
	interval $[0,\frac{v_1^0}{4\varepsilon}]=[0,\frac{\eta}{4\varepsilon}]$. In combination with
	Remark~\ref{Rem:Blowup_Time}, it follows that it suffices to take $\varepsilon$ small enough
	such that $\varepsilon<\frac{\eta\sqrt{\mu}}{4\pi}$. To prove the assertion, we rewrite the
	right-hand side of that inequality as
	 \begin{align*}
		 \frac{\eta\sqrt{\mu}}{4\pi}=\frac{\eta\sqrt{\eta(e^{2\pi^2\eta}-1)}}{4\pi},
	 \end{align*}
	which is, in fact, sharper than the right-hand side in the assertion; for conciseness, we observe
	that $e^{2\pi^2\eta}-1>2\pi^2\eta$ and, hence, that
	\[
		\frac{\eta\sqrt{\mu}}{4\pi}>\frac{\eta^2}{2\sqrt{2}},
	\]
	whence the assertion follows.
\end{proof}


\printbibliography

\end{document}